\documentclass[a4paper]{article}
\usepackage[english]{babel}
\usepackage{a4wide}
\usepackage[utf8]{inputenc}
\usepackage[T1]{fontenc}
\usepackage{braket}
\usepackage{color}
\usepackage[normalem]{ulem}
\usepackage{graphicx}
\graphicspath{ {./images/} }

\usepackage{verbatim}

\usepackage{hyperref}
\usepackage{enumerate}
\usepackage{enumitem}

\usepackage{amsmath, amssymb, amstext, amsthm, dsfont}

\usepackage[arrow, matrix, curve]{xy}
\usepackage{slashed}

\usepackage{makeidx}

\makeindex
\newcommand{\N}{\mathbb{N}} 
\newcommand{\Z}{\mathbb{Z}}

\newcommand{\R}{\mathbb{R}}

\DeclareMathOperator{\B}{B}

\DeclareMathOperator{\e}{e}

\DeclareMathOperator{\supp}{supp}

\DeclareMathOperator{\diff}{d}
\DeclareMathOperator{\BUC}{BUC}

\theoremstyle{plain}
\newtheorem{theorem}{Theorem}

\newtheorem{lemma}{Lemma}[section]
\newtheorem{corollary}[lemma]{Corollary}

\newtheorem{proposition}[lemma]{Proposition}
\newtheorem{remark}[lemma]{Remark}

\newtheorem{definition}[lemma]{Definition}
\theoremstyle{definition}
\newtheorem{example}[lemma]{Example}

\title{A spatial host-parasite model with host immunity:\\
Survival and linear spread of parasites on $\Z$}
\author{Sascha Franck and Cornelia Pokalyuk}
\date{\today}

\begin{document}

\maketitle

\begin{abstract}
     We introduce a generalized version of the frog model to describe the invasion of a parasite population in a spatially structured immobile host population with host immunity on the integer line. Parasites move according to simple symmetric random walks and try to infect any host they meet. Hosts, however, own an immunity against the parasites that protects them from infection for a random number of attacks. Once a host gets infected, it and the infecting parasite die, and a random number of offspring parasites is generated. We show that the positivity of the survival probability of parasites only depends on the mean offspring and mean height of immunity. Furthermore, we prove through the construction of a renewal structure that given survival of the parasite population  parasites invade the host population at linear speed under relatively mild assumptions on the host immunity distribution.
\end{abstract}

\pagestyle{empty}

\pagestyle{headings}
\section{Introduction}
We are interested in the spread of a parasite population in a spatially distributed host population with host immunity. We model this scenario by placing susceptible (immobile) hosts on the vertices of a graph $G$. (In the following we will analyze the model mainly on $\Z$, that is, for $G=(V,E)$ with vertex set $V=\Z$ and edge set $E=\{\{x,y\}:\lVert x-y\rVert_1 = 1\}$). Parasites infect hosts and move on the graph according to symmetric nearest neighbor random walks in continuous time. As hosts often have an immune response against infections, we assume that whenever a parasite tries to infect a host, parasite reproduction might be prevented (with some probability that might depend on the number of previous infection attemps), and the parasite gets killed. To be precise, we assume that there is some sequence $(p_m)_{m\in\N}\subset[0,1]$ such that for each host, if the host is attacked for the $m$-th time by a parasite then this host prevents the infection, independent of everything else, with probability $1-p_m$ and kills the parasite. In the case of a successful infection, which happens with probability $p_m$, the parasite kills the host and sets free a (possibly random) number of offspring. For simplicity we assume that hosts do not reproduce. \\
We can view this model as an extension of the frog model that has been introduced by Telcs and Wormland (1999, \cite{T1999}). The classical frog model is an interacting particle system on some graph $G=(V,E)$, which evolves under the following dynamics. There are two types of particles, which are usually called \textit{active frogs} and \textit{sleeping frogs}, and each site $v\in V$ can host a finite number of frogs of the same type. Active frogs, independently of everything else, perform nearest neighbor random walks on $G$, and sleeping frogs do not move at all. Whenever an active frog jumps onto a vertex $v\in V$, all sleeping frogs on $v$ immediately transform into active frogs and start performing their own simple nearest neighbor random walk, starting at $v$. The model can be formulated both in continuous and discrete time; however, we focus on the continuous case here. Interpreting active frogs as parasites and sleeping frogs as the offspring produced at an infection, in the classical frog model, an infection is always successful, and infecting parasites survive the infection. Therefore, in particular, the population of active frogs can only grow larger.\\
In \cite{A2002}, Alves, Machado, and Popov showed a shape theorem for the set $S_d(n)$ of vertices visited by some active frog up to time $n$ when $G=(\Z^d,\{\{x,y\}:\lVert x-y\rVert_1 = 1\})$ and at $n=0$ there is one active frog in $0$ and one sleeping frog on each other site $x\in\Z^d\setminus\{0\}$. Precisely, they showed that there is a convex (deterministic) set $\mathcal{A}_d$ such that for any $\varepsilon\in(0,1)$
\[
(1-\varepsilon)\mathcal{A}_d\subset \frac{S_d(n)}{n}\subset(1+\varepsilon)\mathcal{A}_d
\]
for all $n$ large enough almost surely. In \cite{A2001} they improved this result to hold for random initial configurations, placing an i.i.d.~number of sleeping frogs on each vertex.
For the frog model in continuous time, Ramírez and Sidoravicius showed in \cite{RS2004} the shape theorem on $\Z^d$ with one sleeping frog on each site. \\
In the one-dimensional, time-continuous case, Quastel, Comets, and Ramírez improved upon this result in \cite{C2009} and obtained a central limit theorem for the fluctuations of $S_1(t)$ around $\mathcal{A}_1t$, and a large deviation principle for this quantity was shown by Bérard and Ramírez in \cite{B2010}.

A similar adaptation of the frog model to the model that we investigate in this paper was introduced by Junge, McDonald, Pulla and Reeves in \cite{J2023}. In their work (in the language of our model), hosts might also prevent parasite reproduction for a random number of times but at a successful infection they do not kill the attacking parasite and, in addition, parasites can sit on top of hosts without infecting and killing them. For this reason, the parasites remain in the system and one can construct a coupling of their model with a classical frog model, with a different offspring distribution, that shows the shape theorem on $\Z^d$. They investigate their model on $d$-ary trees, where results for arbitrary offspring distributions are not available, using quite different methods to ours.

Another version of the frog model that incorporates the potential death of frogs was introduced by Alves, Machado and Popov in \cite{A2002b}. In this model, there is some fixed probability $p\in(0,1]$ such that each time a frog would take a step, it dies with probability $p$ before performing this step. They investigate phase transitions in the parameter $p$ for different statistics of the process. 
In \cite{F2025} we investigate a version of our model (with a specific choice for the sequence $(p_m)_{m\in\N}$ of success probabilities) that is quite similar to the model in \cite{A2002b}. In \cite{F2025}, we set $p_m = p \delta_1(m)$ for some $p\in(0,1]$, i.e., a host gets infected the first time it is attacked with probability $p$, or else it is completely immune and never gets infected while killing any parasite that tries to do so. Similar as in \cite{A2002b} we investigate phase transitions in the parameter $p$. 

 In the classical frog model the population of frogs a.s.~survives. Obviously, this distinguishes our general host parasite model from the frog model. Another main difference is that in the frog model, using a collection of independent simple symmetric random walks to assign each frog its trajectory after waking up yields an easy way to couple initial configurations in a monotone way. However,  in our model this procedure will not provide a monotone coupling because when and where a specific parasite dies depends on the location of hosts that are still in the system; see Example \ref{unmonontonecoupl}. Hence, the trajectories of parasites will no longer be independent in our model. In particular, this means that the techniques relying on this monotone coupling, such as the sub-additivity arguments used in Ramirez and Sidoravicius (\cite{RS2004}) to show a shape theorem, cannot be applied to our model.\\
Importantly, techniques used by Quastel, Comets, and Ramírez in \cite{C2007},\cite{C2009} do not rely on this monotone coupling of the frog model but on the construction of a renewal structure. Inspired by these methods, we investigate our model by constructing a sequence of sites, from which the environment of future immunities and offspring numbers guarantees linear growth. \\ \vspace{0.1cm}\\
The paper is structured as follows. In Section \ref{sec:mainres} we will introduce the model, present the main results of this work and sketch their proofs. Section \ref{sec:goodsites} will contain further constructions and auxiliary results that are needed in the proof of the main theorems of this work. Then, in Section \ref{Sec:ProofsOfTheorems2to4}, we will give the proofs of the main results using the techniques developed in Section \ref{sec:goodsites}. Finally, Section \ref{Construction} gives a formal construction of the process as a strong Markov process on some state space.
\section{(Informal) model definition and main results}\label{sec:mainres}
\subsection{The spatial infection model with host immunity}
Consider the graph $G=(\Z,\{\{x,y\}:\lVert x-y\rVert_1 = 1\})$ and consider two sites $\ell\leq r$ with $\ell \in \Z \cup \{-\infty\}$ and $r\in \Z$. Place one host at each vertex $x$ with $x \notin [\ell, r]$. Equip each host with an immunity and denote the immunity corresponding to the host occupying the vertex $x$ (if there is one) by $I_x$. We assume that $(I_x)_{x \notin[\ell, r]}$ is an i.i.d.~sequence of $\N$-valued random variables, distributed as some $I$, and that $I$ has a finite expectation, which we denote by $\mu_I$. 
Furthermore, let $\eta \in \mathbb{N}_0^{\Z}$ with $\supp \eta \subset [\ell, r]$. Place on site $x \in [\ell, r]$ initially $\eta(x)$ parasites.
Each parasite performs independently of all other parasites a time-continuous symmetric random walk with rate 2.
Whenever a parasite jumps onto a vertex that is occupied by a host, the parasite attacks the host.
If the initial immunity $I_x$ of the host at site $x$ is equal to $i$ and the host is attacked for the $i$th time, the host and the parasite die, and $A_x$ parasites are released at $x$, where $(A_x)_{x \in \Z}$ is an i.i.d.~sequence of $N_0$-valued random variables, distributed as some $A$, independent of $I,(I_x)_{x\notin[\ell, r]}$ and $A$ has a finite expectation, which we denote by $\mu_A$. Otherwise, if the host has been attacked less than $i$ times, only the parasite dies and the host survives. If a parasite jumps to an empty site, the parasite survives. 
In Section \ref{Construction} we give a rigorous definition of the process on a state space that contains the position of the sites occupied by parasites or hosts as well as the remaining immunities of the hosts at sites next to sites free of hosts and show that it is a well-defined strong Markov process if the initial distribution of parasites is sufficiently well behaved. 
We call the corresponding process the spatial infection model with host immunity (SIMI, for short) with initial condition $\zeta=(\ell, r , \eta)$. If $\ell= -\infty$, only to the right of $r$ hosts are placed. We call this case a one-sided host population and correspondingly the case $\ell \in \Z$ a two-sided host population.\\
 Given the process is well defined, we can determine at each time point $t\geq 0$ the largest subinterval of $\Z$ which is not occupied by any host. We denote the boundaries of this interval by
\[
r_t(\zeta) := \sup\left\{x\ge r:\begin{matrix}\text{In the SIMI with initial configuration }\zeta\\\text{at time }t,\text{ there is no living host at site } x\end{matrix}\right\}
\]
and call it the \textit{right front} of the process and call (in the case $\ell \in \Z$)
\[
\ell_t(\zeta) := \inf\left\{x\le \ell:\begin{matrix}\text{In the SIMI with initial configuration }\zeta\\\text{at time }t,\text{ there is no living host at site } x\end{matrix}\right\}
\]
the \textit{left front} of the process. Similarly, we denote by $\eta_t(\zeta)\in\N_0^\Z$ the amounts of parasites on each site at time $t$, by $\iota^\ell_t(\zeta)$ the remaining immunity of the host at site $\ell_t(\zeta)-1$ at time $t$, and by $\iota_t^r(\zeta)$ the remaining immunity of the host at site $r_t(\zeta)+1$ at time $t$. Then $(\ell_t,r_t,\eta_t,\iota_t^\ell,\iota_t^r)$ is enough information to describe the SIMI as a Markov process, and we obtain the following theorem.
 \theoremstyle{plain}\begin{theorem}\label{Theorem:well-def}
 There is a probability space $\mathbf{\Omega}^\prime = (\Omega^\prime,\mathcal{F}^\prime,\mathbf{P}^\prime)$ such that for any (possibly random) initial placement of parasites and hosts given by $\zeta = (\ell,r,\eta)$ defined on some probability space $\mathbf{\Omega}^{\prime\prime}= (\Omega^{\prime\prime},\mathcal{F}^{\prime\prime},\mathbf{P}^{\prime\prime})$ and fulfilling the condition
 \[
 \mathbf{E}^{\prime\prime}\left[f_\theta(\zeta)\right] := \mathbf{E}^{\prime\prime}\left[\sum_{x\in\Z}\eta(x)\e^{\theta (x-r)}\right]<\infty
 \]
 for some $\theta > 0$, the SIMI with initial configuration $\zeta$ can be constructed as a strong Markov process on the state space $(\mathbb{S}_\theta, d_\theta)$ with càdlàg paths on the product probability space $\mathbf{\Omega} = (\Omega,\mathcal{F},\mathbf{P})$ of $\mathbf{\Omega}^\prime$ and $\mathbf{\Omega}^{\prime\prime}$, where the state space $ \mathbb{S}_\theta $ is given by
 \[
 \left\{\zeta=(\ell,r,\eta,\iota^\ell,\iota^r)\in(\Z\cup\{-\infty\})\times\Z\times\N_0^\Z\times\N\times\N:\begin{matrix}\ell\le r, \supp \eta\subset[\ell,r]\\f_\theta(\zeta) <\infty\end{matrix}\right\}
 \]
and
 \[
 \begin{split}
 d_\theta(\zeta,\zeta^\prime) &:= \vert r-r^\prime\vert + \vert \ell-\ell^\prime\vert + \vert \iota^\ell-(\iota^\ell)^\prime\vert + \vert \iota^r-(\iota^r)^\prime\vert \\
 &+\sum_{x\in\Z}\vert\eta(x)-\eta^\prime(x)\vert\e^{-\theta(x-(r\wedge r^\prime))}.
 \end{split}
 \]
 If in addition $\mathbb{E}[A^2] < \infty$, then this Markov process possesses the Feller property.
 \end{theorem}

 \theoremstyle{remark}\begin{remark}
 In Section \ref{Construction} we actually not only construct the process on $\mathbb{S}_\theta$ but also on a larger state space, which tracks additional information of the individual parasite trajectories. This will be helpful in the proof of our main results. In Section \ref{Construction:unt} we will construct the process carrying only the information above.
 \end{remark}
\begin{remark}
Via the distribution of $I$ it is possible to model different responses of the hosts to a parasite. Taking any sequence $(p_m)_{m\in\N}\subset[0,1]$ and an independent sequence $(B_m)_{m\in\N}$ with $B_m \sim \text{Ber}(p_m)$, we can set
\[
I := \inf\{m\in\N: B_m = 1\}.
\]
This means that the $m$-th infection attempt is successful with probability $p_m$. In the case that $p_m \equiv p$, this corresponds to no response of the host to previous infection attempts and a geometrically distributed $I$. However, taking $p_m \downarrow 0$ corresponds to hosts building up an immunity against infections each time that they successfully defend against an infection attempt. Taking on the other hand, $p_m \uparrow 1$ corresponds to parasites continuously damaging the host's defenses until they eventually succeed in infecting it. Controlling the speed at which $p_m$ goes to $0$ (resp. to $1$) corresponds to how fast this immunity builds up (resp. how quickly the host gets damaged). For example, the choice $p_m = \frac{\alpha}{m+\alpha}$ with some $\alpha > 1$ yields the behavior $\mathbf{P}(I > m) \sim Cm^{-\alpha}$ for some constant $C = C(\alpha)$.
\end{remark}
\subsection{Main results on survival and speed of spread}
Next we come to our main results concerning the probability of survival of the parasite population as well as the speed at which the parasite population spreads in the host population when we condition on survival.

For the rest of the paper we fix some initial configuration $\zeta$ defined on $\mathbf{\Omega}^{\prime\prime}$ as in Theorem \ref{Theorem:well-def}. We note that by construction the initial configuration $\zeta$ is independent of the immunities, the offspring numbers, as well as the trajectories of parasites that are used in the construction of the process, which are defined on $\mathbf{\Omega}^\prime$. The SIMI with this random initial configuration is constructed by first sampling $\zeta$ and then constructing the process conditionally on the realization. The first results concern the probability of survival of the parasite population for infinite time, i.e., the following event:
\[
\mathcal{S}(\zeta) := \bigcap_{t\ge 0}\left\{\begin{matrix}\text{In the SIMI with initial configuration }\zeta,\\\text{ there is at least one living parasite at time } t\end{matrix}\right\}.
\]
\subsubsection{One-sided host population}
In this subsection we consider the case $\ell=-\infty$ and assume w.l.o.g. $r=0$, that is, initially all sites to the left of and including the origin are either empty or occupied by parasites, and hosts are placed on all sites to the right of the origin. Since the initial condition is in this case always of the form $\zeta=(\ell, r, \eta)= (-\infty, 0, \eta)$, we shortly write $\eta$ for the initial condition instead of $\zeta.$

 Since random walks are recurrent on $\Z$, as long as the parasite population is alive, at least one parasite will hit the front in a finite amount of time. Thus, the event of survival only depends on $(A_x)_{x> 0}$ and $(I_x)_{x>0}$ and is given by
\[
\mathcal{S}(\zeta) = \mathcal{S}(\eta)=\bigcap_{n=1}^\infty\left\{\sum_{x\le 0}\eta(x) + \sum_{k=1}^{n-1} A_{k} \ge \sum_{j=1}^{n}I_{j}\right\}.
\]
This is just the event that a random walk with step distribution $A-I$ and starting at $\sum_{x\le 0}\eta(x)-I_{1}$ stays non-negative. Hence, we can infer the classical result from \cite[Lemma 6.1.3]{S1974} to obtain the following
\theoremstyle{plain}\begin{lemma}\label{lem:survivalpos}
Assume $\mu_A := \mathbf{E}[A] > \mu_I := \mathbf{E}[I]$, and the initial configuration $\eta$ satisfies
\[
\mathbf{P}\left(\sum_{x\le 0}\eta(x) \ge I\right) =:\delta_0 > 0
\]then
\[ \mathbf{P}(\mathcal{S}(\eta)) \ge \delta_0\exp\left(-\sum_{n=1}^\infty\frac{1}{n}\mathbf{P}\left(\sum_{k=1}^{n} A_{k} < \sum_{j=2}^{n+1}I_{j}\right)\right) >0.\]
If $\mu_A \le \mu_I$, then
\[
\mathcal{S}(\eta) = \left\{\sum_{x\le r}\eta(x) = \infty\right\}
\]
\end{lemma}
\begin{remark}
If $\eta$ is such that
 \[
 \sum_{x\le 0}\eta(x) \overset{d}{=} A,
 \]
then
 \[
 \mathbf{P}(\mathcal{S}(\eta)) = \exp\left(-\sum_{n=1}^\infty\frac{1}{n}\mathbf{P}\left(\sum_{k=1}^{n} A_{k-1} < \sum_{j=1}^{n}I_{j}\right)\right).
 \]
 Also, trivially, the inclusion
 \[
 \left\{ \sum_{x\le 0}\eta(x) = \infty\right\}\subset \mathcal{S}(\eta)
 \]
holds.
\end{remark}
In the following, we sometimes drop the reference to the initial configuration and just write $\mathcal{S}$ instead of $\mathcal{S}(\eta)$ and assume we fixed some initial configuration $\eta$.

If the random variables $(I_k)_{k\in\Z}$ are geometrically distributed for some parameter $p\in(0,1)$, the survival probability is given by the survival probability of a Galton-Watson process.
\theoremstyle{plain}\begin{lemma}
Suppose that $I\sim$ Geom$(p)$ for some $p\in\left(\frac{1}{\mu_A},1\right]$ and $\eta(x) = A_0 \delta_{x,0}$, where $A_0$ is distributed as $A$ conditioned to be $> 0$. Then
\[
\mathbf{P}(\Omega\setminus\mathcal{S}) = \sum_{m=1}^\infty\sum_{n=1}^\infty\mathbf{P}(A_{0}=m\vert A_{0}>0)\frac{m}{n}\mathbf{P}\left(m+\sum_{k=1}^n \tilde{A}_k B_k = n\right) < 1
\]
where $(B_k)_{k\ge1}$ is an i.i.d.~sequence of Bernoulli variables with parameter $p$ that is independent of $(\tilde{A}_k)_{k\ge1}$, and $(\tilde{A}_k)_{k\ge1}$ are independent and distributed as $A_1$. If, for some $a\in\N$, $A=a$ almost surely, we get the explicit formula
\[
\mathbf{P}(\Omega\setminus\mathcal{S}) = \sum_{n=1}^\infty\frac{1}{n}\binom{an}{n-1}p^{n-1}(1-p)^{n(a-1)+1}
\]
which is the unique solution $y\in(0,1)$ to the equation
\begin{align}\label{extinctionprobab}
y=(py+1-p)^a.
\end{align}
\end{lemma}
\begin{proof}
Since the random variables $(I_k)_k$ are geometrically distributed, the number of parasites that are generated in $n$ infection events (or attempts) is distributed as $\sum_{i=1}^{n} B_i \tilde{A}_i$.
At the onset of the infection process, $A_0$ parasites are generated. Since at each infection event, the infecting parasite dies, the parasite population gets extinct at the infection event, at which for the first time the number of parasites equals the total number of happened infection events, i.e., at the $\tau$-th infection event, with
\[\tau= \inf\{n | A_0 + \sum_{i=1}^{n} B_i \tilde{A}_i= n \}.\]
This number is the first generation at which the position $S_n = X_0 + \dots+ X_n$ of a random walk $(X_i)_{i\geq 0}$ with step sizes $X_{i}- X_{i-1} = B_i \tilde{A}_i -1$ and started in $X_0=A_0$ particles is $0$. By the hitting time theorem \cite{vdHandKeane}, \[\mathbf{P}(\tau=n| A_0=m) = \frac{m}{n} \mathbf{P}(S_n=0 |S_0=m).\]
Summing over all possible values for $A_0$ yields the first statement. \\
For the second statement, i.e., if $A=a$ almost surely, we calculate
\begin{align*}
\mathbf{P}(\Omega\setminus\mathcal{S}) & = \sum_{m=1}^\infty\sum_{n=1}^\infty\mathbf{P}(A_{0}=m\vert A_{0}>0)\frac{m}{n}\mathbf{P}\left(m+\sum_{k=1}^n \tilde{A}_k B_k = n\right) \\
& =\sum_{\ell=1}^\infty \frac{a }{ a \ell} \mathbf{P}\left(\sum_{k=1}^{a\ell} B_k= \frac{a \ell -a}{a} \right) \\
& =\sum_{\ell=1}^\infty \frac{1 }{\ell} 
\binom{a \ell}{\ell-1 } p^{\ell-1} (1-p)^{\ell(a-1) +1}.
\end{align*}
Renaming $\ell$ into $n$ yields the statement.\\
For the last statement, i.e., that the extinction probability solves equation \eqref{extinctionprobab}, note that the extinction probability of the parasite population is equal to the extinction probability of a branching process with an offspring distribution with probability weights $p_0= 1-p$, $p_a= p$, and $p_i=0$ for all $i \neq a$
 started with $a$ individuals. The generating function of this offspring distribution is $f(z) = (1-p + p z^a)$. Hence, 
 a branching process with this offspring distribution started with a single individual dies with probability $z$, where $z \in (0,1)$ solves 
 \[z=(1-p + pz^a). \]
If the branching process is started with $a$ 
individuals the extinction probability $y$ solves 
 \[y = z^a= (1-p + p z^a)^a = (1-p + p y)^a,\]
 which yields the claim.
\end{proof}

\begin{remark}
 For geometrically distributed immunities, each infection attempt is successful with probability $p$, independent of how many attempts were already made at this host. Thus, the extinction probability is bounded by the expression above for any underlying infinite graph $(V,E)$ on which the SIMI evolves and any initial host configuration. For recurrent graphs and infinitely large host populations, the extinction probability coincides.

 For immunity distributions that are not memoryless, each infection attempt's success depends on the number of previous infection attempts; thus the extinction probability cannot be bounded independent of the trajectories of parasites and hence depends on the underlying graph structure.
\end{remark}
\theoremstyle{definition}\begin{definition}
 For some fixed random initial configuration $\eta$ such that $\mathbf{P}(\mathcal{S}(\eta)) > 0$ we denote by
 \[
 \mathbb{P}^\eta := \mathbf{P}(\cdot\vert\mathcal{S}(\eta))
 \]
 the measure, conditioned on the survival of the parasites for infinite time. Similar to for $\mathcal{S}$, we often drop the reference to the initial configuration and just write $\mathbb{P}$ instead of $\mathbb{P}^\eta$. Also, we denote by $\mathbb{E}^\eta$ (resp. just $\mathbb{E}$) the corresponding expectation operator.
\end{definition}
Given survival, we are interested in the speed of growth. Under mild conditions on the distribution of $I$, we obtain that $(r_t)$ exhibits ballistic growth.

\theoremstyle{plain}\begin{theorem}\label{Theorem: ballistic growth} 
Assume that the distributions $A,I$ fulfill the following conditions:
\begin{equation}\label{alphacondition}
\mathbf{E}[A] = \mu_A > \mu_I =\mathbf{E}[I]
\end{equation}
and there is an 
\begin{equation}\label{momentcondition}
\alpha>1~~~\text{ such that }~~\mathbf{E}[I^{2\alpha}] <\infty.
\end{equation}
Let $\eta$ be any initial configuration such that the SIMI is well-defined in the sense of Theorem \ref{Theorem:well-def} and satisfies $\mathbf{P}(\mathcal{S}(\eta)) > 0$. Then there are constants $C_1,C_2$ such that conditioned on the survival of the parasite population, i.e., under $\mathbb{P}^\eta$, almost surely the following holds:
 \[
0 <C_1 \le \liminf_{t\to\infty}\frac{r_t}{t}\le\limsup_{t\to\infty}\frac{r_t}{t} \le C_2 < \infty.
\]
\end{theorem}
In Figure \ref{fig:r_t} below we depict a simulation of the process $(r_t)_{t\ge0}$ for $\alpha$ close to the boundary of the parameter range in Theorem \ref{Theorem: ballistic growth}, namely $\alpha > 1.05$.
\begin{figure}[ht] \label{fig:r_t}
\centering
 \includegraphics[width =\linewidth]{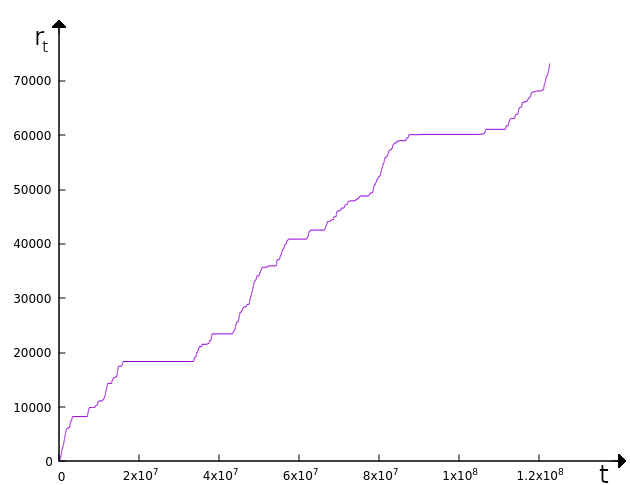}
 \caption{$(r_t)_{t\ge0}$ for $\mathbf{P}(I\ge n) = \left(\frac{n-1}{3}+1\right)^{-2.1}$ and $A=4$ a.s.}
\end{figure}
\pagebreak
We give a formal proof of Theorem \ref{Theorem: ballistic growth} in Section \ref{Sec:ProofsOfTheorems2to4}. The main steps of the proof are as follows.

\begin{proof}[Sketch of the proof of Theorem \ref{Theorem: ballistic growth}] 
The upper bound follows by simply coupling with a branching random walk.

For the lower bound, a key observation is that only parasites quite close to the front push the front further. The reasons for this are that a) parasites can generate offspring only at the front, b) parasites move diffusively, and c) the tail of the immunity distribution falls so fast that sufficiently high barriers of immunities appear so rarely.

This structure motivates the construction of auxiliary jump times that upper bound the actual jump times from above but are based only on parasites born relatively close to the front. 
More precisely, to estimate the jump time from site $n-1$ to site $n$, we consider a small neighborhood left to $n$ of random size $K_n$. The neighborhood is extended to the left until the number of parasites accumulating in the neighborhood is lower bounded by a number that grows linearly with the distance to the front; see Definition \ref{auxiliaryfront}.
Since by assumption the accumulated number of parasites grows on average faster than the accumulated number of immunities, one can show that the tail of the size $K_n$ of the neighborhood falls quickly as long as the neighborhood does not cross the origin; see Lemma \ref{lem:Kmtail}. \phantom\qedhere
\end{proof}

To control the movement of the front at sites whose corresponding neighborhoods would cross the origin, we proceed as follows. We show that there exists a site $M$ from which on, the neighborhood of each site $M+n$ does not cross $M$ for $n\geq k_0$ (with $k_0\in\N$ depending on the tail of the immunity distribution as well as (inversely proportional) on the difference between the average number of parasite offspring numbers and average immunities, see \eqref{qkcond}). This site $M$ has a.s.~a finite distance to the origin, and hence $M+ k_0$ is reached a.s.~in finite time; see Definition \ref{Msite} and Lemma \ref{lem:welldef}.

For sites $n \geq k_0$ we show by using results on collections of random walks hitting a one-sided barrier that the jump times from $M+n$ to $M+n+1$ have $q$-th moments for any $q \in [1, \frac{(4\wedge\alpha)+1}{2})$, see Lemma \ref{NuMoments}. Since neighborhoods of sites sufficiently far apart do not overlap with high probability, the corresponding auxiliary jump times only weakly depend on each other. More precisely, we show that the sequence of jump time $ (\nu_{M+n})_{n\geq k_0}$ forms a $\phi$-mixing sequence, see Lemma \ref{conditinallyphimixing}, which allows us to conclude a law of large numbers for $ (\nu_{M+n})_{n\geq k_0}$ and show the claimed lower bound. 

Under some stronger assumption, which we do not believe to be optimal, we can show the stronger result that the front satisfies an almost sure law of large numbers. 
\theoremstyle{plain}\begin{theorem}\label{Theorem: lln}
Suppose that for some $\alpha > \frac{3+\sqrt{11}}{2}$ and some $\varepsilon_A > 0$ we have
\begin{itemize}
 \item $\mathbf{E}[\vert I\vert^{2\alpha}] < \infty$, $\mathbb{E}[A^{\frac{4}{(4\wedge\alpha)-3}+\varepsilon_A}]$ and $\mathbf{P}(A-I \ge 4) > 0$.
\end{itemize}
Then, there is a (deterministic) $\gamma \in (0,\infty)$ such that, starting in the initial configuration given by $\eta = A\delta_{0}$ and conditioned on the survival of parasites, almost surely
 \[
\lim_{t\to\infty} \frac{r_t}{t} = \gamma.
 \]
\end{theorem}
As for Theorem \ref{Theorem: ballistic growth}, a proof of the theorem can be found in Section \ref{Sec:ProofsOfTheorems2to4}.
\begin{proof}[Sketch of the proof of Theorem \ref{Theorem: lln}] 
The approach is similar to the proof of Theorem \ref{Theorem: ballistic growth}. We also consider for each site $n$ a neighborhood of size $K_n$ to the left of this site. However, to arrive not only at a lower bound but also on the convergence of $r_t/t$, it is not sufficient to construct a single site M. Instead, in Definition \ref{Msites}, we identify a sequence of sites $(M^i)_{i\ge0}$ such that the upcoming births and immunities after each site $M^i$ have the same properties as the site $M$, that is, the number of living parasites grows linearly in the distance of the front from each $M^i$, and in addition, even for small $n < k_0$, there are enough parasites born between $M^i$ and $M^i+n-1$ to ensure a fast infection of site $M^i +n$ with sufficiently high probability. \\
Then, in Lemma \ref{lem:X_mn_conv}, we show a subadditivity property of the front in between reaching these good sites, conditionally on the event that the initial configuration of upcoming offspring and immunities is already good, i.e., $\{M^0 = 0\}$. In particular, we will use approximating jump times $\overline{\nu}_n^i = \nu_{M^i+n}$ defined similarly as $\nu_{n+M}$ (here defined to the right of $M^i$ instead of $M$), which will be a $\phi$-mixing sequence with $\sup_{n\ge 1}\mathbb{E}[\vert\overline{\nu}^i_n\vert^q] < \infty$ for all $q\in\left[1,\frac{(4\wedge\alpha)+1}{2}\right)$. With these approximating jump times, we show, in Lemma \ref{lem:T_Tfinexp}, that moving from one site $M^i$ to the next $M^{i+1}$ happens with finite expectation, which yields the claimed convergence, conditionally on $\{M^0 = 0\}$, after noting that $(M^{i+1}-M^i)_{i\ge0}$ will be i.i.d.~Finally, in Proposition \ref{prop:renewal_sites}, we construct a sequence of renewal sites from which on the front becomes independent of the past. This allows us to conclude that the limit of $\frac{r_t}{t}$ cannot be random, conditioned on survival, and thus, the limit of $\frac{r_t}{t}$ we obtained with positive probability must already hold almost surely. \phantom\qedhere
\end{proof}
 \subsubsection{Two-sided host population}
In this subsection we investigate the survival probability in the two-sided model. We use the results on the one-sided model to obtain the following result on the positive probability of survival and linear spread.
\theoremstyle{plain}\begin{theorem}\label{Theorem: polynomial decay}
Consider a two-sided host population with initial configuration $\zeta = (0,0,A_0\delta_0)$. We make the same assumptions on $A,I$ as in Theorem \ref{Theorem: ballistic growth}, but now further restricting
 \[
 \alpha > 3.
 \]
 Then there is a positive probability for the parasites to survive for infinite time, and there exists a $v>0$ such that for any $0 < \lambda < v$, we have
 \[
 \mathbf{P}\left(\bigcap_{t\ge 0}\{\text{all hosts inside }\{-\lfloor\lambda t\rfloor,\dots,\lfloor \lambda t\rfloor\}\text{ are infected at time }t\}\right) > 0.
 \]
\end{theorem}
\begin{proof}[Sketch of the proof of Theorem \ref{Theorem: polynomial decay}] 
 We construct an event of positive probability in which the right front is only driven by parasites born on a site $x$ with $x\ge 1$, while the left front is driven only by parasites born on a site $x$ with $x\le -1$. We do this by showing that in the one-sided model started with a single parasite, there is a positive probability that the front always stays above $\lfloor\lambda t\rfloor$ and that for any time $t\ge 0$ there is no parasite on any site $x$ with $x\le -1-\lfloor \lambda t\rfloor$. Imposing this event on both the right and left fronts, after they reach $1$ and $-1$, respectively, then yields that the two sides evolve as independent copies of the one-sided model, because parasites born on the right side never reach the linear line $-1-\lfloor\lambda t\rfloor$, while the left front is always below this line, and vice versa for parasites born on the left side. This then yields the claim.\\
To show the result on the one-sided model, we first note that because in the one-sided model the new parasites get only added on the right front, the speed in the negative direction is bounded by the speed of random walk, which is on the scale $\sqrt{t}$ with high probability. Hence, with positive probability for all $t\ge 0$, there are no parasites to the left of $-1-\lfloor\lambda t\rfloor$ at time $t$. To see that with positive probability the front moves with linear speed, we use that $\alpha>3$, which implies that for any $q\in\left(2,\frac{\alpha+1}{2}\right)$ we have $\mathbf{E}[\vert\nu_n\vert^q]<\infty$ for $n\ge 1$. This will allow us to obtain better estimates on the fluctuations of 
\[
\frac{1}{n}\sum_{k=1}^n\nu_k
\] 
around its mean $v$. We then show that the event
\[ G_{k_0} := \left\{M = 0, \forall 1\le n < k_0: \rho_n \le \frac{n}{\lambda}\right\}\] has positive probability, and that on the event $G_{k_0}$, by using tail estimates for sums of $\phi$-mixing series, we also have 
\[
\rho_n \le \rho_{k_0-1} + \sum_{k=k_0}^n \nu_{k} \le \frac{n}{\lambda} 
\]
for all $n\ge k_0$ with positive probability. \phantom\qedhere
\end{proof}
 \section{Good sites and auxiliary jump times}\label{sec:goodsites}
We consider here mainly the case of a one-sided host population with initially $r=0$. Also, w.l.o.g., we assume a random initial configuration $\zeta = (-\infty,0,\eta)$ such that $\eta(0) \overset{d}{=}A$. If not, we can simply shift the process to $0$ after the first jump, where, by definition of the process, $\eta_{\rho_1}(1) \overset{d}{=} A$, and use the strong Markov property. \\
\vspace{0.1cm}\\
 In this section we construct the auxiliary jump times $\{\nu_n:n\in\Z\}$, see Definition \ref{auxiliaryfront}, that will be used in the construction of a lower bound for the actual jump times of the front. We will also analyze the tails as well as the mutual dependencies of these jump times in Lemma \ref{NuMoments} and Lemma \ref{conditinallyphimixing}. Furthermore, we will define the site $M$ and the sequence of sites $(M^i)_{i\ge0}$, see Definition \ref{Msites}, and analyze their properties. As explained in the sketches of the proofs for Theorem \ref{Theorem: ballistic growth} and Theorem \ref{Theorem: lln}, we will use slightly different lower bounds for the jump times to show Theorem \ref{Theorem: ballistic growth} than the one we are using to show Theorem \ref{Theorem: lln}. This difference appears only at the first $k_0$ sites after $M$ and $M^i$, respectively. After that, $\nu_{M+n}$ and $\nu_{M^i+n}$ are defined analogously as upper bounds for the jump times for $n\ge k_0$. \\
 The formal construction of the process will be given in Section \ref{Construction}. We shortly summarize some of the notation that will be introduced there, because we need it for the construction of the lower bound. We keep not only the current positions of parasites but also the places at which each parasite entered the system. Therefore we give each parasite a label $(x,i)\in\Z\times\N$, where $x$ is its birthplace and $i$ numerates all parasites that were born on that site. Also, as will become clear in the construction of the times $\{\nu_n:n\in\Z\}$, we need to follow the virtual paths of parasites after their death. The state space will hence be given by a $5$-tuple $w=(r,\mathcal{L},\mathcal{G},F,\iota)$ with
 \[
 r\in\Z,\mathcal{L}\subset\Z\times\N,\mathcal{G}\subset\Z\times\N,F:\mathcal{L}\cup\mathcal{G}\to\Z,\iota\in\N.
 \]
The entry $r$ is the current position of the front, and $\iota$ is the current remaining immunity of the host at site $r+1$. The set $\mathcal{L}$ consists of the labels of parasites that are currently alive, called living parasites, and $\mathcal{G}$ consists of the labels of parasites that already died (but were alive at some point in the past), called ghost parasites. The map $F$ then assigns a current position to each parasite. There are, of course, some technical restrictions we need to make on these tuples to obtain a strong Markov process on this state space (see in particular the assumptions in Theorem \ref{Theorem:well-def}), which we defer to Section \ref{Construction}.\\
We now explain how we can think of our initial configuration $\zeta= (-\infty,0,\eta)$ as an element of this state space that keeps track of more information. First, to obtain a Markov process, we need to keep track of the current remaining immunity and thus will extend our random initial configuration to $(-\infty,0,\eta,I_1)$. Also we keep track of the position where a parasite entered the system instead of only following all actual parasite positions with $\eta$. We extend the configuration $(-\infty,0,\eta,I_1)$ to an element of this new state space by assuming that every parasite is occupying the site on which it was born and there are no initial ghost parasites, i.e., to the configuration
\[
\left(0,\{(x,i)\in\Z\times\N: x\le 0,1\le i\le \eta(x)\},\emptyset,(x,i)\mapsto x,I_1\right).
\]
With this identification we can assume to have a random initial configuration $w = (0,\mathcal{L},\emptyset,F,I_1)$ defined on $\mathbf{\Omega}$. When the initial configuration is given by $w$, we denote by
\[w_t(w) = (r_t(w),\mathcal{L}_t(w),\mathcal{G}_t(w),F_t(w),\iota_t(w))\]
the configuration of the process at time $t$, by $\mathcal{S}(w)$ the event of survival of parasites for infinite time, i.e. $\mathcal{L}_t(w) \neq\emptyset$ for all $t\ge0$, and by $\mathbb{P}^w$ the measure conditioned on this event. For the rest of this work we drop the reference to the initial configuration, by only writing $\mathbb{P},\mathcal{S},w_t$ instead of $\mathbb{P}^w,\mathcal{S}(w),w_t(w)$. \\
Also, we assume that there are independent collections
\[
\mathbf{Y} = \{Y^{x,i}:x\in\Z,i\in\N\} = \{(Y^{x,i}_t)_{t\ge0}:x\in\Z,i\in\Z\}
\]
of i.i.d.~continuous time simple symmetric random walks with rate $2$ and starting at $0$ as well as 
\[
\mathbf{I} = \{I_x:x\in\Z\}
\]
of i.i.d.~random variables distributed as $I$ and 
\[
\mathbf{A} = \{A_x:x\in\Z\}
\]
of i.i.d.~random variables distributed as $A$. For some label $(x,i)\in\Z\times\N$, the path that the parasite with this label takes after it enters the system at some site $y\in\Z$ will be given by $y+Y^{x,i}$. We note that for $x>0$ we will have $y=x$, but if the parasite was already alive in $w$, then it may be at a different site than its birthplace. Also for $x>0$, the initial immunity of the host at $x$ will be given by $I_x$, and the number of parasites that will be produced after the host at site $x$ is infected will be given by $A_x$. Finally we denote by
\[
\rho_n := \inf\{t\ge0:r_t \ge n\}
\]
the time, when the front reaches site $n\ge0$, i.e., when the host at site $n$ gets infected. 

With all the notation introduced, we now lay out the general approach. We identify a site $M^0$ such that from that site on, the upcoming births and deaths are in a good configuration, in the sense that the amount of living parasites grows linearly in the distance of the front to $M^0$ and the linear growth has some high enough slope. Why these properties are useful will become clear after we introduce the jump times $\{\nu_n:n\in\Z\}$ in Definition \ref{auxiliaryfront} that will be used to lower bound the front, which is also why we start this section by constructing these times before constructing $(M^i)_{i\ge0}$. The constructed lower bound in particular will only use parasites born to the right of the site $M^0$, and hence not depend on the initial configuration anymore.
After reaching the good site $M^0$, we simply repeat the procedure to find the next site $M^1 > M^0$, which has a good configuration of upcoming births and deaths. Iterating this yields a sequence $(M^i)_{i\ge0}$ of good sites, such that even ignoring any parasites that were born before reaching $M^i$ and only starting with $A_{M^i}$ parasites at site $M^i$ has a linear speed with high probability. \\
Then we establish a subadditivity property of the front that holds if we start with only $A_0$ parasites on $0$ and no other parasites in the system and the first good site $M^0 =0$. This allows us to prove the law of large numbers in the event that $M^0 = 0$, by controlling the tails of the hitting times of $M^i$. Then, showing that the limit of $\frac{r_t}{t}$ is not random, conditionally on survival, will show the strong law of large numbers in Theorem \ref{Theorem: lln}. For arbitrary initial configurations, this approach, however, does not work, and we will only show the weaker result of ballistic growth in Theorem \ref{Theorem: ballistic growth}.\\
In Definition \ref{auxiliaryfront} we will construct random times $\{\nu_n: n\in\Z\}$ on $\mathbf{\Omega}$ that will be used as the jump times for a lower bound of the front. Precisely, we will show that in the event that $\{M^i = k\}$, the time $\nu_{k+n}$ only depends on $\mathbf{A},\mathbf{I},\mathbf{Y}$ with an index inside $k+n-K_{k+n}$ for some random distance $K_{k+n}\le n$ and is given by the first hitting time of a one-sided barrier by a collection of random walks. In Lemma \ref{lem:Kmtail} we will show that this distance $K_{k+n}$ does not become too large with high probability.
In particular, this will show that on one hand the tails of $\nu_{M^i+k}$ fall fast enough in Lemma \ref{NuMoments} and on the other that the sequence $(\nu_{M^i+n})_{n\ge1}$ is weakly dependent, specifically that it is $\phi$-mixing in Lemma \ref{conditinallyphimixing}. Also, in Proposition \ref{Couplinglowbound} we show that, still on the event $\{M^i = k\}$, we have $\nu_{k+n} \ge \rho_{k+n}-\rho_{k+n-1}$ for any $k\ge0,n\ge 1$. The definition of $M$ that is used in the proof of Theorem \ref{Theorem: ballistic growth} will be slightly different and only yield that on the event $\{M=k\}$ we have $\nu_{k+n}\ge \rho_{k+n}-\rho_{k+n-1}$ for all $n\ge k_0$ with some explicit constant $k_0\in\N$. The different treatment is due to the reason that we want every jump time above $M^i$ to have a small tail, while for $M$ we only need this eventually, i.e., for $n\ge k_0$. \\
We set $\overline{\nu}_n^i := \nu_{M^i+n}$, and using the tails of $\nu_n$ as well as their dependence, we want to obtain good bounds on
\begin{equation}\label{partsums}
\sum_{j=1}^{n}\overline{\nu}_{j}^i
\end{equation}
which is an upper bound for $\rho_{M^i+n}-\rho_{M^i}$. The tail estimates and the weak dependence, using standard results for $\phi$-mixing sequences, allow us to control the moments of their partial sums \eqref{partsums} and in Lemma \ref{Qconvspeed} show that the $q$-th moment only grows like $\mathcal{O}(n^{\frac{q}{2}})$, if $\alpha > 3$ and $q < \frac{(4\wedge\alpha)+1}{2}$, which will be useful in proving that moving from $M^i$ to $M^{i+1}$ happens with finite expectation. 
\subsection{Construction and basic properties of auxiliary jump times and good sites}
Before constructing the good sites $(M^i)_{i\ge0}$, we want to analyze a candidate for an upper bound of the jump times and under which conditions this actually is an upper bound and what tails these candidates have. These conditions then naturally lead to what properties the good sites $(M^i)_{i\ge0}$ should satisfy in order to be able to define the jump times $(\overline{\nu}^i_n)_{n\ge1}$ as an upper bound for the jump times $\rho_{M^i+n}-\rho_{M^i+n-1}$ that have the needed tail behavior of their partial sums \eqref{partsums}. \\
 The key idea in defining $\{\nu_n:n\in\Z\}$ is to obtain estimates on the number of living parasites when the front is at a certain position $x$, independent of the paths $\mathbf{Y}$, only relying on the random variables $\mathbf{A},\mathbf{I}$. Then we estimate the time that these parasites would need to push the front further by establishing results on first hitting times of a one-sided barrier by a collection of simple symmetric random walks. Combining these two estimates, using the independence of $\mathbf{A},\mathbf{I},\mathbf{Y}$, we then obtain an upper bound, denoted by $\nu_n$, for the times the process $(r_t)_{t\ge0}$ takes to jump from $n-1$ to $n$, which was denoted by $\rho_{n}-\rho_{n-1}$.\\
First we observe that for any $n\ge1$ the number of parasites that die at site $n$ is given by $I_n$, and the process $(r_t)_{t\ge0}$ advances exactly when this many parasites have died. Hence, we have
\[
\rho_{n+1} = \inf\{t\ge0:r_t\ge n+1\} = \inf\{t\ge \rho_{n}: \vert\mathcal{G}_t\vert-\vert\mathcal{G}_{\rho_{n}}\vert = I_{n+1}\}
\]
and thus
\[
\vert\mathcal{G}_{\rho_{n+m}}\vert - \vert\mathcal{G}_{\rho_n}\vert = \sum_{k=n+1}^{n+m} I_k.
\]
for all $n\ge0,m\ge1$. In words, the amount of parasites that die during the time that the front needs to advance from site $n$ to site $n+m$ is given by the sum of the immunities between $n+1$ and $n+m$. Because for $x>0$, exactly $A_x$ new parasites are added at $x$ whenever the front jumps onto $x$, the number of living parasites at time $\rho_{n+m}$ that have been generated at one of the sites $\{n, ..., n+m-1\}$ satisfies the inequality
\[
\vert\mathcal{L}_{\rho_{n+m}}\cap(\{n,\dots,n+m-1\}\times\N)\vert \ge \sum_{k=n}^{n+m-1}A_k - \sum_{k=n+1}^{n+m} I_k.
\]
This gives us a lower bound on the number of living parasites at time $\rho_{n+m}$, independent of the paths $\mathbf{Y}$, and leads to the following definition.

For the rest of this work we fix some
\[\beta_A\in(\mu_I,\mu_A), \quad \quad \beta_I\in(\mu_I,\beta_A).\] 
Furthermore, we set
\begin{equation}\label{qkcond}
k_0 := \left\lceil\frac{\alpha+1}{\beta_A-\beta_I}\right\rceil.
\end{equation}

\theoremstyle{definition}\begin{definition}\label{auxiliaryfront}
For any $k\ge 0$ we define
\[
m(A,k) := \inf\{j\ge k: \mathbf{P}(A=j) > 0\}, m(I,k) := \inf\{j\ge k:\mathbf{P}(I=j) > 0\},
\]
and set $m_I := m(I,1),m_A := m(A,\beta_A\vee(4+m(I,1)))$.
For $1\le k < k_0$ we define
\[
\underline{A}^{\text{good}}(k) := \left\{(a_1,\dots,a_k)\in\N_0^k\vert\forall 1\le j \le k:a_j \ge m_A\right\},
\]
and
\[
\underline{I}^{\text{good}}(k) := \left\{(i_1,\dots,i_k)\in\N^k\vert\forall 1\le j \le k:i_j = m_I\right\}.
\]
For $k\ge k_0$ we define
\[
\underline{A}^{\text{good}}(k) := \left\{(a_1,\dots,a_k)\in\N_0^k:\sum_{j=1}^ka_j \ge \beta_Ak\right\},
\]
and
\[
\underline{I}^{\text{good}}(k) := \left\{(i_1,\dots,i_k)\in\N^k:\sum_{j=1}^ki_j \le \beta_Ik\right\}.
\]
For $1\le k < k_0$ the set $\underline{A}^\text{good}(k)$ are the offspring constellations of length $k$ such that on each vertex there are born at least $4$ parasites more than the minimal immunity, and $\underline{I}^\text{good}(k)$ are the immunity constellations of length $k$ such that each immunity has the minimal possible value. For $k\ge k_0$ the set $\underline{A}^\text{good}(k)$ are the offspring constellations of length $k$ with on average at least $\beta_A$ many parasites per site, while $\underline{I}^\text{good}(k)$ are the immunity constellations of length $k$ with on average a strength of at most $\beta_I$ per site.
For $n\in\Z$ we now go backwards from site $n$, starting at $k_0$ steps, until we find two good configurations in both offspring and immunities. Precisely, we define
\[
K_n := \inf\left\{k\ge 1:(A_{n-1},\dots,A_{n-k})\in\underline{A}^\text{good}(k),(I_n,\dots,I_{n-k+1})\in\underline{I}^\text{good}(k)\right\}
\]
as the smallest distance greater than or equal to $k_0$ where both the offsprings and immunities between $n-K_n^m$ and $n$ are good, in the sense of the just defined sets. \\
For $k\in\N$ and a good offspring constellation $\mathbf{a}\in\underline{A}^\text{good}(k)$ we define the corresponding labels of parasites below site $n$ as
\[ 
\underline{\mathcal{V}}^\text{good}_n(k,\mathbf{a}) := \{(x,i)\in\Z\times\N:n-k\le x\le n-1,1\le i\le a_{n-x}\},
\]
and define the random label set
\[
\mathcal{W}_n := \underline{\mathcal{V}}_n^\text{good}(K_n,(A_{n-1},\dots,A_{n-K_n})).
\]
For $(x,i)\in\Z\times\N$ and $k\in\Z$ we define
\[
\tau^{x,i}_k := \inf\{t\ge0:Y^{x,i}_t = k\},
\]
the first time the parasite with label $(x,i)$ has reached the site at distance $k$ to the right of its birthplace, and
\[
\nu_n := \inf\left\{t\ge0:\sum_{(x,i)\in \mathcal{W}_n}\mathds{1}_{\tau^{x,i}_{n-x}\le t} \ge \begin{cases} m_I K_n,&\text{if }1\le K_n < k_0\\\beta_I K_n,&\text{if }K_n\ge k_0\end{cases}\right\}
\]
as the first time at which at least a number equal to at least the accumulated sums of immunities between $n-K_n+1$ and $n$ of parasites born between $n-K_n$ and $n-1$ reached site $n$. 
\end{definition}
\begin{remark}
The role of the specific choice $4$ will become clear in the proof of Lemma \ref{NuMoments} and is just that $4$ is the smallest integer that guarantees high enough moments. \\ Because clearly $(\nu_n)_{n\in\Z}$ and $(K_n)_{n\in\Z}$ are stationary under the unconditioned law $\mathbf{P}$, to simplify notation, we set $\nu := \nu_1,K := K_1$ when only stating properties depending on the distribution of a single variable.
\end{remark}
The idea behind the definition of $\nu_n$ is the following. First, we ignore the deaths of parasites and use the collection $\mathbf{Y}$ to follow the (potentially virtual) paths of each parasite. We then want to distinguish which random walks correspond to living parasites and which correspond to ghost parasites at time $t=\rho_{n-1}$. The real process jumps to site $n$, when $I_n$ many parasites, that were alive at time $\rho_{n-1}$, reached site $n$. Firstly, we note that placing these parasites back on the site where they were born and then following their trajectory only increases the first hitting time of site $n$. Hence, given $\mathcal{F}_{\rho_{n-1}}$, the time $\rho_{n}-\rho_{n-1}$ is smaller than the first time $I_n$ many random walks, corresponding to parasites alive at time $\rho_{n-1}$ and starting from their birthplace, reached site $n$. However, deciding which specific labels correspond to living parasites depends on the trajectories of all parasites before time $\rho_{n-1}$ and thus is too complicated. Hence, for any $k\ge 1$, we only bound the amount of ghost parasites at time $\rho_{n-1}$ that were born inside $[n-k,n-1]$. Once this amount of random walks and at least one additional random walk, all with labels in $[n-k,n-1]$, have reached site $n$, at least one of those (but we do not know which one) has to correspond to a parasite that was alive at time $\rho_{n-1}$. Thus, this gives an upper bound on the jump time from $n-1$ onto $n$. We will make this reasoning precise in the upcoming Proposition \ref{Couplinglowbound}. \\
We note that if $n-K_n < 0$, then the labels $\mathcal{W}_n$ are not necessarily corresponding to real parasites, depending on the initial configuration. But if $n-K_n \ge 0$, then $\mathcal{W}_{n}$ are exactly the labels of parasites that were born during the time $[\rho_{n-K_n},\rho_{n-1}]$. Because parasites can only die at sites to the right of their birthplace, an upper bound for the number of ghost parasites at time $\rho_{n-1}$, born inside $[n-K_n,n-1]$, is given by $I_{n-K_n+1}+\dots+I_{n-1}$. The first and second conditions in the definition of $K_n$ exactly guarantee that the sum of immunities in $[n-K_n+1,n]$ is much smaller than the sum of parasites born in that interval. This will ensure that there are enough random walks corresponding to a parasite born in $[n-K_n,n-1]$, such that the time it takes for $I_{n-K_n+1}+\dots+I_{n-1}$ many of those random walks to reach $n$ has finite moments up to some high enough power, depending on the parameter $\alpha$, which was chosen such that $\mathbf{E}[I^{2\alpha}]<\infty$. The different treatment for $k< k_0$ and $k\ge k_0$ comes from the fact that for $k\ge k_0$, the condition on the sums already guarantees a surplus of $(\beta_A-\beta_I)k_0 \ge \alpha + 1$ parasites, which will correspond to certain moments of $\nu_n$ conditionally on $K_n \ge k_0$; see the proof of Lemma \ref{NuMoments}. To also obtain moments estimates conditionally on $K_n < k_0$, we need to make the surplus of parasites for $k<k_0$ large enough and note that the definition of $K_n$ makes this surplus at least $k(m_A-m_I) \ge 4k$. \\
We formalize the reasoning above in the following proposition, showing that $\nu_n$ is a lower bound for the jump time $\rho_n-\rho_{n-1}$ of the real front, if $K_n\le n$.
\begin{proposition}\label{Couplinglowbound}
For all $n\ge k_0$ we have that
\[
\mathds{1}_{K_n \le n}(\rho_{n}-\rho_{n-1}) \le \nu_n 
\]
\end{proposition}
\begin{proof}
Clearly the right-hand side is non-negative, and hence it suffices to show the inequality on the event $K_n \le n$.
By definition of the process, placing all parasites alive at time $\rho_n$ back to their birthplace, we obtain
\[
\rho_{n} \le \rho_{n-1} + \inf\left\{t\ge0:\sum_{(x,i)\in\mathcal{L}_{\rho_{n-1}}}\mathds{1}_{\tau^{x,i}_{n-x}\le t}\ge I_{n}\right\}
\]
and hence will show that, on $\{K_n\le n\}$, the second term is smaller than $\nu_{n}$. \\
At time $\nu_{n}$ we have that \[Z_n := \begin{cases}m_I K_n,&\text{if }1\le K_n < k_0\\\lceil\beta_I K_n\rceil,&\text{if }K_n\ge k_0\end{cases}\] many random walks $(x+Y^{x,i}_t)_{t\ge0}$ with birth label
\[(x,i)\in \mathcal{W}_{n}= \left\{(x,i)\in\Z\times\N: n-K_{n}\le x < n, 1\le i\le A_x\right\}\] have reached site $n$, and since we are on $\{K_n \le n\}$, these random walks correspond to parasites that were alive in the real process. Let us denote by $\mathcal{B}\subset\mathcal{W}_{n}$ the labels of the $Z_n$ many parasites having reached site $n$ at time $\nu_n$. Since parasites can only die when jumping to the right of the front, and only $I_x$ parasites die at each site $x$, we have that out of the parasites with labels inside $\mathcal{W}_{n}$, only $\sum_{j=1}^{K_{n}-1} I_{n-j}$ many can be ghosts at time $\rho_{n-1}$. In other words 
\[\begin{split}
\vert\mathcal{W}_{n}\cap\mathcal{L}_{\rho_{n-1}}\vert &\ge \sum_{j=1}^{K_{n}}A_{n-j}-\sum_{j=1}^{K_{n}-1} I_{n-j}= \sum_{j=1}^{K_{n}}A_{n-j}+I_{n}-\sum_{j=0}^{K_{n}-1} I_{n-j} \\
&= \sum_{j=1}^{K_{n}}A_{n-j}+I_{n}-\sum_{j=1}^{K_{n}} I_{n-j+1}\ge \sum_{j=1}^{K_{n}}A_{n-j}+I_{n}-Z_n.
\end{split}\]
In particular, since $\vert\mathcal{B}\vert =Z_n$, this implies
\[
\begin{split}
\vert\mathcal{B}\cap\mathcal{L}_{\rho_{n-1}}\vert &=\vert\mathcal{W}_{n}\cap\mathcal{L}_{\rho_{n-1}}\vert - \vert(\mathcal{W}_{n}\setminus\mathcal{B})\cap\mathcal{L}_{\rho_{n-1}}\vert \\&\ge \sum_{j=1}^{K_{n}}A_{n-j}- Z_n +I_{n} - \vert\mathcal{W}_{n}\setminus\mathcal{B}\vert \\
&= \sum_{j=1}^{K_{n}}A_{n-j}-Z_n +I_{n} - \left(\sum_{j=1}^{K_{n}}A_{n-j}-Z_n\right) \ge I_{n}.
\end{split}
\]
Hence $\vert\mathcal{B}\cap\mathcal{L}_{\rho_{n-1}}\vert \ge I_{n}$ living parasites reached site $n$ at time $\nu_{n}$ and thus
\[
\begin{split}
\nu_{n} &\ge \inf\left\{t\ge0:\sum_{(x,i)\in\mathcal{L}_{\rho_{n-1}}}\mathds{1}_{\tau^{x,i}_{n-x}\le t}\ge I_{n}\right\}.
\end{split}
\]
This already concludes the proof.
\end{proof}
This proposition motivates the definition of the good sites $(M^i)_{i\ge0}$ and $M$ that will be used to prove Theorems \ref{Theorem: lln} and \ref{Theorem: ballistic growth}. Namely, we want to define $(M^i)_{i\ge0}$ and $M$ in such a way that 
\begin{equation}\label{eq:K_M_condition}
\forall i\ge0,n\ge 1:K_{M^i+n} \le n~\text{ and }~\forall n\ge k_0:K_{M+n} \le n.
\end{equation}
 Because then, by Proposition \ref{Couplinglowbound}, we have 
\begin{equation}\label{eq:nu_n_upperbound}
\begin{split}
\forall i\ge 0,n\ge1&:\nu_{M^i+n} \ge \rho_{M^i+n}-\rho_{M^i+n-1}~\text{ and }\\\forall n\ge k_0&:\nu_{M+n} \ge \rho_{M+n}-\rho_{M+n-1}.
\end{split}
\end{equation}
Also, in this way, for any $i\ge0$, the sequence $\{\nu_{M^i+n}:n\ge 1\}$, respectively $\{\nu_{M+n}:n\ge k_0\}$, conditionally on $M^i$, respectively conditionally on $M$, only depends on
\[
\{A_x,I_{x+1},Y^{x,j}: x\ge M^i,j\in\N\} ~\text{ resp. }~\{A_x,I_{x+1},Y^{x,j}: x\ge M,j\in\N\},
\]
and in particular will be independent of the initial configuration. \\
 The treatment of the first $k_0$ sites after $M^i$ will be slightly different from the treatment of the first $k_0$ sites after $M$, because in the former case we need even the jumps to these sites to have sufficient moments, while only relying on parasites that are born above $M^i$. In the definition of $M$, however, we can argue more simply and can define $M$ as follows.
\theoremstyle{definition}\begin{definition}\label{Msite}
Set $L_0 = 0$ and for $k \ge 0$, define 
\[L_{k+1}:=\inf\{l \ge L_k+k_0: K_l > l-L_k\}\]
and let
\[
N := \inf\{k\ge0: L_{k+1} = \infty\},~~M := L_{N}.
\]
Furthermore, we set
\[
T := \inf\{t\ge0: r_t \ge M+k_0\}.
\]

\end{definition}

To define $M^i$ we introduce some notation to treat the first $k_0$ sites differently than the other sites. We define
\[
G:\bigcup_{n\in\N}\N_0^n\times\N^n \to \{0,1\},
\]
such that for $n\in\N$ and $\mathbf{a} = (a_0,\dots,a_{n-1})\in\N_0^n,\mathbf{i} = (i_1,\dots,i_{n})\in\N^n$ we have $G(\mathbf{a},\mathbf{i}) = 1$ if and only if 
\begin{itemize}
 \item for all $1\le j < (n+1)\wedge k_0$ it holds that
 \begin{equation}\label{eq:goodsmallconf}
\begin{split}i_j &= \inf\{l\ge 1:\mathbf{P}(I=l)>0\},\\a_{j-1} &\ge \inf\{l\ge \beta_A\vee(i_j+4):\mathbf{P}(A = l) > 0\}.
\end{split}
 \end{equation}
 Note that this is just the condition of a good configuration of length $j<k_0$ in Definition \ref{auxiliaryfront}.
 \item for all $k_0 \le j \le n$ there is some $k_0\le i \le j$ such that
 \[
 (a_{j-i},\dots,a_{j-1})\in \underline{A}^\text{good}(i) ~\text{ and }~(i_{j-i+1},\dots,i_j)\in\underline{I}^\text{good}(i).
 \]
\end{itemize}
\theoremstyle{definition}\begin{definition}\label{Msites}
 Let $L_0^0 := 0$ and recursively for $k\ge0$ set
\[
L_{k+1}^0 := \inf\left\{l > L_k^0: G((A_{L_k^0},\dots,A_{l-1}),(I_{L_k^0+1},\dots,I_l)) = 0\right\}
\]
and define
\[
N^0 := \inf\{k\ge0: L_{k+1}^0 = \infty\},~~M^0 := L_{N^0}^0. 
\]
Then we set, recursively for $i\ge 0$, $L^{i+1}_0 := M^i+1$ and for $k\ge0$ 
\[
L^{i+1}_{k+1} := \inf\left\{l > L_k^{i+1}: G((A_{L_k^{i+1}},\dots,A_{l-1}),(I_{L_k^{i+1}+1},\dots,I_l)) = 0\right\}
\]
as well as
\[
N^{i+1} := \inf\{k\ge 0: L^{i+1}_{k+1} = \infty\},~~M^{i+1} := L^{i+1}_{N^{i+1}}.
\]
For $j\ge0,k\ge 1$ we define the event
\[
\mathcal{G}_{j,j+k} := \bigcap_{m=1}^k\{G((A_j,\dots,A_{j+m-1}),(I_{j+1},\dots,I_{j+m}))= 1\}
\]
of good configurations between $j$ and $j+k$ and set $\mathcal{G}_j := \mathcal{G}_{j,\infty}$.
Also, we set
\[
T^i := \inf\{t\ge0:r_t \ge M^i\}
\]
and $\overline{\nu}_n^i := \nu_{M^i+n}$ for $n\ge 1$.
\end{definition}
\begin{remark}\label{rem:M_lessrestrictive}
 We note that for all $j\ge 0$ we have
 \[
 \{M^i = j\} \subset \mathcal{G}_{j} \subset \bigcap_{n=1}^\infty\{K_{j+n} \le n\}
 \] and 
 \[
 \{M=j\} \subset\bigcap_{n=k_0}^\infty\{K_{j+n}\le n\}.
 \]
\end{remark}
To begin, we verify that the definitions above are actually well-defined and establish the following lemma.
\theoremstyle{plain}\begin{lemma}\label{lem:welldef}
We have $N<\infty$ almost surely. Also, if $\mathbf{P}(A-I\ge 4) > 0$ (c.f.~\eqref{eq:goodsmallconf}), then for any $i\ge0$ we have $N^i < \infty$ almost surely. In particular, $(M^i)_{i\ge0}$ and $M$ are well-defined, and $T,T^0,T^1,\dots < \infty$ almost surely under $\mathbb{P}$.
\end{lemma}
\begin{proof}
We recall $\mathcal{G}_j$ from Definition \ref{Msites} and set $\mathcal{G} : =\mathcal{G}_0$. We give the proof only for $(M^i)_{i\ge0}$ and note that, using Remark \ref{rem:M_lessrestrictive}, the results for $M$ follow analogously.\\
For any $\N_0$-valued random variable $X$ on $\mathbf{\Omega}$ and any $k\ge0$, we recall
 \[
 m(X,k) := \inf\{j\ge k:\mathbf{P}(X= j)>0\}
 \]
 and $m_I = m(I,1), m_A = m(A,\beta_A\vee(4+m(I,1)))$ from Definition \ref{auxiliaryfront}. Since $m_I$ is the smallest value that $I$ can take, the assumption $\mathbf{P}(A-I \ge 4) > 0$ together with $\beta_A < \mathbf{E}[A]$ implies that $m_A$ is well defined. Then, by definition, for $1\le k < k_0$ we have
 \[
 \{G((A_0,\dots,A_{k-1}),(I_1,\dots,I_k)) = 1\} = \bigcap_{j=1}^k\{A_{j-1} \ge m_A,I_j = m_I\}.
 \]

 We observe that $\beta_I > \mathbf{E}[I]\ge m_I$ and by definition \[m_A = m(A,\beta_A\vee(4+m(I,1)))\ge \beta_A.\]
 We obtain that the event
 \begin{equation}\label{eq:M_eqj}
 \bigcap_{j=1}^{k_0}\left\{ \begin{matrix}A_{j-1} \ge m_A,\\I_j = m_I \end{matrix} \right\}\cap\bigcap_{m=1}^\infty\left\{\sum_{k=1}^m A_{k_0+k-1}\ge \beta_A m,\sum_{k=1}^m I_{k_0+k} \le \beta_I m\right\}
 \end{equation}
 is contained in $\mathcal{G}$. The right event is that two independent random walks with step distribution $A-\beta_A$ and $\beta_I-I$ stay non-negative for all times. Since $\mathbf{E}[A]>\beta_A>\beta_I > \mathbf{E}[I]$ this event has positive probability. By assumption, the left event also has positive probability. The independence of the two events shows that $\mathbf{P}(\mathcal{G}) > 0$. \\
 By definition, we have that for $k,m\ge 0,n\ge k_0$ the event  $\{L_k^0 = m, L_{k+1}^0 = m+n\}$ is given by
\[
\{L_k^0=m\}\cap\left(\bigcap_{j=1}^{k_0-1}\left\{ \begin{matrix}A_{m+j-1} \ge m_A,\\I_{m+j} = m_I \end{matrix} \right\}\right)\cap\left(\bigcap_{j=k_0}^{n-1}\{K_{m+j}\le j\}\right)\cap\{K_{m+n} > n\},
\]
using the definition of $K_{m}$ given in \ref{auxiliaryfront}. For $1\le n < k_0$ the event $\{L_k^0 = m, L_{k+1}^0 = m+n\}$ is just given by
\[
 \{L_k^0=m\}\cap\left(\bigcap_{j=1}^{n-1}\left\{ \begin{matrix}A_{m+j-1} \ge m_A,\\I_{m+j} = m_I \end{matrix} \right\}\right)\cap(\left\{A_{m+n-1} < m_A\right\}\cup\{I_{m+n}\neq m_I\}).\]
In particular, going inductively through $k$, the event $\{L^0_k= m\}$ only depends on $A_0,\dots,A_{m-1}$ and $I_1,\dots,I_m$, and the increments $L_1^0-L_0^0,\dots,L_{k+1}^0-L_k^0$ are independent under $\mathbf{P}$ conditionally on $L_k^0 < \infty$. Also, under $\mathbf{P}$ conditionally on $L_k^0 < \infty$, the increments $L_1^0-L_0^0,\dots,L_k^0-L_{k-1}^0$ have the same distribution. This yields that under $\mathbf{P}$, $N^0$ has a geometric distribution with success probability $\mathbf{P}(\mathcal{G}) > 0$ and is thus almost surely finite under $\mathbf{P}$ and hence also under $\mathbb{P}$. An analogous argument, where in view of Remark \ref{rem:M_lessrestrictive} we only need to show the positive probability of the second event in \eqref{eq:M_eqj}, which holds with only the assumption $\mathbf{E}[I] < \mathbf{E}[A]$, shows that also $N$, under $\mathbf{P}$, has a geometric distribution with positive success probability and thus is almost surely finite.\\
 We now investigate $N^{i+1}$ for $i\ge 0$. Since by definition we have
 \[
 \bigcap_{j=1}^\infty\{G((A_{M^{i}},\dots,A_{M^{i}+j-1}),(I_{M^{i}+1},\dots,I_{M^{i}+j})) = 1\},
 \]
the offspring after $M^i$ tend to be larger, and the immunities tend to be smaller. Hence, for any site $L^{i+1}_k > M^i$, it becomes more probable to never find a bad configuration above $L^{i+1}_k$, i.e., achieve the event $\{L_{k+1}^{i+1} = \infty\}$. We see that $N^{i+1}$ is stochastically dominated by a geometric distribution with success probability $\mathbf{P}(\mathcal{G})$ and is thus also almost surely finite under $\mathbf{P}$ and thus also under $\mathbb{P}$. \\
We now show that also $T^0$ is almost surely finite under $\mathbb{P}$. We saw above that $M^0<\infty$ almost surely, and on the event $\{M^0 = k\}$, the time $T^0$ is bounded by the following random time $\sigma^k$. For $j\ge 0$ let
 \[
 S_j := \inf\left\{t\ge 0\vert\, \forall 0\le x\le j\,\forall 0\le i \le A_x: \max_{0\le s\le t}Y^{x,i}_s + x \ge j+1\right\}
 \]
 be the time until all parasites that have been born on the first $j$ sites reach site $j+1$ 
and 
 \[
 \sigma^k := \sum_{j=0}^{k-1} S_j
 \]
 be the accumulation of these times.\\ 
Conditioned on parasites' survival, the time $\rho_{j+1}-\rho_j$ that the front needs to jump from $j$ to $j+1$ is at most $S_j$.
 In particular, this implies that the time $T^i$ to reach site $M^i$ is bounded by
 \[
 T^i \le \sum_{k=0}^\infty\mathds{1}_{M^i = k}\sigma^k
 \]
 $\mathbb{P}$-almost surely. Analogously, the time $T$ to reach site $M+k_0$ is bounded by
\[
T \le \sum_{k=0}^\infty\mathds{1}_{M = k}\sigma^{k+k_0}
\]
$\mathbb{P}$-almost surely.
 
Using \cite[Theorem A.1]{RS2004}, we can estimate the arrival times $\tau^{x,i}_k$ of random walks and derive for large enough $t$
 \[
 \begin{split}
 \mathbf{P}&(S_j> t) \\&\le \mathbf{P}\left(\max_{0\le x\le j}A_x > \log t\right)+\mathbf{P}\left(\max_{0\le x\le j}A_x \le \log t,\bigcup_{x=0}^j\bigcup_{i=1}^{\lfloor\log t\rfloor}\{\tau^{x,i}_{j+1-x}> t\}\right) \\
 &\le \frac{j\mathbf{E}[A]}{\log t} + \sum_{x=0}^j\sum_{i=1}^{\lfloor\log t\rfloor} \frac{2(j+1-x)}{\sqrt{2\pi t}}\cdot 2 \\
 &\le\frac{j\mathbf{E}[A]}{\log t}+2\log t\frac{(j+1)(j+2)}{\sqrt{2\pi t}} \to 0~~(t\to\infty),
 \end{split}
 \]
 and hence
 $S_j < \infty$ $\mathbf{P}$-almost surely and thus also $\sigma^k < \infty$ $\mathbf{P}$-almost surely. \\
 Consequently, also $T_i$ and $T$ are $\mathbb{P}$-almost surely finite.
\end{proof}
\subsection{Auxiliary results}
A key for analyzing the sequences $(M^i)_{i\ge0}$ and $(\nu_n)_{n\in\Z}$ is an estimate of the tail of $K$. Using the moment assumptions on $I$ and $A\ge 0$ a.s., we obtain the estimate for the probability of a bad configuration of length $k\ge k_0$
\[
\mathbb{P}((A_{n-1},\dots,A_{n-k})\notin\underline{A}^\text{good}(k)\text{ or } (I_n,\dots,I_{n-k+1})\notin\underline{I}^\text{good}(k)) = \mathcal{O}(k^{-\alpha}).
\]
Since $K$ is the smallest distance at which a good configuration occurs, the event $\{K>n\}$ is contained in the event that all configurations with length $k=k_0,k_0+1,\dots,n$ are bad. Using classical results of Sparre-Andersen (c.f. \cite{A1954}) on exactly these types of probabilities will allow us to improve the just-obtained estimate and give the following lemma.
\theoremstyle{plain}\begin{lemma}\label{lem:Kmtail}
 For all $\varepsilon > 0$ there is a $C_K = C_K(\varepsilon)\ge 0$ such that for all $n\ge k_0$ we have
 \[
 \mathbf{P}(K > n) \le C_K n^{-(\alpha+1-\varepsilon)}.
 \]
\end{lemma}
The proof will be given in Section \ref{pro:Kmtail}.\\
With the tail of $K$ at hand, we can begin to analyze the sequence $(M^i)_{i\ge0}$. As seen in the proof of Lemma \ref{lem:welldef}, the number of trials $N^0$ stochastically dominates the amount of trials $N^i$ for $i\ge 1$. Hence, the increments $M^{i+1}-M^i$ should be stochastically dominated by $M^0$. In addition, the event $\mathcal{G}_{M^{i+1}}$, that is, the condition of only good offspring numbers and immunities configurations after $M^{i+1}$, overrides the condition $\mathcal{G}_{M^i}$ of the previous good site. This means that the distribution of offspring numbers and immunities after $M^i$ is the same for all $i\ge0$, implying that $(M^{i+1}-M^i)_{i\ge0}$ is i.i.d.~under $\mathbb{P}$. Using the tail of $K$, we can also bound the tail of $M^0$ and thus the tail of $(M^{i+1}-M^i)_{i\ge0}$. Precisely, we have the following lemma.
\theoremstyle{plain}\begin{lemma}\label{lem:M_tail}
 The sequence $(M^{i+1}-M^i)_{i\ge 0}$ is i.i.d.~under $\mathbb{P}$, and for any $\varepsilon > 0$ there is a $C\ge0$ such that for all $l\in\N$ we have \[
 \mathbb{P}(M^{i+1}-M^i \ge l) \le \mathbb{P}(M^0 \ge l) \le Cl^{-(\alpha-\varepsilon)}.
 \]
\end{lemma}
The proof will also be given in Section \ref{pro:M_tail}. Our estimate of the tail of $M^{i+1} - M^i$ is by a factor of order $\frac{1}{l}$ heavier than the tail of $K$ because the event $\{\infty > L_1^0 > l\}$ is contained in the event that $\{K_j> j\}$ for some $j > l$. A union bound over these events yields the estimate.\\
Next, we consider the sequence $\{\nu_n:n\in\Z\}$. By construction, each jump time $\nu_n$ only depends on random variables with indices inside $[n-K_n,n]$. We will make use twofold of the just-obtained tail estimates for $K$. Firstly, we will establish tail estimates for $\nu$ by showing that if $K$ is not too large, the surplus of parasites reaches the boundary sufficiently fast. Secondly, the tail estimates of $K$ allow us to control the dependencies among the jump times $(\nu_n)_{n\in\Z}$, since conditionally on $K_{n+k}\le k$, the variables $\nu_{n+k}$ and $\nu_n$ depend on different sets of independent random variables. We begin by stating the former result, which reads as follows.
\theoremstyle{plain}\begin{lemma}\label{NuMoments}
For any $q < \frac{(4\wedge\alpha)+1}{2}$, there is a $C>0$ such that for all $t>0$ we have
\[
\mathbf{P}(\nu > t) \le C t^{-q}.
\]
In particular we have $\nu\in L^q(\Omega,\mathcal{F},\mathbf{P})$ for all $ q < \frac{(4\wedge\alpha)+1}{2}$.
\end{lemma}
The proof will be given in Section \ref{pro:NuMoments}. We note that the exponent in the tail of $\nu$ is half as large as in the tail of $K$, because a typical random walk is at a distance of $\sqrt{t}$ at time $t$. Hence, the event $\{\nu_n\le t\}$ is only likely if $\{K_n \ll \sqrt{t}\}$. On the other hand, the different treatment of distance $k_0$ in the definition of $K$ can be explained as follows. On the event $\{K_n=k\}$ with $\sqrt{t}\gg k\ge k_0$, at most $\lfloor\beta_I k\rfloor$ parasites out of a set of at least $\lceil\beta_A k\rceil$ independently moving parasites need to reach a one-sided boundary, which is at a distance at most $k$ from their starting point. At time $t$, a single parasite has not reached the boundary with probability $\mathcal{O}\left(\frac{k}{\sqrt{t}}\right)$, and thus, by independence, the probability that at least $(\beta_A-\beta_I)k$ parasites do not reach the boundary is of order $\mathcal{O}\left(\frac{k}{\sqrt{t}^{(\beta_A-\beta_I)k}}\right)$. Using $k\ll \sqrt{t}$, this expression will be maximized in $k= k_0$, and choosing $k_0 =\left\lceil\frac{\alpha+1}{\beta_A-\beta_I}\right\rceil$ yields the desired tail. For $\{K < k_0\}$ we can use the strong condition on the offspring and immunity configurations in the definition of $K_n$ to obtain that there are at least $Km_A$ parasites, and we wait for $Km_I$ many to reach a one-sided boundary at a distance at most $k_0-1$. By definition, we have $Km_A-Km_I \ge 4$, which, analogously as above, gives that the probability of less than $Km_I$ parasites having reached the boundary is bounded by $\mathcal{O}\left(\frac{1}{\sqrt{t}^{4+1}}\right)$.\\
Next, we deal with the sequences $(\overline{\nu}_n^i)_{n\ge 1}$ and $(\nu_{M+n})_{n\ge1}$. By Lemma \ref{lem:M_tail}, each sequence $(\overline{\nu}^i_n)_{n\ge 1}$ has the same distribution, and thus it suffices to analyze the case $i=0$, and we abbreviate $\overline{\nu}_n := \overline{\nu}^0_n$ for any $n\ge1$. Also, we only perform the arguments for $(\overline{\nu}_n)_{n\ge1}$ and note that similar results, with identical proofs, hold for $(\nu_{M+n})_{n\ge k_0}$. \\
By definition, the offspring and immunity configurations above $M^0$ are always good in the sense that $G((A_{M^0},\dots,A_{M^0+n-1}),(I_{M^0+1},\dots,I_{M^0+n})) = 1$. Clearly, this restriction affects the distribution of $\overline{\nu}_n$, and therefore the distribution of $\overline{\nu}_n$ differs from that of $\nu_n$. However, this difference diminishes fast as $n$ tends to infinity. A key observation is that the condition of having only good configurations above $M^0$ is, in fact, only a restriction on the first few sites after $M^0$. For large $n$, a bad configuration between $M^0$ and $M^0 + n$ would be a large deviation for $(A_{M^0},\dots,A_{M^0+n-1})$ or $(I_{M^0+1},\dots,I_{M^0+n})$, which is unlikely to happen. Because $\overline{\nu}_n$ only depends on variables with an index greater than $n-K_n$, this means that the distribution of $\overline{\nu}_n$ is only influenced by the restriction that all configurations above $M^0$ are good, by ruling out the anyway unlikely event of a large $K_n$. We thus obtain the following lemma.
\theoremstyle{plain}\begin{lemma}\label{Tdistconv}
The joint distributions of $(\{\overline{\nu}_n,\overline{\nu}_{n+1},\dots\})_{n\ge1}$ under $\mathbb{P}$ converge to the joint distribution of $\{\nu_1,\nu_2,\dots\}$ under $\mathbf{P}$. In particular, for any $\varepsilon>0$ there is a $C\ge 0$ such that
 \[
 \vert\mathbb{P}((\overline{\nu}_n,\overline{\nu}_{n+1},\dots) \in E)-\mathbf{P}((\nu_1,\nu_2,\dots)\in E)\vert \le Cn^{-\alpha+\varepsilon}
 \]
for all $n\ge k_0,E\in\mathcal{B}([0,\infty)^\N)$.
\end{lemma}
The proof will be given in Section \ref{pro:Tdistconv}, and we note that the same result holds for $(\nu_{M+n})_{n\ge k_0}$ with the same proof. The upper bound is derived from estimating the probability of a large $K_{n+j} > n+j$ for some $j\ge 0$ with a union bound. \\
Finally, we need to establish a weak dependence among the random variables $(\overline{\nu}_n)_{n\ge1}$. As we will see, they form a $\phi$-mixing sequence. This is not surprising because for $\overline{\nu}_{n+k}$ to be influenced by $\overline{\nu}_n$, we must have $K_{n+k}>k$, since otherwise the two variables use different parts of $\mathbf{A},\mathbf{I},\mathbf{Y}$. However, because the variables are defined above $M^0$, the immunities and offspring numbers used are not i.i.d.~Similarly as in Lemma \ref{Tdistconv}, we show that this restriction also diminishes for large $k$, which yields the following lemma.
\theoremstyle{plain}\begin{lemma}\label{conditinallyphimixing}
For any $\varepsilon \in (0,\alpha)$ the sequence $(\overline{\nu}_n)_{n\ge 1}$ is $\phi$-mixing under $\mathbb{P}$ with rate
 \[
 \sup_{\substack{E\in\mathcal{F}_{\le m},\mathbb{P}(E)>0,\\B\in\mathcal{F}_{\ge m+n}}} \vert \mathbb{P}(B\vert E)-\mathbb{P}(B)\vert \le \phi(n):= C n^{-\alpha+\varepsilon},
 \]
 where
 \[
 \mathcal{F}_{\le m} := \sigma(\overline{\nu}_j:1\le j\le m),~~\mathcal{F}_{\ge m} := \sigma(\overline{\nu}_j:j\ge m).
 \]
\end{lemma}
The proof will be given in Section \ref{pro:conditinallyphimixing}. There are two effects that make the events $B$ and $E$ depend on each other. The first is that every configuration of offspring numbers and immunities to the right of $M^i$ is good. The second effect is that $B\in\mathcal{F}_{\ge m+n}$ depends on all random variables from $\mathbf{A},\mathbf{I},\mathbf{Y}$ with index less than $m+n$, in particular the same variables on which $E\in\mathcal{F}_{\le m}$ depends. To deal with the first effect, for some large $k=\Theta(n)$ we split the event of having only good configurations to the right of $M^i$ into the event $\mathcal{G}_{0,m+k}$ that there are only good configurations between $M^i$ and site $M^i+m+k$ and the complement in $\mathcal{G}_{0,m+k}$ of the event $\mathcal{B}_{0,m+k}$ that for some $j > k$ there is a bad configuration between $M^i$ and $M^i + m+ j$. Using a union bound, the event $\mathcal{B}_{0,m+k}$ happens with probability $\mathcal{O}(k^{-\alpha + \varepsilon})$. To deal with the second effect, we split the event $B$ into $B$ intersected with the event that $B$ is using variables with index less than $m+k$ and $B$ intersected with the event that $B$ uses only variables with index greater than $m+k$. Again using a union bound, the former event happens with probability $\mathcal{O}((n-k)^{-\alpha+\varepsilon})$. Now the remaining event $B$ intersected with the event that $B$ uses only variables with index greater than $m+k$ and the event $\mathcal{G}_{0,m+k}\cap E$ are independent, which will make the probability of their intersection close to $\mathbb{P}(B)\mathbb{P}(E)$.\\
Given that the sequence $(\overline{\nu}_n)_{n\ge1}$ is asymptotically stationary and has weak dependence, we can obtain a strong law of large numbers for it. To better control the fluctuations around the mean, we then use classical results on the tails of sums of $\phi$-mixing sequences to obtain the following estimate.
\theoremstyle{plain}\begin{lemma}\label{Qconvspeed}
 If $\alpha >3$, then for all $2\le q < \frac{(4\wedge\alpha)+1}{2}$, there exists a $C_q>0$ such that for all $n\in\N$ we have
 \[
\mathbb{E}\left[\max_{1\le i\le n} \left\vert\sum_{k=1}^i {\left(\overline{\nu}_k-\mathbb{E}[\overline{\nu}_k]\right)} \right\vert^q\right]\le C_qn^{\frac{q}{2}}.
\]
\end{lemma}
The proof will be given in Section \ref{pro:Qconvspeed}.\\
We will make use of this result twofold. The first application is to control the time the front needs to travel from one good site to the next and show it has a finite expectation. The other application will be to show that with positive probability the front moves at linear speed after a good site, which we will use to construct the almost surely finite renewal sites $(R^i)_{i\ge0}$ in Proposition \ref{prop:renewal_sites}. Coming to the first application, we have the following.
\theoremstyle{plain}\begin{lemma}\label{lem:T_Tfinexp}
Let $T^1$ be as in Definition \ref{Msites}. If $\alpha > \frac{3+\sqrt{11}}{2}$, then
 \[
 \mathbb{E}[T^1\vert M^0 = 0]<\infty.
 \]
\end{lemma}
The proof will be given in Section \ref{pro:T_Tfinexp} and relies on controlling large deviations of the sum of $(\overline{\nu}_k)_{k\ge1}$ by using Lemma \ref{Qconvspeed}. \\
Furthermore, after a good site $M^i$ is reached, we will show that with positive probability the front grows linearly fast at a certain small slope $\lambda > 0$ after a good site $M^i$ is reached. At each good site $M^i$, the parasites that were generated to the left of $M^i$ have a positive probability of never catching up to the linearly moving front. Iterating this procedure yields an almost surely finite renewal site $R^1 = M^{J^1}$, such that the front after reaching this site is only fed from parasites that are generated at a site $x$ with $x\ge R^1$. Again, iterating the construction of $R^1$ will lead us to the following renewal structure that splits the jump times into independent epochs between these renewal sites.
\begin{proposition}\label{prop:renewal_sites}
 Let $\alpha > 3$, $\mathbb{E}[I^{2\alpha}]<\infty,\mathbb{E}[A^{\frac{4}{4\wedge\alpha - 3} + \varepsilon_A}]<\infty$ for some $\varepsilon_A > 0$ and assume the initial configuration of only $A_0$ parasites placed at vertex $0$.\\
 There is a sequence $(R^i)_{i\ge0}$ of $\mathbb{P}$-almost surely finite renewal sites, such that $R^0 = 0$ almost surely and
 \[
 \{(R^{i+1}-R^i,\rho_{R^i+1}-\rho_{R^i},\dots, \rho_{R^{i+1}}-\rho_{R^{i+1}-1}): i\ge 0\}
 \]
 is an independent collection. The collection starting with $i\ge1$ is also identically distributed.
 \end{proposition}
 The proof will be given in Section \ref{pro:rho_alphamix} and will make the construction described above rigorous. \\
 In the proof of Theorem \ref{Theorem: lln}, we will make use of this renewal structure to obtain a zero-one law under $\mathbb{P}$ for the limit $\frac{r_t}{t}$, which will allow us to transfer a convergence with positive probability (on the event $\{M^0 = 0\}$) obtained through Liggett's Subadditive Ergodic Theorem \cite{L1985} and Lemma \ref{lem:T_Tfinexp} to a $\mathbb{P}$-almost sure convergence.
\subsection{Proofs of auxiliary results}\label{sec:proofsoflemmata}
\subsubsection{Proof of Lemma \ref{lem:Kmtail}}\label{pro:Kmtail}
We begin by establishing the following result on the asymptotic for coefficients of a power series. This will be useful because the probabilities we are interested in are given by such series; see Corollary \ref{cor:N_ntail} (cf. \cite[Theorem 1]{A1954}). \theoremstyle{plain}\begin{lemma}\label{lem:exppowasym}Let $(w_n)_{n\in\N}\in\mathcal{O}(\e^{-\theta n}n^{-\gamma})$ for some $\theta \ge 0,\gamma>1$. Then, defining $(c_n)_{n\in\N_0}$ as the coefficients of the power series \[ \sum_{n=0}^\infty c_n s^n := \exp\left(\sum_{n=1}^\infty w_ns^n\right)~~~~(\vert s\vert \le 1), \] we have $(c_n)_{n\in\N_0}\in\mathcal{O}(\e^{-\theta n}n^{-\gamma})$.\end{lemma}\begin{proof}Let $C_w\ge 0$ be such that $\vert w_n\vert \le C_w\e^{-\theta n}n^{-\gamma}$ for all $n\in\N$. We begin by investigating the coefficients of the power series defined by\[\sum_{n=1}^\infty b_n(k) s^n := \left(\sum_{n=1}^\infty w_ns^n\right)^k~~~~(\vert s\vert \le 1),\]and show by induction over $k\in\N$ that, setting $C_1 := 2^{1+\gamma}\sum_{j = 1}^{\infty} j^{-\gamma}$, we have \begin{equation}\label{eq:expowasympolIH} \vert b_n(k)\vert \le C_w^k C_1^{k-1}\e^{-\theta n}n^{-\gamma},~~~\text{ for all }n\in\N.\end{equation} The induction root $k=1$ holds by assumption, so we continue with the induction step. Let $k\ge 1$ and assume \eqref{eq:expowasympolIH} holds; then for $\vert s\vert \le 1$ we have\[\begin{split}\sum_{n=1}^\infty b_n(k+1)s^n &= \left(\sum_{n=1}^\infty w_n s^n\right)^k\left(\sum_{n=1}^\infty w_n s^n\right) = \left(\sum_{n=1}^\infty b_n(k) s^n\right)\left(\sum_{n=1}^\infty w_n s^n\right) \\&= \sum_{n=0}^\infty\sum_{j=0}^n w_{j+1}s^{j+1}b_{n-j+1}(k)s^{n-j+1} = \sum_{n=0}^\infty s^{n+2}\sum_{j=1}^{n+1}w_jb_{n-j+2}(k) \\ &= \sum_{n=2}^\infty s^n\sum_{j=1}^{n-1}w_jb_{n-j}(k).\end{split}\]Now for $n\ge 2$ we obtain
\begin{equation}\label{eq:replicatingasympol}
\begin{split}
\left\vert\sum_{j=1}^{n-1}w_jb_{n-j}(k)\right\vert &\le C_wC_w^kC_1^{k-1}\sum_{j=1}^{n-1}\e^{-\theta j}j^{-\gamma}\e^{-\theta(n-j)}(n-j)^{-\gamma} \\
&\le 2C_w^{k+1}C_1^{k-1}\e^{-\theta n}\sum_{j = 1}^{\left\lfloor\frac{n}{2}\right\rfloor} (j(n-j))^{-\gamma} \\
&=2C_w^{k+1}C_1^{k-1}\e^{-\theta n}n^{-\gamma}\sum_{j = 1}^{\left\lfloor\frac{n}{2}\right\rfloor} \left(j\left(1-\frac{j}{n}\right)\right)^{-\gamma} \\
&\le 2C_w^{k+1}C_1^{k-1}\e^{-\theta n}n^{-\gamma}\sum_{j = 1}^{\left\lfloor\frac{n}{2}\right\rfloor} \left(j\cdot\frac{1}{2}\right)^{-\gamma} \\
&\le2C_w^{k+1}C_1^{k-1}\left( 2^{\gamma}\sum_{j = 1}^{\infty} j^{-\gamma}\right) \e^{-\theta n}n^{-\gamma} = C_w^{k+1}C_1^k \e^{-\theta n}n^{-\gamma}.
\end{split}
\end{equation}This finishes the induction and shows \eqref{eq:expowasympolIH} holds for all $k\in\N$. Noting that
\[
\begin{split}
\exp\left(\sum_{n=1}^\infty w_ns^n\right) &= 1+\sum_{k=1}^\infty\frac{1}{k!}\left(\sum_{n=1}^\infty w_ns^n\right)^k = 1+\sum_{k=1}^\infty\sum_{n=1}^\infty \frac{b_n(k)}{k!} s^n \\
&= 1+\sum_{n=1}^\infty s^n\sum_{k=1}^\infty \frac{b_n(k)}{k!} 
\end{split}
\]
we estimate
\[
\left\vert \sum_{k=1}^\infty \frac{b_n(k)}{k!}\right\vert \le \e^{-\theta n}n^{-\gamma}\sum_{k=1}^\infty \frac{1}{k!}C_w^kC_1^{k-1} =\frac{\e^{C_wC_1} - 1}{C_1} \e^{-\theta n}n^{-\gamma} =\mathcal{O}(\e^{-\theta n}n^{-\gamma}),
\]
which finishes the proof.
\end{proof}
\theoremstyle{definition}\begin{definition}
We adopt the notation of \cite{A1954} and, for $n\in\N$, denote by $N_n^A$ the amount of partial sums \[ \sum_{k=1}^1(\beta_A-A_k),\sum_{k=1}^2 (\beta_A-A_k),\dots,\sum_{k=1}^n (\beta_A-A_k) \]that are positive, and similarly we denote by $N_n^I$ the amount of partial sums 
\[ \sum_{k=1}^1 (I_k-\beta_I ),\dots,\sum_{k=1}^n( I_k-\beta_I ) \] 
that are positive.
\end{definition}
\theoremstyle{plain}\begin{corollary}\label{cor:N_ntail}
There is some $\theta > 0$ such that 
\[ \mathbf{P}(N_n^I = n) = \mathcal{O}(n^{-(\alpha+1)}),~~~\mathbf{P}(N_n^A = n) = \mathcal{O}\left(\frac{\exp(-\theta n)}{n^{\frac{3}{2}}}\right) \]
\end{corollary}
\begin{proof} As shown in \cite[Theorem 1]{A1954} we have \[ \sum_{n=0}^\infty \mathbf{P}(N_n^A = n)s^n = \exp\left(\sum_{n=1}^\infty\frac{\mathbf{P}\left(\sum_{j=1}^n A_j < \beta_A n\right)}{n}s^n\right), \]Noting that $A\ge 0$ almost surely allows us to apply \cite[Theorem 1]{B1960} with \[X_j = -A_j+\mu_A, \quad \quad a = \mu_A-\beta_A > 0\] and conclude \[\mathbf{P}\left(\sum_{j=1}^n A_j < \beta_A n\right) = \mathcal{O}\left(\frac{\exp(-\theta n)}{n^{\frac{1}{2}}}\right)\]for some $\theta > 0$. If $\mathbf{P}(A < \beta_A) =\mathbf{P}(X_j > a) > 0$, the assumptions of \cite[Theorem 1]{B1960} are satisfied as explained before the theorem there, and if $\mathbf{P}(A < \beta_A) = 0$, then the left-hand side is just $0$ and thus trivially the estimation holds. Now applying Lemma \ref{lem:exppowasym} yields the second claim. The first claim follows analogously from Lemma \ref{lem:exppowasym} and an application of Markov's as well as Rosenthal's inequality. Since, by assumption \eqref{momentcondition}, $\mathbf{E}[I^{2\alpha}]<\infty$, we have 
\[
\begin{split}
\mathbf{P}\left(\sum_{k=1}^n I_k>\beta_In\right) &= \mathbf{P}\left(\sum_{k=1}^n I_k-\mu_I>(\beta_I-\mu_I)n\right) \le \frac{\mathbf{E}\left[\left\vert\sum_{k=1}^n I_k-\mu_I\right\vert^{2\alpha}\right]}{(\beta_I-\mu_I)^{2\alpha}n^{2\alpha}} \\
&\le \frac{C_{2\alpha}\mathbf{E}[\vert I_1-\mu_I\vert^2]^{\alpha}n^{\alpha}}{(\beta_I-\mu_I)^{2\alpha}n^{2\alpha}} =\frac{C_{2\alpha}\mathbf{E}[\vert I_1-\mu_I\vert^2]^{\alpha}}{(\beta_I-\mu_I)^{2\alpha}} n^{-\alpha}.
\end{split}
\]
\end{proof}
Now we have all the ingredients to estimate the tail of $K$.
\begin{proof}[Proof of Lemma \ref{lem:Kmtail}]First we observe that for $n \gg k_0$ large enough\[\begin{split}\mathbf{P}(K > n) &\le \mathbf{P}\left(\bigcap_{k=k_0}^n\left(\left\{\sum_{j=1}^k A_{1-j} < \beta_A k\right\}\cup\left\{\sum_{j=1}^kI_{2-j} > \beta_I k\right\}\right) \right) \\&\le \mathbf{P}(N_n^A+N_n^I \ge n-k_0+1) \\&\le \mathbf{P}(N_n^A \ge \left\lceil\lambda\log n\right\rceil) + \mathbf{P}(N_n^I > n-k_0-\left\lceil\lambda\log n\right\rceil)\end{split}\]for any $\lambda > 0$. Now using \cite[Theorem 1, (3.3)]{A1954} and then Corollary \ref{cor:N_ntail}, we obtain for any $1\le k \le n$
\[\mathbf{P}(N_n^A = k) \le \mathbf{P}(N_{k}^A = k) =\mathcal{O}\left(\e^{-\theta k}k^{-\frac{3}{2}}\right)\] 
and
\[\mathbf{P}(N_n^I = k) \le \mathbf{P}(N_{k}^I = k) =\mathcal{O}\left(k^{-(\alpha +1)}\right).\]
Setting $\lambda = \frac{\alpha+1}{\theta}$, we have
\[
\begin{split}
\mathbf{P}(N_n^A \ge \left\lceil\lambda\log n\right\rceil) &= \sum_{k = \left\lceil\lambda\log n\right\rceil}^n \mathbf{P}(N_n^A = k) \le \sum_{k = \left\lceil\lambda\log n\right\rceil}^n C_A\e^{-\theta k}k^{-\frac{3}{2}} \\
&\le \sum_{k = \left\lceil\lambda\log n\right\rceil}^n C_A\e^{-\theta k} = C_A \frac{\e^{-\theta\left\lceil\lambda\log n\right\rceil} - \e^{-\theta(n+1)}}{1-\e^{-\theta}} \\
&\le \frac{C_A}{1-\e^{-\theta}}\e^{-\theta \frac{\alpha+1}{\theta}\log n} = \mathcal{O}\left(n^{-(\alpha+1)}\right)
\end{split}
\]
and
\[
\begin{split}
\mathbf{P}(N_n^I > n-k_0-\left\lceil\lambda\log n\right\rceil) &= \sum_{k=n-k_0-\left\lceil\lambda\log n\right\rceil +1}^{n}\mathbf{P}(N_n^I = k) \\
&\le \sum_{k=n-k_0-\left\lceil\lambda\log n\right\rceil +1}^{n}C_I k^{-(\alpha+1)} \\ 
&\le C_I(k_0+\left\lceil\lambda\log n\right\rceil)(n-k_0-\left\lceil\lambda\log n\right\rceil +1)^{-(\alpha+1)} \\ 
&\le 2C_I\lambda n^{-(\alpha+1)}\frac{\log n}{(1-\frac{2\lambda\log n}{n})^{\alpha+1}} \\
&=\mathcal{O}\left(n^{-(\alpha+1)}\log n\right),
\end{split}
\]using that $2\lambda \log n > k_0+\left\lceil\lambda \log n\right\rceil$ for $n$ large enough. In particular, for any $\varepsilon > 0$ we have
\[\begin{split}
\mathbf{P}(K > n) &\le \mathbf{P}(N_n^A \ge \left\lceil\lambda\log n\right\rceil) + \mathbf{P}(N_n^I > n-k_0-\left\lceil\lambda\log n\right\rceil) \\
&= \mathcal{O}\left(n^{-(\alpha+1)}\right) +\mathcal{O}\left(n^{-(\alpha+1)}\log n\right) = \mathcal{O}\left(n^{-(\alpha+1)+\varepsilon}\right)
\end{split}\]
which finishes the proof.\end{proof}\begin{remark} As seen in the proof, we actually obtain the better estimate \[ \mathbf{P}(K> n) = \mathcal{O}\left(n^{-(\alpha+1)}\log n\right), \] but in the proofs of Lemma \ref{lem:M_tail} and Lemma \ref{NuMoments} we will lose an $n^\varepsilon$ factor anyway. \end{remark}
\subsubsection{Proof of Lemma \ref{lem:M_tail}}\label{pro:M_tail}
\begin{proof}[Proof of Lemma \ref{lem:M_tail}]Clearly \[ \begin{split} \mathbb{P}(M^{i+1}-M^i \ge l) &= \frac{1}{\mathbf{P}(\mathcal{S})}\sum_{j=0}^\infty \mathbf{P}(M^i = j,\mathcal{S},M^{i+1} \ge j+l) \end{split} \] The event $\{M^i = j,\mathcal{S},M^{i+1}\ge j+l\}$ is given by the intersection of an event $B^{<j}$ depending on $(A_0,\dots,A_{j-1})$ and $(I_1,\dots, I_j)$ defined through the event $\mathcal{S}$ and the event $\{M^i\ge j\}$, as well as the intersection of the event \[ \mathcal{G}_j = \bigcap_{m=1}^\infty\{G((A_j,\dots, A_{j+m-1}),(I_{j+1},\dots,I_{j+m})) = 1\} \] 
as also $\{M^i=j\}$ should be fulfilled and finally the intersection of some event $B_{>j}$ depending on $A_{j+1},\dots$ and on $I_{j+2},\dots$ defined through the event $\{M^{i+1} \ge j+l\}$. Precisely, $B_{>j}$ is the event that for some \[k\ge0 \text{ and }l_0 = 0 < l_1 < \dots < l_k\] such that $l_k \ge l-2> l_{k-1}$ we have 
\[ 
\begin{split} 
&\left\{\begin{matrix}\text{all configurations from }j+1+l_0\text{ up to }j+1+l_1-1\text{ are good }\\\text{and the configuration from } j+1+l_0 \text{ to } j+1+l_1 \text{ is bad,}\\\vdots\\\text{all configurations from }j+2+l_{k-1}\text{ up to }j+1+l_k-1\text{ are good }\\\text{and the configuration from } j+2+l_{k-1} \text{ to } j+1+l_k \text{ is bad.}\end{matrix}\right\}. 
\end{split} 
\] 
Using that $\mathbf{A},\mathbf{I}$ are i.i.d.~under $\mathbf{P}$, we obtain that \[ \begin{split} \mathbb{P}(M^{i+1}-M^i \ge l) &= \frac{1}{\mathbf{P}(\mathcal{S})}\sum_{j=0}^\infty \mathbf{P}(B^{<j})\mathbf{P}(\mathcal{G}_j,B_{> j}) \\ &= \mathbf{P}(\mathcal{G}_0,B_{> 0}) \frac{1}{\mathbf{P}(\mathcal{S})}\sum_{j=0}^\infty \mathbf{P}(B^{<j}) \\ &= \mathbf{P}(\mathcal{G}_0,B_{> 0}) \frac{1}{\mathbf{P}(\mathcal{S})}\sum_{j=0}^\infty \mathbf{P}(B^{<j})\frac{\mathbf{P}(\mathcal{G}_j)}{\mathbf{P}(\mathcal{G}_j)} \\ &= \mathbf{P}(\mathcal{G}_0,B_{> 0}) \frac{1}{\mathbf{P}(\mathcal{S})\mathbf{P}(\mathcal{G}_0)}\sum_{j=0}^\infty \mathbf{P}(B^{<j},\mathcal{G}_j) \\ &= \mathbf{P}(\mathcal{G}_0,B_{> 0}) \frac{1}{\mathbf{P}(\mathcal{S})\mathbf{P}(\mathcal{G}_0)}\sum_{j=0}^\infty \mathbf{P}(M^i = j,\mathcal{S}) \\& = \frac{\mathbf{P}(\mathcal{G}_0,B_{> 0})}{\mathbf{P}(\mathcal{G}_0)} = \mathbf{P}(B_{>0}\vert \mathcal{G}_0). \end{split} \] This yields that $(M^{i+1}-M^i)_{i\ge0}$ are identically distributed. To see that $\mathbb{P}(M^{i+1}- M^i\geq l)$ is dominated by $\mathbb{P}(M^0\geq l)$, note that the event $B_{>0}$ (which we just defined) is more likely to happen conditioned on survival, i.e., on the event \[\bigcap_{m=1}^\infty\left\{\sum_{k=1}^mA_{k-1}-I_k \ge 0\right\}, \] rather than conditioned on $\mathcal{G}_0$, i.e., \[ \mathbb{P}(M^0 \ge l) = \mathbf{P}(B_{>0}\vert \mathcal{S}) \ge \mathbf{P}(B_{>0}\vert \mathcal{G}_0) = \mathbb{P}(M^{i+1}-M^i \ge l). \]The independence follows similarly after noting that \begin{equation}\label{Godddecouple} \mathcal{G}_j\cap\mathcal{G}_{j+k} = \mathcal{G}_{j+k}\cap\bigcap_{m=1}^{k}\{G((A_j,\dots, A_{j+m-1}),(I_{j+1},\dots,I_{j+m})) = 1\}, \end{equation} and the two events $\mathcal{G}_{j+k}$ and $\cap_{m=1}^{k}\{G((A_j,\dots, A_{j+m-1}),(I_{j+1},\dots,I_{j+m})) = 1\}$ ) on the right-hand side are independent.\\To obtain the tail estimate, we proceed as follows. We set $\delta_1 := \mathbf{P}(\mathcal{G}_0)$. Then for any $m\in\N$ and $l\ge 0$ we have\[\begin{split} \mathbb{P}(M^0>l) &= \mathbb{P}(N^0 > m,M^0>l) + \sum_{n=1}^m\mathbb{P}(N^0=n,L^0_n> l) \\&\le \mathbf{P}(N^0 > m) + \sum_{n=1}^m\mathbb{P}(l<L_n^0<\infty).\end{split}\]As seen in the proof of Lemma \ref{lem:welldef}, $N^0$ is geometrically distributed under $\mathbf{P}$, and hence the first part is equal to $(1-\delta_1)^m$. For $\lambda\in(0,1)$ and $n\in \{1,\dots,m\}$ we define the event\[E^\lambda_n := \{L_1^0-L_0^0<l^\lambda,\dots,L_{n}^0-L_{n-1}^0<l^\lambda\}\]such that on $E^\lambda_n$ we have $L_{n}^0 < nl^\lambda$. For $l \ge m^{\frac{1}{1-\lambda}}$ we thus have\[E^\lambda_n\cap\{l < L_n^0<\infty\} = \emptyset.\]Lemma \ref{lem:Kmtail} gives us that, fixing some $\varepsilon_K\in(0,\varepsilon)$, there is a $C_1= C_1(\alpha,\varepsilon_K)\ge 0$ with\[\begin{split} \mathbb{P}(l<L_n^0<\infty) &=\mathbb{P}((E_n^\lambda)^C,l<L_{n}^0<\infty) \le \sum_{k=1}^{n}\mathbb{P}(L_k^0-L_{k-1}^0\ge l^\lambda,l<L_{n}^0<\infty) \\ &\le \sum_{k=1}^{n}\mathbb{P}(l^\lambda\le L_k^0-L_{k-1}^0<\infty) = n\mathbb{P}(l^\lambda \le L_1^0 < \infty) \\ &\le n\mathbb{P}\left(\bigcup_{j = \lfloor l^\lambda\rfloor}\{K_{j} > j\}\right) \le C_1nl^{-\lambda(\alpha-\varepsilon_K)}.\end{split}\]Now we choose $\lambda \in (\frac{\alpha-\varepsilon}{\alpha-\varepsilon_K},1), C_2 = \frac{\lambda(\alpha-\varepsilon_K)}{\log\frac{1}{1-\delta_1}}$ and $m =\left\lceil C_2\log l \right\rceil$, then for $l_0$ large enough such that \[l_0 > (1+C_2\log l_0)^{\frac{1}{1-\lambda}}\] there is a $C_3>0$ such that for all $l\ge l_0$ we obtain\[\mathbb{P}(M^0>l) \le l^{-\lambda(\alpha-\varepsilon_K)}+C_1\frac{2+3C_2\log l+(C_2\log l)^2}{2}l^{-\lambda(\alpha-\varepsilon_K)} \le C_3 l^{-\alpha+\varepsilon}. \]Possibly increasing the constant $C_3$, this estimate also holds for any $0\le l < l_0$, which finishes the proof.\end{proof}

\subsubsection{Proof of Lemma \ref{NuMoments}}\label{pro:NuMoments}
Now having a good control on the tail of $K$, we can use it to control the tail of $\nu$. We split the event $\nu>t$ into the parts where $K \le \lambda(t)$ and where $K>\lambda(t)$, for some suitable function $\lambda(t)\ll\sqrt{t}$; see \eqref{gamma} for a definition. We treat the first part by applying large deviation results on random walks and the second part by ignoring the random walk part and just estimating $K>\lambda(t)$ with Lemma \ref{lem:Kmtail}.\begin{proof}[Proof of Lemma \ref{NuMoments}]We recall the notations \[m_I,m_A,\underline{A}^\text{good}(k),\underline{I}^\text{good}(k), \underline{\mathcal{V}}_n^\text{good}(k,\mathbf{a}) \text{ and }\tau_k^{x,i}\] from Definition \ref{auxiliaryfront} and define for \[k\in\N,\mathbf{a}=(a_1,\dots,a_k)\in\underline{A}^\text{good}(k)\]the quantity\[Z(k) := \begin{cases}m_I k,&\text{if }1\le k < k_0\\\beta_I k,&\text{if }k\ge k_0\end{cases}\]and random time\[\mu_n(k,\mathbf{a}) := \inf\left\{t\ge0:\sum_{(x,i)\in\underline{\mathcal{V}}^\text{good}_n(k,\mathbf{a})}\mathds{1}_{\tau^{x,i}_{n-x}\le t} \ge Z(k)\right\}.\]Adopting the notation $\mu(k,\mathbf{a}) := \mu_1(k,\mathbf{a})$, we observe\[\nu = \nu_1 = \sum_{k=1}^\infty\sum_{\mathbf{a}\in\underline{A}^\text{good}(k)}\mu(k,\mathbf{a})\mathds{1}_{K=k}\prod_{j=1}^k\mathds{1}_{A_{1-j}=a_j}.\]Since $\mu(k,\mathbf{a})$ only depends on $\mathbf{Y}$, we obtain
\begin{equation}\label{TKindep}
\begin{split}
\mathbf{P}&\left(K=k,\bigcap_{j=1}^k\{A_{1-j}=a_j\},\nu\in B\right) \\&= \mathbf{P}\left(K=k,\bigcap_{j=1}^k\{A_{1-j}=a_j\}\right)\mathbf{P}(\mu(k,\mathbf{a})\in B)
\end{split}
\end{equation}
for all $k\ge 1,\mathbf{a}\in\underline{A}^\text{good}(k),B\in\mathcal{B}([0,\infty))$. In particular, setting 
\begin{align}\label{gamma}
\lambda(t) := \lfloor t^{\frac12-\varepsilon_1}\rfloor
\end{align} for some $\varepsilon_1>0$ to be fixed later, we can bound $\mathbf{P}(\nu >t)$ by
\begin{equation}\label{eq:tailsplit}\begin{split} &\mathbf{P}(\nu> t, K> \lambda(t))+ \sum_{k=1}^{\lambda(t)}\sum_{\mathbf{a}\in\underline{A}^\text{good}(k)}\mathbf{P}\left(\nu>t,K = k,\bigcap_{j=1}^k\{A_{1-j}=a_j\}\right) \\&\le C_K \lambda(t)^{-(\alpha+1-\varepsilon_K)}+ \sum_{k=1}^{\lambda(t)}\sum_{\mathbf{a}\in\underline{A}^\text{good}(k)}\mathbf{P}(\mu(k,\mathbf{a})>t)\mathbf{P}\left(K = k,\bigcap_{j=1}^k\{A_{1-j}=a_j\}\right) \\&\le C_K \lambda(t)^{-(\alpha+1-\varepsilon_K)}+ \max_{\substack{1\le k \le \lambda(t),\\\mathbf{a}\in\underline{A}^\text{good}(k)}}\mathbf{P}(\mu(k,\mathbf{a})>t)\end{split}
\end{equation}
for any $\varepsilon_K>0$ and $C_K=C_K(\varepsilon_{K})$ as in Lemma \ref{lem:Kmtail}. \\First we will show that for large enough $t$, the maximum over $k_0 \le k \le \lambda(t)$ is attained in $k_0$ and is in $\mathcal{O}\left(t^{-\frac{k_0(\beta_A-\beta_I)}{2}}\right)$. Then we will show that the maximum over $1\le k < k_0$ is in $\mathcal{O}\left(t^{-\frac{5}{2}}\right)$, which for large $t$ yields the desired tail behavior.\\ 
Clearly for any $k\ge k_0,\mathbf{a}\in\underline{A}^\text{good}(k)$ using a simple coupling argument, placing all walkers at the maximal distance shows 
\begin{equation}\label{eq:nutildetail1} \begin{split} \mathbf{P}(\mu(k,\mathbf{a})>t) &\le \mathbf{P}(\mu(k,(0,\dots,0,\lceil\beta_Ak\rceil))>t) \\ &=\mathbf{P}\left(\sum_{j=1}^{\lceil\beta_Ak\rceil}\mathds{1}_{\tau^{1-k,j}_k>t} \ge \lceil\beta_Ak\rceil-\lceil\beta_Ik\rceil + 1\right) \\ &= \mathbf{P}\left(\sum_{j=1}^{\lceil\beta_Ak\rceil}\mathds{1}_{\tau^{1-k,j}_k>t} \ge \frac{\lceil\beta_Ak\rceil-\lceil\beta_Ik\rceil + 1}{\lceil\beta_Ak\rceil}\lceil\beta_Ak\rceil\right) \end{split} \end{equation} 
The last line of \eqref{eq:nutildetail1} is just the probability of a binomial distribution with success probability 
\[
\pi_k(t) := \mathbf{P}(\tau^{0,1}_k > t) 
\]
and $\lceil\beta_Ak\rceil$ trials to be larger than $b_k\lceil\beta_Ak\rceil$, with 
\[
b_k := \frac{\lceil\beta_Ak\rceil-\lceil\beta_Ik\rceil + 1}{\lceil\beta_Ak\rceil}.
\]
Thus we can apply \cite[Eq. (5)]{B1964}, using as their parameters $n,r,k$ the values $n=b_k\lceil\beta_Ak\rceil,r=\frac{1}{b_k}, k = 0$ and noting that the error term $R_0$ is non-positive, to bound the last line of \eqref{eq:nutildetail1} by
\begin{align}\label{Esti_pi_k_2}\binom{\lceil\beta_Ak\rceil}{b_k\lceil\beta_Ak\rceil}\pi_k(t)^{b_k\lceil\beta_Ak\rceil}(1-\pi_k(t))^{\lceil\beta_Ak\rceil-b_k\lceil\beta_Ak\rceil}\left(\frac{1-\pi_k(t)}{1-\frac{\pi_k(t)}{b_k}}\right).\end{align} 
Next, we estimate the success probability using \cite[Theorem A.1]{RS2004} and obtain a constant $C_0 \ge 1$ such that for any $t\ge 0$ and $k\le t^{\frac12-\varepsilon_1}$ we have
\begin{align}\label{EstimatePi_k} 
\pi_k(t) \leq C_0 \frac{k}{\sqrt{t}}. 
\end{align}
Taking $t\ge t_0 := \left(\frac{2 C_0 \beta_A}{\beta_A-\beta_I}\right)^{\frac{1}{\varepsilon_1}}$ and any $k\le t^{\frac12-\varepsilon_1}$ we arrive at 
\begin{align}\label{Esti_pi_k_1}
\pi_k(t) \le C_0t^{-\varepsilon_1} \le \frac{\beta_A-\beta_I}{2\beta_A} < b_k. \end{align}
To further bound \eqref{Esti_pi_k_2}, we estimate the binomial coefficient by the central coefficient and Stirling's formula, which gives 
\begin{equation}\label{Esti_binom_cent} \begin{split} \binom{\lceil\beta_Ak\rceil}{b_k\lceil\beta_Ak\rceil} &\le \binom{2\left\lceil\frac{\lceil\beta_Ak\rceil}{2}\right\rceil}{\left\lceil\frac{\lceil\beta_Ak\rceil}{2}\right\rceil} \le \frac{\e^{\frac{1}{12\cdot2\left\lceil\frac{\lceil\beta_Ak\rceil}{2}\right\rceil}}\sqrt{2\pi\left(2\left\lceil\frac{\lceil\beta_Ak\rceil}{2}\right\rceil\right)}\left(\frac{2\left\lceil\frac{\lceil\beta_Ak\rceil}{2}\right\rceil}{\e}\right)^{2\left\lceil\frac{\lceil\beta_Ak\rceil}{2}\right\rceil}}{\left(\sqrt{2\pi\left\lceil\frac{\lceil\beta_Ak\rceil}{2}\right\rceil}\left(\frac{\left\lceil\frac{\lceil\beta_Ak\rceil}{2}\right\rceil}{\e}\right)^{\left\lceil\frac{\lceil\beta_Ak\rceil}{2}\right\rceil}\right)^2}\\&\le 4^{\beta_Ak}. \end{split} \end{equation} 
Next we estimate both $1-\pi_k(t)$ terms in \eqref{Esti_pi_k_2} by $1$ and note that by the choice of $t_0$, using \eqref{Esti_pi_k_1}, we can bound \begin{equation}\label{Esti_frac} \frac{1}{1-\frac{\pi_k(t)}{b_k}} \le \frac{1}{1-\frac12} =2. \end{equation} Noting that $b_k\lceil\beta_A k\rceil = \lceil\beta_Ak\rceil-\lceil\beta_Ik\rceil+1 \ge (\beta_A-\beta_I)k$, $t\ge t_0$, and combining the estimates \eqref{Esti_binom_cent},\eqref{Esti_frac} and \eqref{EstimatePi_k}, we can bound the term in \eqref{Esti_pi_k_2} by
\begin{equation}\label{eq:tomax}2\cdot4^{\beta_Ak}\left(C_0\frac{k}{\sqrt{t}}\right)^{(\beta_A-\beta_I)k}. \end{equation} 
We want to show that there is a $t_1\ge t_0$ such that for all $t\ge t_1$ the maximum of \eqref{eq:tomax} over $k_0\le k\le t^{\frac12-\varepsilon_1}$ is attained in $k_0$. To do so we treat $k$ as a continuous variable and take the derivative with respect to $k$ to show that it is negative. By a straightforward calculation we find that the derivative is given by 
\[ W(k,t,\beta_A,\beta_I)\cdot\left(C(\beta_A,\beta_I)+ (\beta_A-\beta_I)\log\left(\frac{k}{\sqrt{t}}\right)\right) \]
for some positive expressions $W$ and $C$. Explicitly, 
\[ W(k,t,\beta_A,\beta_I) = 2\cdot 4^{\beta_Ak}\left(C_0\frac{k}{\sqrt{t}}\right)^{\left(\beta_A-\beta_I\right)k} \] 
and 
\[ \begin{split} C(\beta_A,\beta_I) = &\left(\log\left(C_0\right)+1\right)(\beta_A-\beta_I)+2\beta_A\log(2). \end{split} \] 
This means that the sign of the derivative is determined by the term in the brackets and calculating 
\[ \begin{split} &C(\beta_A,\beta_I)+(\beta_A-\beta_I)\log\left(\frac{k}{\sqrt{t}}\right) \le 0 \\ \iff&(\beta_A-\beta_I)\log\left(\frac{k}{\sqrt{t}}\right) \le- C(\beta_A,\beta_I) \\ &\iff k \le \sqrt{t}\exp\left(-\frac{C(\beta_A,\beta_I)}{(\beta_A-\beta_I)}\right) \end{split} \]
yields that the derivative is negative if $k$ is small. In particular, the expression in \eqref{eq:tomax}, for fixed $t$, is decreasing in $k$ up to $\sqrt{t}\exp\left(-\frac{C(\beta_A,\beta_I)}{(\beta_A-\beta_I)}\right)$. Now taking 
\[ t_1 := \max\left\{t_0,\exp\left(\frac{C(\beta_A,\beta_I)}{\varepsilon_1(\beta_A-\beta_I)}\right)\right\} \]
we have $t^{\frac12-\varepsilon_1}\le \sqrt{t}\exp\left(-\frac{C(\beta_A-\beta_I)}{(\beta_A-\beta_I)}\right)$ for any $t\ge t_1$ and thus 
\[ \begin{split} \max_{\substack{k_0\le k\le \lambda(t),\\\mathbf{a}\in\underline{A}^\text{good}(k)}}\mathbf{P}(\mu(k,\mathbf{a})>t) &\le 2\cdot4^{\beta_Ak_0}\left(C_0 \frac{k_0}{\sqrt{t}}\right)^{(\beta_A-\beta_I)k_0} \\ & = C_1(k_0,\beta_A,\beta_I) t^{-\frac{(\beta_A-\beta_I)k_0}{2}} \end{split} \]
with $C_1=C_1(k_0,\beta_A,\beta_I):= 2\cdot4^{\beta_Ak_0}\left(C_0 k_0\right)^{(\beta_A-\beta_I)k_0}$. Next, we estimate that for $1\le k < k_0$ and $\mathbf{a}\in\underline{A}^{\text{good}}(k)$ we have at least $a_j \ge m_A\ge 4+m_I$ parasites on each site for $1\le j \le k$ and thus, again by a coupling argument, obtain \[ \begin{split} \mathbf{P}(\mu(k,\mathbf{a}) > t) &\le \mathbf{P}\left(\sum_{x=1}^k\sum_{i=1}^{m_A}\mathds{1}_{\tau^{x,i}_x \le t} < m_Ik\right) \\ &\le \mathbf{P}\left(\sum_{i=1}^{km_A}\mathds{1}_{\tau^{0,i}_{k}\le t} < km_I\right) \\ &= \sum_{j=0}^{km_I-1}\binom{km_A}{j}\mathbf{P}(\tau^{0,1}_k\le t)^j\mathbf{P}(\tau^{0,1}_k> t)^{km_A-j} \\ &\le C_2 t^{-\frac{k(m_A-m_I)+1}{2}} \le C_2t^{-\frac{5}{2}} \end{split} \] for some constant $C_2 > 0$ independent of $k,t$.\\ 
In particular, this implies that for large enough $t\ge t_1$ we have 
\[ \max_{\substack{1\le k\le \lambda(t),\\\mathbf{a}\in\underline{A}^\text{good}(k)}}\mathbf{P}(\mu(k,\mathbf{a})>t) \le \max\left\{C_1 t^{-\frac{(\beta_A-\beta_I)k_0}{2}},C_2t^{-\frac{5}{2}}\right\} \]
Choosing some $q< \frac{(4\wedge \alpha) +1}{2}$ and plugging this into estimate \eqref{eq:tailsplit}, fixing some $\varepsilon_{1} \in\left(0, \frac{1}{2}-\frac{q}{\alpha+1}\right)$ and $\varepsilon_K := \alpha+1-\frac{q}{\frac12-\varepsilon_1}$, we obtain a constant $C\ge0$ such that for large enough $t\ge t_1$ we have
\[
\begin{split}
\mathbf{P}(\nu > t) &\le C_Kt^{-(\frac12-\varepsilon_1)(\alpha+1-\varepsilon_K)}+ \max_{\substack{1\le k\le \lambda(t),\\\mathbf{a}\in\underline{A}^\text{good}(k)}}\mathbf{P}(\mu(k,\mathbf{a})>t) \\& \le C_Kt^{-q}+ \max\left\{C_1 t^{-\frac{(\beta_A-\beta_I)k_0}{2}},C_2t^{-\frac{5}{2}}\right\} \le C t^{-q}
\end{split}
\]where we used that $\varepsilon_{1},\varepsilon_K$ were exactly chosen in such a way that \[\left(\frac{1}{2}-\varepsilon_1\right)(\alpha+1-\varepsilon_K) = q\]and by definition \[\min\left\{k_0\frac{\beta_A-\beta_I}{2},\frac{5}{2}\right\} \ge \frac{(4\wedge\alpha)+1}{2} > q.\]This shows the claimed tail behavior by possibly increasing the constant $C$ for small $t\le t_1$. Furthermore, by the layer-cake formula we obtain, fixing some $q^\prime \in \left(q,\frac{(4\wedge\alpha)+1}{2}\right)$,\begin{equation}\label{eq:NuExpLC}\begin{split}\mathbf{E}[\vert \nu\vert^q] &=\int_0^\infty qt^{q-1}\mathbf{P}(\nu>t)\diff t\\&\le \int_0^{t_1}qt^{q-1}\diff t + \int_{t_1}^\infty qt^{q-1}Ct^{-q^\prime}\diff t\\&= t_1^q + \frac{q Ct_1^{q-q^\prime}}{q^\prime-q} <\infty.\end{split}\end{equation}\end{proof}

\subsubsection{Proof of Lemma \ref{Tdistconv}}\label{pro:Tdistconv}
We start by showing that $(\overline{\nu}_n)_{n\ge1}$ is independent of $M^0$.
\theoremstyle{plain}\begin{lemma}\label{Tdistribution} For $E\in\mathcal{B}([0,\infty)^\N)$ we have \[ \mathbb{P}\left((\overline{\nu}_1,\overline{\nu}_2,\dots)\in E\right) = \mathbf{P}\left(\left.(\nu_1,\nu_2,\dots)\in E\right\vert M^0=0\right). \]
\end{lemma}
\begin{proof} The claim follows analogously to the proof for Lemma \ref{lem:M_tail}. Using the notation used in that proof, we see that
\[\begin{split} \mathbb{P}&\left((\overline{\nu}_1,\overline{\nu}_2,\dots)\in E\right) \\ &= \frac{1}{\mathbf{P}(\mathcal{S})}\sum_{j=0}^\infty\mathbf{P}(B^{<j},\mathcal{G}_j,(\nu_{j+1},\dots,\nu_{j+k_0-1},\nu_{j+k_0},\nu_{j+k_0+1},\dots)\in E) \end{split}\]
and then we can perform the same steps as there, because \[\mathcal{G}_j\cap\{(\nu_{j+1},\dots,\nu_{j+k_0-1},\nu_{j+k_0},\nu_{j+k_0+1},\dots)\in E\}\]only depends on variables with index above $j$ and $B^{<j}$ depends only on variables with index below $j$.\end{proof}
\theoremstyle{plain}\begin{corollary}\label{cor:overnufinmom} For any $q\in\left[1,\frac{(\alpha\wedge 4)+1}{2}\right)$, we have $\mathbb{E}[\vert\overline{\nu}_n\vert^q] < \infty$.\end{corollary}\begin{proof} For $n\ge 1$ the claim follows directly by Lemma \ref{NuMoments} and Lemma \ref{Tdistribution}, noting that \[\mathbb{P}(\overline{\nu}_n > t) = \mathbf{P}(\nu_n > t\vert M^0 = 0) \le \frac{\mathbf{P}(\nu_n > t)}{\mathbf{P}(M^0=0)}. \qedhere\]\end{proof}
With this characterization of the joint distribution of $(\overline{\nu}_{k+n})_{n\ge1}$ for any starting point $k\ge0$, we can show that their joint distribution converges to the distribution of $(\nu_n)_{n\ge1}$ for $k\to\infty$.
\begin{proof}[Proof of Lemma \ref{Tdistconv}]
For shorthand we write\[\{\underline{\nu}_k\in E\} := \{(\nu_k,\nu_{k+1},\dots)\in E\}\text{ and } \{\underline{\overline{\nu}}_k\in E\} := \{(\overline{\nu}_k,\overline{\nu}_{k+1},\dots)\in E\}\]for any $k\ge1$. First we observe that by Lemma \ref{Tdistribution}, for $n\ge k_0$ we have\[\begin{split}&\vert\mathbb{P}(\underline{\overline{\nu}}_n \in E) -\mathbf{P}(\underline{\nu}_1 \in E) \vert =\left\vert\frac{\mathbf{P}(\underline{\nu}_n\in E,M^0=0)-\mathbf{P}(M^0=0)\mathbf{P}(\underline{\nu}_1 \in E)}{\mathbf{P}(M^0=0)}\right\vert .\end{split}\]Hence it suffices to show \[\vert\mathbf{P}(\underline{\nu}_n\in E,M^0=0)-\mathbf{P}(M^0=0)\mathbf{P}(\underline{\nu}_1\in E)\vert\le Cn^{-\alpha+\varepsilon}\] for large $n$. We observe that for all $m\ge k_0$ we have\[\{M^0=0\} = \mathcal{G}_{0,m}\setminus\left(\mathcal{G}_{0,m}\cap\bigcup_{j=m+1}^\infty\{K_j > j\}\right)\]
and hence 
\begin{equation}\label{eq:probdiff}\mathbf{P}(M^0 = 0) = \mathbf{P}(\mathcal{G}_{0,m}) - \mathbf{P}\left(\mathcal{G}_{0,m},\bigcup_{j=m+1}^\infty\{K_j > j\}\right).\end{equation}
In particular, we note that $\mathcal{G}_{0,m}$ only depends on $A_0,\dots, A_{m-1},I_1,\dots I_m$.
\\
Next, we denote by
\[\widetilde{K}_{n} := \inf\{n-K_{n},n+1-K_{n+1},\dots\}\]
the smallest site, which is used to determine if $\underline{\nu}_n\in E$. In particular, because for any $m\ge 1$ the event $\{K_{n} \le m\}$ only depends on $A_{n-m-1},\dots,A_{n-1}$ and $I_{n-m},\dots,I_n$, the event
\[\{\widetilde{K}_n \ge n-k,\underline{\nu}_n\in E\} \in\sigma(A_{x},I_{x+1},Y^{x,i}:x\ge n-k,i\in\N).\]
For $k_0\le k\le n-k_0$, we estimate, repeatedly applying \eqref{eq:probdiff} and using the dependencies written above
\begin{align*} \vert\mathbf{P}(\underline{\nu}_n&\in E,M^0=0) -\mathbf{P}(M^0=0)\mathbf{P}(\underline{\nu}_1\in E)\vert \\
\le&\mathbf{P}(\underline{\nu}_n\in E,\widetilde{K}_n < n-k, M^0=0) \\&+ \vert\mathbf{P}(\underline{\nu}_n\in E,\widetilde{K}_n\ge n-k, M^0 = 0) - \mathbf{P}(M^0=0)\mathbf{P}(\underline{\nu}_1\in E)\vert \\
\le&\mathbf{P}(\widetilde{K}_n < n - k) + \mathbf{P}\left(\underline{\nu}_n\in E, \widetilde{K}_n\ge n-k,\mathcal{G}_{0,n-k},\bigcup_{j=n-k+1}^\infty\{K_j > j\}\right) \\&+ \left\vert\mathbf{P}\left(\underline{\nu}_n\in E,\widetilde{K}_n\ge n-k,\mathcal{G}_{0,n-k}\right) - \mathbf{P}(M^0 = 0)\mathbf{P}(\underline{\nu}_1\in E)\right\vert\\
\le& \mathbf{P}(\widetilde{K}_n < n-k) + \mathbf{P}\left(\bigcup_{j=n-k+1}^\infty\{K_j > j\}\right) \\&+\left\vert\mathbf{P}\left(\underline{\nu}_n\in E,\widetilde{K}_n\ge n-k\right)\mathbf{P}\left(\mathcal{G}_{0,n-k}\right) - \mathbf{P}\left(\mathcal{G}_{0,n-k}\right)\mathbf{P}(\underline{\nu}_1\in E)\right\vert \\&+ \mathbf{P}\left(\mathcal{G}_{0,n-k},\bigcup_{j=n-k+1}^\infty\{K_j^{k_0} > j\}\right)\mathbf{P}(\underline{\nu}_1\in E) \\
\le& \mathbf{P}(\widetilde{K}_n < n - k) + 2\mathbf{P}\left(\bigcup_{j=n-k+1}^\infty\{K_j > j\}\right) \\&+ \left\vert\mathbf{P}\left(\underline{\nu}_n\in E,\widetilde{K}_n\ge n-k\right) - \mathbf{P}(\underline{\nu}_n\in E)\right\vert \\
\le& 2\mathbf{P}(\widetilde{K}_n < n- k) + 2\mathbf{P}\left(\bigcup_{j=n-k+1}^\infty\{K_j > j\}\right)\\
\le & 2\sum_{j=0}^\infty\mathbf{P}(K_{n+j}> j+k) + 2\sum_{j=n-k+1}^\infty\mathbf{P}(K_{n+j} > j) \\
\le& 2\sum_{j=0}^\infty C(j+k)^{-(\alpha+1)+\varepsilon}+2\sum_{j=n-k+1}Cj^{-(\alpha+1)+\varepsilon},\end{align*}for any $\varepsilon > 0$ and some constant $C\ge0$, using Lemma \ref{lem:Kmtail} in the last line.
\\Now for $n>2k_0$, plugging in $k = \lfloor\frac{n}{2}\rfloor$, we obtain a constant $C\ge 0$ such that
\[\vert\mathbb{P}(\underline{\overline{\nu}}_n\in E) -\mathbf{P}(\underline{\nu}_1\in E)\vert \le C n^{-\alpha+\varepsilon},\]
which concludes the proof.
\end{proof}
With this convergence we can easily show that $(\overline{\nu}_n)_{n\ge k_0}$ has asymptotically the same moments as $(\nu_n)_{n\ge1}$.
\theoremstyle{plain}\begin{corollary}\label{cor:momentconvergence}
For all $1\le q < \frac{(4\wedge\alpha)+1}{2}$ we have
\[\lim_{n\to \infty} \mathbb{E}[\vert \overline{\nu}_n\vert^q] = \mathbf{E}[\vert\nu\vert^q].\]
\end{corollary}
\begin{proof}
The layer cake formula yields
\[
\begin{split}
\vert\mathbb{E}[\vert \overline{\nu}_n\vert^q] -\mathbf{E}[\vert\nu\vert^q]\vert &= \left\vert\int_0^\infty qt^{q-1}\mathbb{P}(\overline{\nu}_n>t)\diff t-\int_0^\infty qt^{q-1}\mathbf{P}(\nu>t)\diff t\right\vert\\
& \le\int_0^\infty qt^{q-1}\vert\mathbb{P}(\overline{\nu}_n>t)-\mathbf{P}(\nu>t)\vert\diff t .
\end{split}
\]Because\[qt^{q-1}\vert\mathbb{P}(\overline{\nu}_n>t)-\mathbf{P}(\nu>t)\vert \le qt^{q-1}\frac{2}{\mathbf{P}(M^0=0)}\mathbf{P}(\nu>t) \in L^1([0,\infty))\]we can interchange the limit $n\to\infty$ with the integral and obtain from Lemma \ref{Tdistconv} that\[\lim_{n\to\infty}\vert\mathbb{E}[\vert \overline{\nu}_n\vert^q] -\mathbf{E}[\vert\nu\vert^q]\vert = 0.\qedhere\]\end{proof}
\subsubsection{Proof of Lemma \ref{conditinallyphimixing}}\label{pro:conditinallyphimixing}
We need to introduce some additional notation to make precise the reasoning from the sketch of the proof.
\theoremstyle{definition}
\theoremstyle{definition}\begin{definition} \label{helpsets}
For $m\in\N$ we define the index set
\[\underline{C}^\text{good}(m) := \left\{(\mathbf{a},\mathbf{i})\in \N_0^m\times\N^m\left\vert G(\mathbf{a},\mathbf{i}) = 1\right.\right\}\]
of all possible good configurations $((A_0,\dots,A_{m-1}),(I_1,\dots,I_m))$ given $\{M^0=0\}$.\\
For $m\in\N_0,n\in\N,(\mathbf{a},\mathbf{i})\in\underline{C}^\text{good}(m)$ we define the event that $(A_0,\dots,A_{m-1})$ equals $\mathbf{a}$ and $(I_1,\dots,I_m)$ equals $\mathbf{i}$, i.e.,
\[\{\mathbf{A}_m = \mathbf{a}\} := \bigcap_{j=0}^{m-1}\{A_{j}=a_j\}, \{\mathbf{I}_m = \mathbf{i}\} := \bigcap_{k=1}^{m}\{I_{k} = i_k\}\]Also, we define the event that, given $m\ge 1$ and some $(\mathbf{a},\mathbf{i})\in\underline{C}^\text{good}(m)$, also the upcoming random configuration after site $m$ is good \[((a_1,\dots,a_{m},A_{m},\dots,A_{m+n-1}),(i_1,\dots,i_m,I_{m+1},\dots,I_{m+n}))\in\underline{C}^\text{good}(m+n).\]That is, $\mathcal{G}_{m,n}(\mathbf{a},\mathbf{i})$ is the event\[ \bigcap_{x=m+1}^{m+n}\left\{G((a_0,\dots,a_{m-1},A_m,\dots,A_{m+n-1}),(i_1,\dots,i_m,I_{m+1},\dots,I_{m+n})) = 1\right\} \]with the obvious extension to $\mathcal{G}_{m,\infty}(\mathbf{a},\mathbf{i})$. We note that this is just an extension of the definition $\mathcal{G}_{0,n}$ from \ref{Msites}, where we fixed the first $m$ entries. We observe that because the first $m$ entries are fixed, the event $\mathcal{G}_{m,n}(\mathbf{a},\mathbf{i})$ only depends on $A_m,\dots,A_{m+n-1},I_{m+1},\dots,I_{m+n}$.\\ Also, we define the event $\mathcal{B}_{m,n}(\mathbf{a},\mathbf{i})$ that \[((a_1,\dots,a_{m},A_{m},\dots,A_{x}),(i_1,\dots,i_m,I_{m+1},\dots,I_{x}))\notin\underline{C}^\text{good}(m+x)\]for some $x > m+n$. That is, $\mathcal{B}_{m,n}(\mathbf{a},\mathbf{i})$ is the event\[ \bigcup_{x=m+n+1}^\infty\left(\left\{G((a_0,\dots,a_{m-1},A_m,\dots,A_{x}),(i_1,\dots, i_m,I_{m+1},\dots, I_x)) = 0\right\}\right).\]Again, $\mathcal{B}_{m,n}(\mathbf{a},\mathbf{i})$ only depends on $A_m,A_{m+1},\dots,I_{m+1},I_{m+2},\dots$.\end{definition}\begin{proof}[Proof of Lemma \ref{conditinallyphimixing}]Using Lemma \ref{Tdistribution}, we only need to show
\begin{equation}\label{survphi}\begin{split}&\left\vert\frac{\mathbf{P}\left(E,B\vert M^0=0\right)}{\mathbf{P}(E\vert M^0=0)} -\mathbf{P}(B\vert M^0=0)\right\vert \\ & = \left\vert\frac{\mathbf{P}(E,B,M^0=0)}{\mathbf{P}(E,M^0=0)}- \mathbf{P}(B\vert M^0=0)\right\vert \le C n^{-\alpha+\varepsilon}\end{split}\end{equation}
for some $C>0$, all $m\ge k_0,n\in\N$ and\[E\in\sigma(\widetilde{\nu}_1(0),\dots,\widetilde{\nu}_{k_0-1}(0),\nu_{k_0},\dots\nu_m),\mathbf{P}(E\vert M^0 = 0)>0,B \in\sigma(\nu_k:k\ge m+n).\]The idea is to split up the probability into different events, similarly as in the proof of Lemma \ref{Tdistconv}, where one event is unlikely and on the other event we can use independence. Then, given $M^0 = 0$, we estimate how far above $m$ the event $E$ influences $\mathbf{A},\mathbf{I}$ and how far below $m+n$ the event $B$ influences $\mathbf{A},\mathbf{I},\mathbf{Y}$. Making the distance $n$ large, these areas of influence will be disjoint, and the two events will become independent, except on some unlikely event. We recall\[\widetilde{K}_m = \inf\{m+j-K_{m+j}:j\ge0\}\]and that \[\{\widetilde{K}_{m+n}\ge m+k,B\}\in\sigma(A_x,I_{x+1},Y^{x,i}:x\ge m+k,i\in\N).\]We also recall that, for any $l\ge k_0$, we have\[\{M^0 = 0\} = \mathcal{G}_{0,l}\setminus(\mathcal{G}_{0,l}\cap\mathcal{B}_{0,l}) = \mathcal{G}_{0,l}\setminus\left(\mathcal{G}_{0,l}\cap\bigcup_{j=l+1}^\infty\{K_j>j\}\right).\]With this in mind, we compute that for any $n>2k_0$ and $k_0<k\le n-k_0$, the probability $\mathbf{P}(E,B,M^0=0)$ is equal to
\[
\begin{split}
&\quad\mathbf{P}\left(E,B,\widetilde{K}_{m+n} \ge m+k,\mathcal{G}_{0,m+k}\right) - \mathbf{P}\left(E,B,\widetilde{K}_{m+n} \ge m+k,\mathcal{G}_{0,m+k},\mathcal{B}_{0,m+k}\right) \\&\quad+\mathbf{P}\left(E,B,\widetilde{K}_{m+n}< m+k,\mathcal{G}_{0,m+k}\right)  - \mathbf{P}\left(E,B,\widetilde{K}_{m+n}< m+k,\mathcal{G}_{0,m+k},\mathcal{B}_{0,m+k}\right) \nonumber .\label{probsplit}
\end{split}
\]
The first probability contains independent events and thus is equal to \[\mathbf{P}\left(E,\mathcal{G}_{0,m+k}\right)\mathbf{P}\left(B,\widetilde{K}_{m+n}\ge m+k\right)\] Now to estimate \eqref{survphi}, we use the triangle inequality, putting together the first term with the term $-\mathbf{P}(B\vert M^0=0)$ and then each other term for itself. This gives that the term in \eqref{survphi} is bounded by
\[\begin{split}&\left\vert\frac{\mathbf{P}\left(E,\mathcal{G}_{0,m+k}\right)\mathbf{P}\left(B,\widetilde{K}_{m+n}\ge m+k\right) - \mathbf{P}(B,M^0 = 0)}{\mathbf{P}(E,M^0= 0)}\right\vert \\&+ \frac{\mathbf{P}\left(E,B,\widetilde{K}_{m+n} \ge m+k,\mathcal{G}_{0,m+k},\mathcal{B}_{0,m+k}\right)}{\mathbf{P}(E,M^0= 0)}\\&+\frac{\mathbf{P}\left(E,B,\widetilde{K}_{m+n}< m+k,\mathcal{G}_{0,m+k}\right)}{\mathbf{P}(E,M^0= 0)} \\ &+\frac{\mathbf{P}\left(E,B,\widetilde{K}_{m+n}< m+k,\mathcal{G}_{0,m+k},\mathcal{B}_{0,m+k}\right)}{\mathbf{P}(E,M^0= 0)}.\end{split}\]
To estimate these terms, we observe that for $(\mathbf{a},\mathbf{i})\in\underline{C}^\text{good}(m)$ and $l > k_0$, possibly $l=\infty$, we have\begin{align*}\mathbf{P}&(\mathbf{A}_m = \mathbf{a},\mathbf{I}_m = \mathbf{i}, E,\mathcal{G}_{0,m+l}) = \mathbf{P}\left(\mathbf{A}_m = \mathbf{a},\mathbf{I}_m = \mathbf{i}, E,\mathcal{G}_{m,m+l}(\mathbf{a},\mathbf{i})\right) \\&= \mathbf{P}\left(\mathbf{A}_m = \mathbf{a},\mathbf{I}_m = \mathbf{i}, E\right)\mathbf{P}\left(\mathcal{G}_{m,m+l}(\mathbf{a},\mathbf{i})\right).\end{align*}We note that
\begin{equation}\label{eq:G_contain}
\mathcal{G}_{m,m+l}(\mathbf{a},\mathbf{i}) \supset \bigcap_{n=1}^\infty\left\{\sum_{j=1}^n A_{m+j-1} \ge \beta_A n,\sum_{j=1}^n I_{m+j}\le \beta_I n\right\}
\end{equation}
for any $(\mathbf{a},\mathbf{i})\in\underline{C}^\text{good}(m)$. Because $\mu_I < \beta_I < \beta_A < \mu_A$, the right-hand side of \eqref{eq:G_contain} has positive probability, which we denote by $\delta$. \\We compute\begin{equation}\label{unifAbound}\begin{split} &\frac{\mathbf{P}\left(E,\mathcal{G}_{0,m+k}\right)}{\mathbf{P}(E,M^0=0)} =\sum_{(\mathbf{a},\mathbf{i})\in\underline{C}^\text{good}(m)}\frac{\mathbf{P}\left(\mathbf{A}_m = \mathbf{a},\mathbf{I}_m = \mathbf{i}, E,\mathcal{G}_{0,m+k}\right)}{\sum_{(\mathbf{a},\mathbf{i})\in\underline{C}^\text{good}(m)}\mathbf{P}\left((\mathbf{A}_m = \mathbf{a},\mathbf{I}_m = \mathbf{i}, E,M^0=0\right)} \\ &= \sum_{(\mathbf{a},\mathbf{i})\in\underline{C}^\text{good}(m)}\frac{\mathbf{P}\left((\mathbf{A}_m = \mathbf{a},\mathbf{I}_m = \mathbf{i}, E\right)\mathbf{P}\left(\mathcal{G}_{m,k}(\mathbf{a},\mathbf{i})\right)}{\sum_{(\mathbf{a},\mathbf{i})\in\underline{C}^\text{good}(m)}\mathbf{P}\left(\mathbf{A}_m = \mathbf{a},\mathbf{I}_m = \mathbf{i}, E\right)\mathbf{P}\left(\mathcal{G}_{m,\infty}(\mathbf{a},\mathbf{i})\right)}\\ &\le \sum_{(\mathbf{a},\mathbf{i})\in\underline{C}^\text{good}(m)}\frac{\mathbf{P}\left((\mathbf{A}_m = \mathbf{a},\mathbf{I}_m = \mathbf{i}, E\right)}{\sum_{(\mathbf{a},\mathbf{i})\in\underline{C}^\text{good}(m)}\mathbf{P}\left((\mathbf{A}_m = \mathbf{a},\mathbf{I}_m = \mathbf{i}, E\right)\delta} = \frac{1}{\delta}. \end{split}\end{equation}Hence, using the uniform bound \eqref{unifAbound}, we can show for $k_0 < k^\prime < k < n-k_0$ that\begin{equation}\label{unifAconv}\begin{split}&\left\vert\frac{\mathbf{P}\left(E,\mathcal{G}_{0,m+k}\right)}{\mathbf{P}(E,M^0=0)}-1\right\vert = \frac{\mathbf{P}\left(E,\mathcal{G}_{0,m+k}\right)-\mathbf{P}(E,M^0=0)}{\mathbf{P}(E,M^0=0)}\\&=\frac{\mathbf{P}\left(E,\mathcal{G}_{0,m+k},\mathcal{B}_{0,m+k}\right)}{\mathbf{P}(E,M^0=0)} = \frac{\mathbf{P}\left(E,\mathcal{G}_{0,m+k},\bigcup_{j=m+k+1}^\infty\{K_j>j\}\right)}{\mathbf{P}(E,M^0=0)}\\&\le \frac{\mathbf{P}\left(E,\mathcal{G}_{0,m+k^\prime},\bigcup_{j=m+k+1}^\infty\{K_j>j-m-k^\prime\}\right)}{\mathbf{P}(E,M^0=0)}\\& = \frac{\mathbf{P}\left(E,\mathcal{G}_{0,m+k^\prime}\right)\mathbf{P}\left(\bigcup_{j=m+k+1}^\infty\{K_j>j-m-k^\prime\}\right)}{\mathbf{P}(E,M^0=0)}\\&\le \frac{1}{\delta}\sum_{j=1}^\infty\mathbf{P}(K_{j+m+k} > j+k-k^\prime) \le C_1 (k-k^\prime)^{-\alpha+\varepsilon},\end{split}\end{equation}for any $\varepsilon\in (0,\alpha)$ and some $C_1 = C_1(\varepsilon)\ge0$, using Lemma \ref{lem:Kmtail}.
With this we can now begin to estimate each term of the triangle inequality used for \eqref{probsplit}. First, by an analogous argument as we used starting from the second line of \eqref{unifAconv},\[\begin{split}&\frac{\mathbf{P}\left(E,B,\widetilde{K}_{m+n} \ge m+k,\mathcal{G}_{0,m+k},\mathcal{B}_{0,m+k}\right)}{\mathbf{P}(E,M^0=0)} \\&\le\frac{\mathbf{P}\left(E,\mathcal{G}_{0,m+k},\mathcal{B}_{0,m+k}\right)}{\mathbf{P}(E,M^0=0)}\le C_1(k-k^\prime)^{-\alpha+\varepsilon}.\end{split}\]Similarly,\[\begin{split}&\frac{\mathbf{P}\left(E,B,\widetilde{K}_{m+n} < m+k,\mathcal{G}_{0,m+k},\mathcal{B}_{0,m+k}\right)}{\mathbf{P}(E,M^0=0)}\\ &\le\frac{\mathbf{P}\left(E,\mathcal{G}_{0,m+k},\mathcal{B}_{0,m+k}\right)}{\mathbf{P}(E,M^0=0)}\le C_1(k-k^\prime)^{-\alpha+\varepsilon}.\end{split}\]Next, using \eqref{unifAbound}, we find\[\begin{split}&\frac{\mathbf{P}\left(E,\mathcal{G}_{0,m+k},B,\widetilde{K}_{m+n}< m+k\right)}{\mathbf{P}(E,M^0=0)}\le\frac{\mathbf{P}\left(E,\mathcal{G}_{0,m+k},\widetilde{K}_{m+n}< m+k\right)}{\mathbf{P}(E,M^0=0)}\\= &\frac{\mathbf{P}\left(E,\mathcal{G}_{0,m+k}\right)}{\mathbf{P}(E,M^0=0)}\mathbf{P}\left(\widetilde{K}_{m+n}< m+k\right) \le \frac{\mathbf{P}\left(\widetilde{K}_{m+n}< m+k\right)}{\delta} \le C_2 (n-k)^{-\alpha+\varepsilon}\end{split}\]for some $C_2 = C_2(\varepsilon)$, using Lemma \ref{lem:Kmtail}. Finally, to estimate the first term of \eqref{probsplit} minus $\mathbf{P}(B\vert M^0 =0)$, using the same split up as for \eqref{probsplit} but now on $\mathbf{P}(B,M^0=0)$ (that is, with $E$ replaced by $\Omega$), we compute\begin{align*}&\left\vert \frac{\mathbf{P}\left(E,\mathcal{G}_{0,m+k}\right)}{\mathbf{P}(E,M^0=0)}\mathbf{P}\left(B,\widetilde{K}_{m+n}\ge m+k\right) -\mathbf{P}(B\vert M^0= 0)\right\vert \\\le&\left\vert \frac{\mathbf{P}\left(E,\mathcal{G}_{0,m+k}\right)}{\mathbf{P}(E,M^0=0)} -\frac{\mathbf{P}\left(\mathcal{G}_{0,m+k}\right)}{\mathbf{P}(M^0=0)}\right\vert\mathbf{P}\left(B,\widetilde{K}_{m+n}\ge m+k\right) \\&+\frac{\mathbf{P}\left(B,\widetilde{K}_{m+n}\ge m+k,\mathcal{G}_{0,m+k},\mathcal{B}_{0,m+k}\right)}{\mathbf{P}(M^0=0)}\\&+\frac{\mathbf{P}\left(B,\widetilde{K}_{m+n}<m+k,M^0=0\right)}{\mathbf{P}(M^0=0)} \\\le&\left\vert \frac{\mathbf{P}\left(E,\mathcal{G}_{0,m+k}\right)}{\mathbf{P}(E,M^0=0)}-\frac{\mathbf{P}\left(\mathcal{G}_{0,m+k}\right)}{\mathbf{P}(M^0=0)}\right\vert +C_1(k-k^\prime)^{-\alpha+\varepsilon}+ C_2(n-k)^{-\alpha+\varepsilon} \\\le& \left\vert \frac{\mathbf{P}\left(E,\mathcal{G}_{0,m+k}\right)}{\mathbf{P}(E,M^0=0)}-1\right\vert+ \left\vert 1-\frac{\mathbf{P}\left(\mathcal{G}_{0,m+k}\right)}{\mathbf{P}(M^0=0)}\right\vert +C_1(k-k^\prime)^{-\alpha+\varepsilon}+ C_2(n-k)^{-\alpha+\varepsilon} \\\le& 3C_1(k-k^\prime)^{-\alpha+\varepsilon}+ C_2(n-k)^{-\alpha+\varepsilon}.\end{align*}Now plugging in $k = \left\lceil\frac{2}{3}n\right\rceil,k^\prime = \left\lceil\frac{1}{3}n\right\rceil$ and combining all estimates, we find a constant $C = C(\varepsilon)\ge 0$ such that\[\left\vert\frac{\mathbf{P}(E,B,M^0=0)}{\mathbf{P}(E,M^0=0)} -\mathbf{P}(B\vert M^0=0)\right \vert \le C n^{-\alpha+\varepsilon}\]for large enough $n\in\N$. Possibly increasing this constant $C$, the estimate holds already for all $n\in\N$, which finishes the proof.\end{proof}

\subsubsection{Proof of Lemma \ref{Qconvspeed}}\label{pro:Qconvspeed}
First, we establish that the variance of the partial sum grows only linearly in $n$, given that the variance is finite.\theoremstyle{plain}\begin{lemma}\label{L2bound}If $\alpha>3$, there is a $C>0$ such that for all $n\ge1$ we have\[\mathbb{E}\left[\left\vert\sum_{k=1}^n\left(\overline{\nu}_k-\mathbb{E}[\overline{\nu}_k]\right)\right\vert^2\right] \le Cn.\]\end{lemma}
\begin{proof}We set $Q_k := \overline{\nu}_k -\mathbb{E}[\overline{\nu}_k]$ and note that because $\alpha>3$ we obtain from Lemma \ref{NuMoments} that $(Q_k)_{k\in\Z}\subset L^2(\mathbf{\Omega})$.
Now a classical result of Ibragimov \cite[Lemma 1.1]{I1962} states that $\phi$-mixing implies weak correlations, in the sense that
\[\mathbb{E}[Q_mQ_{m+n}] \le 2\sqrt{\phi(n)\mathbb{E}[Q_m^2]\mathbb{E}[Q_{m+n}^2]}.\]
Using Corollary \ref{cor:overnufinmom} and Corollary \ref{cor:momentconvergence}, there is a constant $C_Q \ge 0$ such that $\mathbb{E}[Q_n^2] \le C_Q$ for all $n\ge 1$. With this, we compute
\[
\begin{split}
\mathbb{E}\left[\left(\sum_{k=1}^{n} Q_k\right)^2\right] & \le \sum_{k=1}^{n}\sum_{j=1}^{n}\vert\mathbb{E}[Q_kQ_j]\vert = \sum_{k=1}^{n}\sum_{j=k}^{n}\vert\mathbb{E}[Q_kQ_j]\vert + \sum_{k=2}^{n}\sum_{j=1}^{k-1}\vert\mathbb{E}[Q_kQ_j]\vert \\
&= \sum_{k=1}^{n}\sum_{j=0}^{n-k}\vert\mathbb{E}[Q_kQ_{k+j}]\vert +\sum_{j=1}^{n}\sum_{k=1}^{n-j}\vert\mathbb{E}[Q_{j+k}Q_j]\vert \\
&= \sum_{k=1}^{n}\left(\mathbb{E}[Q_k^2] + 2\sum_{j=1}^{n-k}\vert\mathbb{E}[Q_kQ_{k+j}]\vert \right) \\
&\le nC_Q + 2\sum_{k=1}^{n-1}\sum_{j=1}^{n-k} 2\sqrt{\phi(j)}C_Q \le C_Q\left(n + 4C\sum_{k=1}^{n-1}\sum_{j=1}^{n-k}j^{-\frac{\alpha-\varepsilon}{2}}\right)\\
&\le C_Q\left(n + 4C\sum_{k=1}^{n-1}\sum_{j=1}^{\infty} j^{-\frac{\alpha-\varepsilon}{2}}\right) \le n C_Q\left(1+4C\sum_{j=1}^{\infty} j^{-\frac{\alpha-\varepsilon}{2}}\right),
\end{split}
\]where we used Lemma \ref{conditinallyphimixing} for some $\varepsilon\in(0,1)$ and note the sum is finite, by the assumption $\alpha >3$.\end{proof}Now it follows from classical results that the bounds on the second moments transfer to higher moments.\begin{proof}[Proof of Lemma \ref{Qconvspeed}]Using the estimate in Lemma \ref{L2bound}, we can conclude by a classical result for $\phi$-mixing sequences from Ibragimov \cite[Lemma 1.9]{I1962} that \[\mathbb{E}\left[\left\vert\sum_{k=1}^n\overline{\nu}_k-\mathbb{E}[\overline{\nu}_k]\right\vert^q\right] \le C_! n^{\frac{q}{2}}\]
for some constant $C_1 > 0$. Thus, using another classical result of Serfling \cite[Corollary B1]{S1970}, we already have\[\mathbb{E}\left[\max_{1\le i\le n}\left\vert\sum_{k=1}^i\overline{\nu}_k-\mathbb{E}[\overline{\nu}_k]\right\vert^q\right] \le C n^{\frac{q}{2}}\]
for some constant $C_2>0$.\end{proof}

\subsubsection{Proof of Lemma \ref{lem:T_Tfinexp}}\label{pro:T_Tfinexp}
\begin{proof}[Proof of Lemma \ref{lem:T_Tfinexp}]
 We can make use of the variables $(\overline{\nu}_n)_{n\ge1} = (\overline{\nu}_n^0)_{n\ge1}$ to bound the time $T^1$ almost surely, namely
 \[
T^1 \le \sum_{k=1}^{M^1-M^0}\overline{\nu}_k. 
 \]
 Hence, we can estimate for any $\delta \in (0,1),\varepsilon\in(0,\alpha)$ and $C_\varepsilon$ as in Lemma \ref{lem:M_tail} that
 \[
 \begin{split}
 \mathbb{P}(T^1 > t \vert M^0 = 0) &\le \mathbb{P}(M^1 > t^\delta\vert M^0 = 0) + \mathbb{P}\left(\sum_{k=1}^{\lceil t^\delta\rceil} \overline{\nu}_k > t\right)\\
 &\le C_\varepsilon t^{-\delta(\alpha-\varepsilon)} + \mathbb{P}\left(\sum_{k=1}^{\lceil t^\delta\rceil}\overline{\nu}_k-\mathbb{E}[\nu_k] > t-\lceil t^\delta\rceil\sup_{k\ge1}\mathbb{E}[\overline{\nu}_k]\right)
 \end{split}
 \]
 For $t$ large enough it holds that $t - \lceil t^\delta\rceil\sup_{k\ge 1}\mathbb{E}[\overline{\nu}_k] \ge \frac{t}{2}$. Hence, using the Markov inequality, we can estimate the probability in the last line by
 \[
 \mathbb{E}\left[\left\vert\sum_{k=1}^{\lceil t^{\delta}\rceil} \overline{\nu}_k-\mathbb{E}[\overline{\nu}_k]\right\vert^{q}\right]\left\vert \frac{t}{2}\right\vert^{-q}
 \]
 for any $q \in\left[1, \frac{(\alpha\wedge 4)+1}{2}\right)$. Using Lemma \ref{Qconvspeed}, we obtain a constant $C> 0$ such that for large enough $t$
 \[
 \mathbb{E}\left[\left\vert\sum_{k=1}^{\lceil t^{\delta}\rceil} \overline{\nu}_k-\mathbb{E}[\overline{\nu}_k]\right\vert^{q}\right] \le C t^{\delta\frac{q}{2}}.
 \]
Now all that is left to check is that taking $\delta > \frac{1}{\alpha}$ close enough to $\frac{1}{\alpha}$ and a small $\varepsilon \in\left(0,\alpha-\frac{1}{\delta}\right)$, we can take $q < \frac{1+(\alpha\wedge 4)}{2}$ large enough, such that we also have $q\left(1-\frac{\delta}{2}\right) > 1$. The condition for that to be possible is that
 \[
 \frac{\alpha+1}{2} > \frac{2}{2-\frac{1}{\alpha}}
 \]
which is equivalent to
 \[
 2\alpha^2 - 3\alpha - 1 > 0.
 \]
Solving for $\alpha$ yields that we must make the assumption
 \[
 \alpha > \frac{3+\sqrt{11}}{2},
 \]
noting that this is also the reason why we chose the parameter $4 = \left\lceil \frac{3+\sqrt{11}}{2}\right\rceil$ in the definition of good configurations \ref{auxiliaryfront}.
\end{proof}

\subsubsection{Proof of Proposition \ref{prop:renewal_sites}}\label{pro:rho_alphamix}
 In this section we will show the existence of the renewal sites claimed in Proposition \ref{prop:renewal_sites}. The key idea is to observe that for each good site $M^i$ there is a positive probability that the front grows linearly fast, and if it does not, then the probability that it falls below linear speed only at a large distance to $M^i$ decreases sufficiently fast. Also, retrying for this linear growth at the next good site is independent from the last attempt. On the other hand, each parasite that is already generated before reaching the good site $M^i$ only moves diffusely and thus has a positive probability of never catching up to a linearly moving front. Then for each parasite, we can repeat this trial until at some point none of the parasites generated below a certain good site will interact with the front after that good site is infected. This approach of a renewal-time point is inspired by the work of \cite{C2009}, in which, by showing finite moments of this renewal-time, even stronger results can be deduced. \\
 We begin with establishing that the front can grow linearly fast after a good site with high probability.
 \theoremstyle{plain}\begin{lemma}\label{lem:front_above}
Assume that $\alpha > 3$ and $\lambda < \frac{1}{\sup_{j\ge1}\mathbb{E}[\overline{\nu}_j^0]}$. There is a $\delta_1 > 0$ such that 
 \[
\mathbb{P}\left(\bigcap_{k=1}^\infty\left\{\sum_{j=1}^k\overline{\nu}_j^0\le \frac{k}{\lambda}\right\}\right)\ge \delta_1.
 \]
Also, for $n\in\N$, we have
 \[
\mathbb{P}\left(\bigcap_{k=1}^n\left\{\sum_{j=1}^k\overline{\nu}_j^0\le \frac{k}{\lambda}\right\},\bigcup_{m > n}\left\{\sum_{j=1}^m\overline{\nu}_j^0 > \frac{m}{\lambda}\right\}\right) \le C n^{-\frac{q}{2}}
 \]
 for any $q < \frac{1+\alpha\wedge 4}{2}$ and some $C> 0$.
 \end{lemma}
 \begin{proof}
 The first claim follows from Lemma \ref{Uinf} and will be shown in Section \ref{subsec:proof_poldec}. The second claim can be seen as follows. For $\lambda < \frac{1}{\sup_{j\ge1}\mathbb{E}[\overline{\nu}_j^0]}$ we have by Lemma \ref{Qconvspeed} for any $q < \frac{1+\alpha\wedge 4}{2}$ that
\[
\begin{split}
 \mathbb{P}&\left(\bigcap_{k=1}^n\left\{\sum_{j=1}^k\overline{\nu}_j^0\le \frac{k}{\lambda}\right\},\bigcup_{m > n}\left\{\sum_{j=1}^m\overline{\nu}_j^0 > \frac{m}{\lambda}\right\}\right) \le \mathbb{P}\left(\sup_{m>n}\left\{\frac{1}{m}\sum_{j=1}^m\overline{\nu}_j^0 > \frac{1}{\lambda}\right\}\right) \\
 &\le\sum_{k=0}^\infty\mathbb{P}\left(\sup_{2^kn \le m \le 2^{k+1}n}\frac{1}{2^kn}\sum_{j=1}^n\overline{\nu}_j^0 > \frac{1}{\lambda}\right) \\ 
 &\le \sum_{k=0}^\infty \left(\left(\frac{1}{\lambda}-\sup_{j\ge1}\mathbb{E}[\overline{\nu}_j^0]\right)2^kn\right)^{-q}\mathbb{E}\left[\left\vert\sum_{j=1}^{2^{k+1}n}\overline{\nu}_j^0-\mathbb{E}[\overline{\nu}_j^0]\right\vert^q\right] \\
 &\le \sum_{k=0}^\infty C 2^{-\frac{q}{2}k} n^{-\frac{q}{2}} \le C n^{-\frac{q}{2}}
\end{split}
\]
for some $C>0$ that may vary from line to line.
 \end{proof}
We cite the following well-known fact about a simple symmetric random walk not moving ballistically with high probability.
 \theoremstyle{plain}\begin{lemma}\label{lem:walk_above}
Recall that $(Y_t)_{t\ge0}$ is a simple symmetric random walk starting in $0$. For any $\lambda > 0$ there is a $\delta_2 > 0$ such that
 \[
 \mathbb{P}\left(\bigcap_{t\ge0}\left\{Y_t \le 1+\lfloor\lambda t\rfloor\right\}\right) \ge \delta_2.
 \]
 Also, there are $C_1,c_2 > 0$ such that for any $n\in\N$ we have
 \[
 \mathbb{P}\left(\bigcap_{k=1}^n\left\{\tau_k \ge \frac{k-1}{\lambda}\right\},\bigcup_{m>n}\left\{\tau_m < \frac{m-1}{\lambda}\right\}\right) \le C_1 \exp(-c_2n).
 \]
 where 
 \[
 \tau_k = \inf\{t\ge0: Y_t \ge k\}.
 \]
 \end{lemma}
 \begin{proof}
 This is a classically known result for random walks. A proof can be found in \cite[Lemma 8]{C2007}.
 \end{proof}
To construct the renewal sites, for any $j\ge0$ we want to find a renewal site $j < R_\text{good}(j)<\infty$, such that the front after reaching $R_\text{good}(j)$ only depends on the parasites born at a site $x$ with $x\ge R_\text{good}(j)$, and we can uniformly in $j$ bound the tail of $R_\text{good}(j) - j$. Then we can iteratively define $R^{i+1} := R_\text{good}(R^i)$. \\
To begin the construction, we first consider a single parasite $(x,i)$, and we will show that there is an almost surely finite good site $ M^{J_\text{good}(x,i)} > x$ such that after $T^{J_\text{good}(x,i)}$ the front moves linearly fast and the parasite $(x,i)$ will be either dead or it will never reach the linearly moving front again. Then we will show that if we define 
\[
R_1(j) := M^{J^1_\text{good}(j)} := \sup_{0\le x\le j,1\le i \le A_x} M^{J_\text{good}(x,i)} 
\]
for $j\ge0$ we have
\[
\lim_{n\to\infty}\sup_{j\ge0}\mathbb{P}(R_1(j)-j > n) = 0.
\]
In particular, none of the parasites born between $0$ and $j$ will catch up to the front after time $T^{J^1_\text{good}(j)}$. \\
Next, we observe that for some fixed $\delta > 0$, there is a probability of at least $\delta$ that also none of the parasites born between $j$ and $R_1(j)$ catch up to the linear moving front after $T^{J^1_\text{good}(j)}$. If this happens, we have found the renewal site $R_\text{good}(j)=M^{J_\text{good}^1(j)}$. However, if some parasite born between $j$ and $R_1(j)$ does catch up to the front after $T^{J_\text{good}^1(j)}$, arguing analogously as before, this does not happen too far away from $R_1(j)$. In particular, we can find another almost surely finite good site $R_2(j) = M^{J^2_\text{good}(j)}$ such that $\mathbb{P}(R_2(j)-j > n)$ goes to $0$ uniformly over all $j\ge 0$ as $n\to\infty$ and none of the parasites generated between $j$ and $R_1(j)$ catch up to the linearly moving front after $T^{J^2_\text{good}(j)}$. Again, there is a probability of at least $\delta$ that none of the parasites born between $R_1(j)$ and $R_2(j)$ catch up to the linearly moving front after time $T^{J^2_\text{good}(j)}$. If this event occurs, we have found the renewal site $R_\text{good}(j) = M^{J^2_\text{good}(j)}$. If the event does not occur, we repeat the procedure above for the parasites generated between $R_1(j)$ and $R_2(j)$, then for those between $R_2(j)$ and $R_3(j)$, and so on until at some $R_{\text{good}}(j) = M^{J_\text{good}(j)}$, none of the parasites generated before $T^{J_\text{good}(j)}$ will catch up to the linearly moving front. Furthermore, $\mathbb{P}(R_\text{good}(j) - j > n)$ tends to $0$ uniformly over all $j\ge0$ as $n\to\infty$. This will yield the renewal sites $(R^i)_{i\ge0}$ by iteratively finding the next $R_\text{good}(R^i)$, from which on the front moves linearly and the parasites that were born below $R_\text{good}(R^i)$ never catch up to the front and thus don't influence the front anymore. The uniform tail bound will imply that $R^i < \infty$ for all $i\ge0$.
\theoremstyle{plain}\begin{lemma}\label{lem:renewalsite} 
Suppose the assumptions of Proposition \ref{prop:renewal_sites}. For any $j\ge0$ there is an almost surely finite $R_\text{good}(j) > j$ such that after the front has reached $R_\text{good}(j)$, none of the parasites that were born at a site $x$ with $x < R_\text{good}(j)$ reach the front anymore. Also, it holds that \[ \lim_{n\to\infty}\sup_{j\ge0}\mathbb{P}(R_\text{good}(j)-j > n) = 0. \]
\end{lemma}
\begin{proof}
We carry out the construction explained before this lemma. For $x\ge0$ and $i\in\N$, let \[J_1(x,i) := J(x) := \inf\{k\ge 0: M^k > x\} \] be the index of the first good site after $x$. By Lemma \ref{lem:M_tail}, we have that
\[\mathbb{P}(M^{J(x)} - x > n) \le C_1n^{-q_1} \]
for any $q_1 < \alpha$ and a suitable $C_1 > 0$. At time $T^{J_1(x,i)}$, when the front is at $M^{J_1(x,i)}$, the parasite $(x,i)$ can be either already dead, or it has to be at a site $y < M^{J_1(x,1)}$. We note here that since we condition on the survival of parasites, we have $T^{J_1(x,i)} < \infty$ almost surely under $\mathbb{P}$. If the parasite is already dead, we set 
\[J_\text{good}(x,i) := J_1(x,i). \]
If the parasite $(x,i)$ is still alive, then by Lemma \ref{lem:front_above} and Lemma \ref{lem:walk_above}, for some suitable $\lambda$ small enough, there is a positive probability of at least $\delta_2\delta_1$ that the front after time $T^{J_1(x,i)}$ stays above the linear line $M^{J_1(x,i)}+\lfloor \lambda (t-T^{J_1(x,i)})\rfloor$, and the path of parasite $(x,i)$ after reaching $M^{J_1(x,i)}$ stays below a linear line with slope $\frac{\lambda}{2}$. In particular, in this event the parasite $(x,i)$ never influences the front again, and we set 
\[ J_\text{good}(x,i) := J_1(x,i). \]
If one of these events does not occur, then again by Lemma \ref{lem:front_above} and Lemma \ref{lem:walk_above}, the vertex $B_1(x,i)$ at which either the sum in Lemma \ref{lem:front_above} is too large or the hitting time in Lemma \ref{lem:walk_above} is too small satisfies 
\[ \mathbb{P}(n < B_1(x,i) - M^{J_1(x,i)} <\infty\vert J_\text{good}(x,i)\neq J_1(x,i), M^{J_1(x,i)}) \le C_2 n^{-\frac{q_2}{2}}~~a.s.~\] 
for any $q_2 < \frac{1+\alpha\wedge 4}{2}$ and a suitable $C_2 > 0$. In this case, we repeat the process of finding the next good site after $B_1(x,i)$ and set 
\[ J_2(x,i) := J(B_1(x,i)) = \inf\{k\ge 0: M^k > B_1(x,i)\}, \]
which again satisfies 
\[ \mathbb{P}(M^{J_2(x,i)}- B_1(x,i) > n\vert J_\text{good}(x,i)\neq J_1(x,i),B_1(x,i)) \le C_1 n^{-q_1} ~~a.s.~\] 
for $q_1<\alpha$ and $C_1> 0$ as before. If the parasite $(x,i)$ is dead at time $T^{J_2(x,i)}$, we set 
\[J_\text{good}(x,i) := J_2(x,i). \]
Otherwise, we first check if the accumulated sums of $(\overline{\nu}_k^{J_2(x,i)})_{k\ge0}$ stay small in a sense of Lemma \ref{lem:front_above} and, consequently, the front stays above the linear line $M^{J_2(x,i)}+\lfloor \lambda (t-T^{J_2(x,i)})\rfloor$, and if the path of the parasite $(x,i)$ after reaching $M^{J_2(x,i)}$ stays below a linear line with slope $\frac{\lambda}{2}$ in the sense of Lemma \ref{lem:walk_above}. If this event occurs, the parasite $(x,i)$ never reaches the front after $T^{J_2(x,i)}$, and we set \[J_\text{good}(x,i) := J_2(x,i). \]
 We note that $(\overline{\nu}_k^{J_2(x,i)})_{k\ge1}$ depends only on variables with index above $M^{J_2(x,i)}$. Also, the time when the parasite $(x,i)$ reaches $M^{J_2(x,i)}$ is a stopping time with respect to the filtration $\sigma(\mathbf{A},\mathbf{I},w_s:0\le s\le t)$, where we recall that $w_t\in\mathbb{L}_\theta$ is the state of the SIMI at time $t$; see Section \ref{Construction}. In particular, conditionally on $M^{J_2(x,i)}$, the path of the parasite $(x,i)$ after reaching $M^{J_2(x,i)}$ is independent of the path of the parasite before reaching $M^{J_2(x,i)}$. This implies that the event $\{J_\text{good}(x,i) = J_2(x,i)\}$ occurs with probability at least $\delta_1\delta_2$, conditionally on $\{J_\text{good}(x,i) \neq J_1(x,i)\}$ and the values of $B_1(x,i),M^{J_2(x,i)}$.\\
 Again, if the event does not occur, applying Lemma \ref{lem:front_above} and Lemma \ref{lem:walk_above} and using the independence we just described, the site $B_2(x,i)$ at which either the accumulated sums of $(\overline{\nu}_k^{J_2(x,i)})_{k\ge1}$ are above $\frac{k}{\lambda}$ or the path of the parasite $(x,i)$ after reaching $M^{J_2(x,i)}$ is above a linear line with slope $\frac{\lambda}{2}$, satisfies
 \[
 \mathbb{P}\left(n < B_2(x,i) - M^{J_2(x,i)} <\infty\left\vert\begin{matrix} J_1(x,i) \neq J_\text{good}(x,i),\\B_1(x,i),M^{J_2(x,i)}\end{matrix}\right.\right) \le C_2 n^{-\frac{q_2}{2}}~~a.s.
 \]
 with $q_2 < \frac{1+\alpha\wedge 4}{2},C_2 > 0$ as above. \\
 Iterating this procedure, we find that for $k\ge1$, given the last attempt was not successful (i.e., we did not set $J_\text{good}(x,i) = J_k(x,i)$), then the next try is independent from the last try and successful with probability at least $\delta_1\cdot\delta_2$, with $\delta_1,\delta_2$ as in Lemma \ref{lem:front_above} and Lemma \ref{lem:walk_above}. In particular, the number of needed attempts is stochastically dominated by a geometric distribution with success probability $\delta_1\cdot\delta_2$. Also, as reasoned above, the $k+1$-th attempt costs a random number of vertices $B_{k+1}(x,i)-B_{k}(x,i)$, which is independent of the $k$-th attempt and satisfies
 \[
 \begin{split}
\mathbb{P}&\left(n< B_{k+1}(x,i) - B_{k}(x,i) < \infty\left\vert J_\text{good}(x,i) \neq J_k(x,i), B_k(x,i)\right.\right) \\
&\le C_1 n^{-q_1} + C_2 n^{-\frac{q_2}{2}} ~~~a.s.
\end{split}
 \]
Arguing the same way as we did in the proof of Lemma \ref{lem:M_tail}, this yields that
 \[
 \mathbb{P}(M^{J_\text{good}(x,i)} - x > n) \le C_3 n^{-q_3}
 \]
 for some $q_3 < \frac{1+\alpha \wedge 4}{4}$ and a $C_3 > 0$. \\
 We now need to show the uniform limit
 \[
\lim_{n\to\infty}\sup_{j\ge0}\mathbb{P}\left(\sup_{0\le x\le j, 1\le i \le A_x} M^{J_\text{good}(x,i)} > j+n\right) = 0.
 \]
 By assumption $\alpha > 3$, and thus we can choose some $\varepsilon_A$ with $\mathbb{E}[A^{\frac{4}{4\wedge\alpha - 3}}]<\infty$ and \[ \vartheta \in \left(\frac{(4\wedge \alpha)-3}{4+\varepsilon_A((4\wedge \alpha)-3)},\frac{(4\wedge\alpha)-3}{4}\right)\text{ and }q_3\in\left(\vartheta+1,\frac{(4\wedge\alpha) +1}{4}\right).\] We now compute for any $0\le x \le j$ that
 \[
 \begin{split}
 \mathbb{P}&\left(\sup_{1\le i\le A_x}M^{J_\text{good}(x,i)} > j+n\right) \\& \le \mathbb{P}\left(A > (j-x+n)^\vartheta\right) + \mathbb{P}\left(\bigcup_{i=1}^{(j-x+n)^\vartheta}\left\{M^{J_\text{good}}-x > j-x+n\right\}\right) \\
 &\le \mathbb{P}(A > (j-x+n)^\vartheta) + (j-x+n)^\vartheta C_3\left(j-x+n\right)^{-q_3}.
 \end{split}
 \]
 By assumption on $A$, we can fix some $q_4 > 1$ such that $\mathbb{E}[A^{\frac{q_4}{\vartheta}}] <\infty$ and thus obtain 
 \[
\sup_{j\ge0} \sum_{0\le x\le j}\mathbb{P}(A > (j-x+n)^\vartheta) \le \sum_{x\ge0}\mathbb{P}(A > (x+n)^\vartheta) \le C_4 n^{1-q_4}
 \]
 for some constant $C_4$ independent of $j$ and $n$.
 Using the definition of $\vartheta$, we have $\vartheta-q_3 < -1$, and we see that for some constant $C_5$ independent of $j$ and $n$, we have
 \[
 \sup_{j\ge0}\sum_{0\le x\le j} (j-x+n)^\vartheta(j-x+n)^{-q_3} \le C_5 n^{1+\vartheta -q_3}.
 \]
 In particular, we define
 \[
R_1(j) := \sup_{\substack{0\le x\le j\\1\le i \le A_x}} M^{J_\text{good}(x,i)}~\text{ and }~ J^1_\text{good}(j) := \sup_{\substack{0\le x\le j\\1\le i \le A_x}} J_\text{good}(x,i),
 \]
 observing that $R_1(j) = M^{J^1_\text{good}(j)}$. By construction, the accumulated sums of $(\overline{\nu}_k^{J^1_\text{good}(j)})_{k\ge1}$ grow at most linearly with slope $\frac{1}{\lambda}$ and, consequently, the front stays above $R_1(j) + \lfloor \lambda (t-T^{J_\text{good}^1(j)})\rfloor$ for $t\ge T^{J_\text{good}^1(j)}$ and none of the parasites born between $0$ and $j$ catch up to the linearly moving front after $T^{J^1_\text{good}(j)}$. Also, we have that 
 \[
\sup_{j\ge0}\mathbb{P}(R_1(j)-j > n) \le C_4n^{1-q_4}+C_5n^{1+\vartheta - q_3} \overset{n\to\infty}{\to}0. 
 \]
 Next, we want to see that the parasites born between $j$ and $R_1(j)$ have a positive probability to also stay below the linearly moving front after $T^{J_\text{good}^1(j)}$. We note that the paths of parasites generated between $j$ and $R_1(j)$ are influenced by the event that there possibly are already good sites $M^\ell < R_1(j)$ from which the accumulated sums of $(\overline{\nu}^\ell_k)_{k\ge1}$ grow at most with slope $\frac{1}{\lambda}$. This event does occur if for some $(x,i)$ with $0\le x\le j,1\le i\le A_x$, we have $J_\text{good}(x,i) < J_\text{good}^1(j)$. However, knowing this event occurs only affects the paths of parasites until the time when these parasites reach $R_1(j)$. To see this, observe that, by definition, we have $\overline{\nu}^\ell_{M^{\ell+1}-M^\ell + k} = \overline{\nu}^{\ell+1}_k$ for any $\ell\ge0,k\ge1$. Hence, the event that the accumulated sums of $(\overline{\nu}_k^{J^1_\text{good}(j)})_{k\ge1}$ grow at most linearly with slope $\frac{1}{\lambda}$ already implies that also for any $\ell < J^1_\text{good}(j)$ such that the accumulated sums of $(\overline{\nu}_k^{\ell})_{1\le k < R_1(j)-M^\ell}$ grow at most linearly with slope $\frac{1}{\lambda}$, already all the accumulated sums of $(\overline{\nu}_k^{\ell})_{k\ge1}$ grow at most linearly with slope $\frac{1}{\lambda}$. In particular, note that the event that the sums of $(\overline{\nu}_k^{\ell})_{1\le k < R_1(j)-M^\ell}$ grow at most linearly with slope $\frac{1}{\lambda}$ depends only on the paths of parasites until they reach $R_1(j)$, and the event that the accumulated sums of $(\overline{\nu}_k^{J^1_\text{good}(j)})_{k\ge1}$ grow at most linearly with slope $\frac{1}{\lambda}$ depends only on parasites born at $x$ with $x\ge R_1(j)$. In conclusion, this implies that after the time when a parasite generated between $j$ and $R_1(j)$ reaches $R_1(j)$, its path is that of a (unconditioned) simple symmetric random walk and, consequently, has a probability of at least $\delta_2$ to never catch up to a linearly moving line with slope $\frac{\lambda}{2}$. \\
 We fix some large enough constant $L\in\N$, such that for some $\delta_3 \in (0,1)$ we have
 \[
 C_4L^{1-q_4} + C_5L^{1+\vartheta-q_3} < 1-\delta_3.
 \]
 Then as computed above, we have $\sup_{j\ge0}\mathbb{P}(R_1(j) - j > L) < 1-\delta_3$. In the event $R_1(j) -j \le L$ and 
 \[
\sum_{k=0}^{R_1(j)-j-1} A_{j+k} \le 2(\mathbb{E}[A]+4) L ,
 \]
 as argued above and using Lemma \ref{lem:walk_above}, with probability at least $\delta_2^{2(\mathbb{E}[A]+4) L}$, none of the parasites generated between $j$ and $R_1(j)$ catch up to the linearly moving front after $T^{J^1_\text{good}(j)}$. We note that we can find some $\delta_4 > 0$ such that
 \[
\sup_{j\ge0}\mathbb{P}\left(\left.\sum_{k=0}^{R_1(j) -1}A_{j+k} > 2(\mathbb{E}[A]+4)L\right\vert R_1(j)-j \le L\right) < 1-\delta_4. 
 \]
 Hence, with probability at least $\delta_2^{2(\mathbb{E}[A]+4)L}\delta_3\delta_4$, our search for the site $R_\text{good}(j)$ ends at $R_\text{good}(j) = R_1(j)$, and none of the parasites born between $j$ and $R_1(j)$ catch up to the linearly moving front after $T^{J^1_\text{good}(j)}$. \\
 If $R_1(j) -j > L$ or too many parasites were generated between $j$ and $R_1(j)$, we can repeat the procedure to find $R_1(j)$, now starting at $R_1(j)$ instead of $j$. Thereby, we find a vertex $R_2(j) = M^{J_\text{good}^2(j)}$ such that none of the parasites born between $j$ and $R_1(j)$ (hence also those generated between $0$ and $R_1(j)$) catch up to the linearly moving front after $T^{J^2_\text{good}(j)}$ and we have
 \[
\lim_{n\to\infty}\sup_{j\ge0}\mathbb{P}(R_2(j) -j > n) = 0 .
 \]
 Checking once more if $R_2(j) - R_1(j) \le L$ and 
 \[
 \sum_{k=0}^{R_2(j)-R_1(j)-1} A_{R_1(j)+k} \le 2(\mathbb{E}[A]+4) L,
 \]
 we see that with probability at least $\delta_2^{2(\mathbb{E}[A]+4) L}\delta_3\delta_4$, also the parasites born between $R_1(j)$ and $R_2(j)$ do not contribute to the front after $T^{J^2_\text{good}(j)}$. Iterating this procedure, we find an almost surely finite site $R_\text{good}(j) = M^{J_\text{good}(j)} > j$ such that none of the parasites born between $0$ and $R_\text{good}(j)$ contribute to the front after time $T^{J_\text{good}(j)}$, the sums of $(\overline{\nu}_k^{J_\text{good}(j)})_{k\ge1}$ grow at most with slope $\frac{1}{\lambda}$, and
 \[
\lim_{n\to\infty}\sup_{j\ge0}\mathbb{P}(R_\text{good}(j) > j+n) = 0. \qedhere
 \]
 \end{proof}
 With this renewal site $R_\text{good}(j)$ that has a uniform tail over $j$, we can now iteratively construct a sequence of renewal sites $(R^i)_{i\ge0}$ simply by retrying for the next renewal site with $R_\text{good}(\cdot)$.
 \begin{proof}[Proof of Proposition \ref{prop:renewal_sites}]
 We set $R^0 = 0$ and for $i\ge 0$ define
 \[
 R^{i+1} := R_\text{good}(R^i) > R^i,
 \]
 with $R_\text{good}(\cdot)$ as in Lemma \ref{lem:renewalsite}. Using the uniform tail bound in Lemma \ref{lem:renewalsite}, we obtain that for any $i\ge0$, we have that $R^{i+1} <\infty$ almost surely, and by construction, none of the parasites generated below $R^{i+1}$ contribute to the front after $\rho_{R^{i+1}}$. For each $i\ge0$ and $\ell\ge 0,m > 0$, the event $\{R^{i+1}-R^i = m,R^i = \ell\}$ can be expressed as
 \[
 G_{[0,\ell)}\cap G_{[\ell,\ell+m)}\cap G_{[\ell+m,\infty)}
 \]
 for some events
 \[
 \begin{split}
 G_{[\ell,\ell+m)} &\in \sigma(A_x,I_{x+1},Y^{x,i}:\ell \le x < \ell + m,i\in\N) \\
 G_{[0,\ell)} &\in \sigma(A_x,I_{x+1},Y^{x,i}:0 \le x < \ell, i\in\N)
 \end{split}
 \]
 and
 \[
 G_{[\ell+m,\infty)} = \bigcap_{k=1}^\infty\left\{K_{\ell+m+k}\le k,\sum_{j=1}^k\nu_{\ell+m+j} \le \frac{k}{\lambda}\right\}\in\sigma\left(\begin{matrix}A_x,I_{x+1},\\Y^{x,i}\end{matrix}: \begin{matrix}x \ge \ell + m,\\i\in\N\end{matrix}\right).
 \]
 Note in particular that these three events are independent. By construction of the renewal sites, we have that for any $B\in\mathcal{B}(\R^m)$ the event
 \[
 \{R^{i+1}-R^i = m,R^i = \ell\}\cap\left\{\left(\rho_{k+1}-\rho_k:\ell \le k < \ell+m\right)\in B\right\}
 \]
 can be expressed as
 \[
 G_{[0,\ell)}\cap \widetilde{G}_{[\ell,\ell+m)}\cap G_{[\ell+m,\infty)}
 \]
 with $G_{[0,\ell)}, G_{[\ell+m,\infty)}$ as above and some event
 \[
 G_{[\ell,\ell+m)}\supset\widetilde{G}_{[\ell,\ell+m)} \in \sigma(A_x,I_{x+1},Y^{x,i}:\ell \le x < \ell + m,i\in\N).
 \]
 In particular, this implies the independence of
 \[
 \{(R^{i+1}-R^i,\rho_{R^i+1}-\rho_{R^i},\dots, \rho_{R^{i+1}}-\rho_{R^{i+1}-1}): i\ge 0\},
 \]
 and also that the subcollection with $i\ge1$ is identically distributed.
 \end{proof}
 \section{Proofs of the main Theorems}
 \label{Sec:ProofsOfTheorems2to4}
 \subsection{Proof of Theorem \ref{Theorem: lln}}
 In this section, we will show that the process satisfies a strong law of large numbers, in the sense that 
 \[\lim_{t\to\infty}\frac{r_t}{t} = \gamma~~~\text{almost surely}\]
 for some deterministic $\gamma > 0$. We recall the approach. \\
In Section \ref{sec:goodsites} we identified a sequence $(M^i)_{i\ge0}\subset \Z$, where the upcoming births and deaths are in a good configuration. In particular, the sequence only depends on $\mathbf{A}$ and $\mathbf{I}$, and the increments $(M^{i+1}-M^i)_{i\ge0}$ are i.i.d.~\\
The good state of upcoming births and deaths after each $M^i$ allowed us to construct a lower bound for the front in Definition \ref{Msites}, only depending on parasites that are born to the right of $M^i$. In particular, in Lemma \ref{lem:T_Tfinexp}, using Lemma \ref{Qconvspeed}, we will be able to show that moving from $M^i$ to $M^{i+1}$ has finite expectation under our assumptions on $\alpha$. \\
 In Lemma \ref{lem:X_mn_conv} we will show a subadditivity property for the times of moving to each site $(M^i)_{i\ge0}$ conditionally on the event that $\{M^0 = 0\}$, which yields a strong law of large numbers for the arrival times at the sites $(M^i)_{i\ge0}$, conditionally on $\{M^0=0\}$. Since the increments of $(M^i)_{i\ge0}$ are i.i.d., this yields also a strong law of large numbers for the infection times of any site $n$, conditionally on the event that $\{M^0 = 0\}$.\\
 However, by Proposition \ref{prop:renewal_sites}, the increments $(\rho_{n+1}-\rho_n)_{n\ge0}$ have a $\mathbb{P}$-trivial tail sigma field and thus the limit of $\frac{\rho_n}{n}$ cannot be random. In particular, this will allow us to transfer the strong law of large numbers from the event $\{M^0 = 0\}$ to the entire event of survival.\\ 
 The key in showing the subadditivity is to use a different coupling to construct the process, which is monotone in the initial configuration; see Proposition \ref{prop:untaged_mon}, in contrast to the construction that we used to define the auxiliary jump times; see also Example \ref{unmonontonecoupl}. In the construction, we will use a classical graphical representation using Poisson point processes to sample the jumps that happen on each site. We refer to Section \ref{Construction:unt} for the details. In this construction, the state of the process will consist of a triple $(r,\eta,\iota)$, where $r\in\Z$ is the position of the front, $\eta\in\N_0^{\Z}$ with $\supp \eta \subset(-\infty,r]$ is the number of living parasites on each site, and $\iota\in\N$ is the remaining immunity of the host at the site $r+1$. In this notation we define the natural initial configurations
\[
\zeta(r) := (r,A_r\delta_r,I_{r+1})
\]
of starting with only $A_r$ many parasites at the vertex $r$. For any initial configuration $\zeta = (r,\eta,\iota)$ as above and any time $t_0 \in\R$, we denote by
\[
(r_t(\zeta;t_0),\eta_t(\zeta;t_0),\iota_t(\zeta;t_0)_{t\ge t_0}
\]
the state of the SIMI at time $t$, when starting in the initial configuration $\zeta$ and using the Poisson point processes starting at time $t_0$. We will use that by Proposition \ref{prop:untaged_mon}, the coupling is such that for any initial time $t_0$, any $r\in\Z$, and configurations $\eta_1,\eta_2\in\N_0^{\Z}$ with $\supp\eta_1,\supp\eta_2 \subset (-\infty,r]$, the following holds. If $t_0$ is a stopping time with respect to the information of the collections $\mathbf{A},\mathbf{I}$ and the Poisson point processes up to time $t$, and $\eta_1,\eta_2$ are measurable with respect to the information of $\mathbf{A},\mathbf{I}$ and the Poisson point processes up to time $t_0$, such that almost surely $\eta_2$ holds more parasites than $\eta_1$ in each interval around the front, then the front constructed from $\eta_2$ starting at time $t_0$ is always above the front constructed from $\eta_1$ starting at time $t_0$. In this coupling we define the following random variables, on which we want to apply Liggett's Subadditive Ergodic Theorem \cite{L1985}.
\theoremstyle{definition}\begin{definition}\label{def:X_mn_Ligg}
In the coupling defined in Section \ref{Construction:unt}, we define
\[
\widetilde{T}^i := \inf\{t\ge0: r_t(\zeta(0);0) \ge M^i\}
\] and for $0\le m < n$ we set
\[
X_{m,n} := \inf\{t\ge 0: r_{\widetilde{T}^m+t}(\zeta(M^m);\widetilde{T}^m) \ge M^n\}.
\]
\end{definition}
\begin{remark}
 We note that $X_{0,n}$ always starts at the site $M^0$. Hence, only in the event $\{M^0 = 0\}$ do we have that $X_{0,n} = \widetilde{T}^n$ for $n\ge 0$. Also we note that by the definition of $(M^i)_{i\ge0}$, the process started from $M^i$ cannot die out, and thus $X_{m,n} < \infty$ almost surely for any $0\le m < n$.
\end{remark}
\theoremstyle{plain}\begin{lemma}\label{lem:X_mn_conv}
With $(X_{m,n})_{0\le m < n}$ as in Definition \ref{def:X_mn_Ligg}, we have that for some deterministic $\rho > 0$
 \[
 \lim_{n\to\infty}\frac{X_{0,n}}{n} = \rho
 \]
 almost surely. In particular, we have
 \[
 \lim_{n\to\infty}\frac{\widetilde{T}^n}{n} = \rho
 \]
 almost surely on the event $\{M^0=0\}$.
\end{lemma}
\begin{proof}
 We want to apply Liggett's Subadditive Ergodic Theorem. Since clearly $X_{m,n} \ge 0$ for all $0\le m<n$, we thus have to check the four conditions:
 \begin{itemize}
 \item[a)]\label{Ligga} $\mathbb{E}[X_{0,n}] < \infty$ for all $n\in\N$
 \item[b)]\label{Liggb} $X_{0,n}\le X_{0,m} + X_{m,n}$ for all $0\le m < n$.
 \item[c)]\label{Liggc} The joint distributions of $\{X_{m+1,m+k+1}:k\ge 1\}$ are the same as those of $\{X_{m,m+k}:k\ge 1\}$ for each $m\ge 0$.
 \item[d)]\label{Liggd} For each $k\ge 1$, the process $\{X_{nk,(n+1)k}:n\ge1\}$ is stationary and ergodic.
 \end{itemize}
 We now verify each condition.
 \begin{itemize}
 \item[a)] This follows from Lemma \ref{lem:T_Tfinexp} and property b), since clearly $T^1 \overset{\diff}{=} \widetilde{T}^1$.
 \item[b)] The configuration $\eta_{\widetilde{T}^m}(\zeta(M^0);\widetilde{T}^0)$ holds $A_{M^m}$ parasites on site $M^m$ and then all the remaining parasites that were generated between $M^0$ and $M^m$ and are still alive. In particular, the configuration $\eta_{\widetilde{T}^m}(\zeta(M^0);\widetilde{T}^0)$ dominates the configuration $\zeta(M^m)$, which holds only $A_{M^m}$ parasites on site $M^m$. By Proposition \ref{prop:untaged_mon} and the strong Markov property of the process, this implies the subadditivity. We note here that for this to hold, we need to enlarge the state space of triples $(r,\eta,\iota)$ to also hold the information of all upcoming immunities and offspring, because the offspring and immunity configurations above $M^m$ are always good and thus influence the future of the process. The state of the process will be given by a 4-tuple $(r,\eta,(\iota_n)_{n\ge1},(a_n)_{n\ge1})$ where $\iota_n$ is the immunity of the host at $r+n$ and $a_n$ are the offspring generated after the host at $r+n$ gets infected. Clearly the resulting process is still a strong Markov process with respect to the information of $\mathbf{A},\mathbf{I}$ and the Poisson point processes up to time $t$, and clearly $\widetilde{T}^m$ is a stopping time with respect to this filtration.
 \item[c)]This is clearly the case by translation invariance.
 \item[d)] The stationarity follows again by translation invariance. To see the ergodicity, we note that $\widetilde{T}^m$ is a stopping time with respect to the information of $\mathbf{A},\mathbf{I}$ and the Poisson point processes up to time $t$. Also, we note that by Lemma \ref{lem:M_tail} we have
 \[
 \{(M^{m+1}-M^m,A_{M^m},\dots,A_{M^{m+1}-1},I_{M^m+1},\dots,I_{M^{m+1}}):m\ge0 \}
 \]
 is i.d.d.~and, by definition, the process between $M^m$ and $M^n$ only uses the offspring given by $A_{M^m},\dots,A_{M^n-1}$, immunities given by $I_{M^m+1},\dots,I_{M^n}$, and the Poisson point processes between $\widetilde{T}^m$ and $\widetilde{T}^m + X_{m,n}$. Hence, the variables
 \[
 (X_{nk,(n+1)k})_{n\ge1}
 \]
 use different parts of the collections $\mathbf{A},\mathbf{I}$ and Poisson point processes if they their indices are far enough apart and, in particular, are an ergodic sequence. 
 \end{itemize}
Having checked all the conditions, we can apply Liggett's Subadditive Ergodic Theorem and conclude the claimed almost sure convergence. Noting that, by definition, on the event $\{M^0=0\}$ we have $X_{0,n} = \widetilde{T}^n$ yields the second claim.
\end{proof}

We have established a law of large numbers for the hitting times of the good sites. To show the claimed law of large numbers for the front, we need the next proposition to transfer a law of large numbers of hitting times to a law of large numbers of the corresponding counting process.
\begin{proposition}\label{auxiliaryfrontconvergence}
Let $(X_n)_{n\ge1}$ be a sequence of non-negative random variables such that there is a $\lambda > 0$ with
\[
\lim_{n\to\infty}\frac{1}{n}\sum_{j=1}^n X_j = \lambda
\]
almost surely. Then
\[
\lim_{t\to\infty}\frac{1}{t}\sup\left\{n\ge 1: \sum_{j=1}^n X_j \le t\right\} = \frac{1}{\lambda}
\]
almost surely.
\end{proposition}
\begin{proof}
We set
\[
m_t := \sup\left\{n\ge 1: \sum_{j=1}^n X_j \le t\right\}.
\]
Let $\varepsilon > 0$ and let $\omega$ be such that 
\[
\lim_{n\to\infty}\frac{1}{n}\sum_{j=1}^nX_j(\omega) = \lambda.
\]
We choose $n_0= n_0(\omega)\in\N$ large enough such that 
\[
\left\vert\frac{1}{n}\sum_{k=1}^n X_k(\omega) - \lambda\right\vert < \varepsilon~\text{ and }~\frac{n+1}{n} < 1+\varepsilon
\]
for all $n\ge n_0$. Set $t_0 =t_0(\omega):= \sum_{k=1}^{n_0} X_k(\omega)$. Clearly for all $t\ge t_0$ we have $m_t(\omega) \ge n_0$, thus we have
\[
\frac{\sum_{k=1}^{m_t}X_k}{m_t},\frac{\sum_{k=1}^{m_t+1}X_k}{m_t+1} \in (\lambda-\varepsilon,\lambda+\varepsilon).
\] 
and by definition of $m_t$ we have
\[
\sum_{k=1}^{m_t}X_k \le t < \sum_{k=1}^{m_t+1}X_k.
\]
Dividing by $m_t$, we obtain
\[
\lambda-\varepsilon< \frac{\sum_{k=1}^{m_t}X_k}{m_t}\le \frac{t}{m_t} < \frac{m_t+1}{m_t}\frac{\sum_{k=1}^{m_t+1}X_k}{m_t+1} < (1+\varepsilon)(\lambda+\varepsilon)
\]
and hence $\lim_{t\to\infty}\frac{m_t}{t} = \frac{1}{\lambda}>0$.
\end{proof}

\begin{proof}[Proof of Theorem \ref{Theorem: lln}]
By Lemma \ref{lem:X_mn_conv} we have that
\[
\lim_{n\to\infty}\frac{T^n}{n} = \rho
\]
for some $\rho \in (0,\infty)$ holds almost surely on the event $\{M^0 = 0\}$.
Since the sequence of increments $(M^{k+1}-M^k)_{k\ge0}$ is i.i.d.~under $\mathbb{P}$, the strong law of large numbers implies that $\mathbb{P}$-almost surely
\[
\lim_{n\to\infty}\frac{M^n}{n} = \mathbb{E}[M^1-M^0].
\]
For $n\in\N$ we set
\[
Q_n := \sup\{m\ge0: M^m \le n\},
\]
then by Proposition \ref{auxiliaryfrontconvergence} we have
\[
\lim_{n\to\infty}\frac{Q_n}{n} = \frac{1}{\mathbb{E}[M^1-M^0]}
\]
$\mathbb{P}$-almost surely. Hence, on the event $\{M^0 = 0\}$ we almost surely have
\[
\frac{\rho_n}{n} = \frac{T^{Q_n}}{Q_n}\frac{Q_n}{n} + \frac{\rho_n-T^{Q_n}}{n} \overset{n\to\infty}{\to} \frac{\rho}{\mathbb{E}[M^1-M^0]}.
\]
To see that the second summand
\[
\frac{\rho_n-T^{Q_n}}{n}
\]
tends to $0$, simply note that by the definition of $Q_n$ we have \[0\le \rho_n - T^{Q_n} < T^{Q_n+1} - T^{Q_n}.\] Now we simply compute, using the estimates in the proof of Lemma \ref{lem:T_Tfinexp}, that for any $\varepsilon > 0$
\[
\mathbb{P}(T^{Q_n+1}-T^{Q_n} > \varepsilon n) \le \mathbb{P}(T^1 > \varepsilon n\vert M^0 = 0) \le Cn^{-(1+\delta)}
\]
for some suitable $C>0$ and $\delta > 0$ small enough. In particular, the Borel-Cantelli Lemma implies
\[
0\le \lim_{n\to\infty}\frac{\rho_n-T^{Q_n}}{n} \le \lim_{n\to\infty}\frac{T^{Q_n+1}-T^{Q_n}}{n} = 0
\]
$\mathbb{P}$-almost surely. With this we have shown that
\[
 \lim_{n\to\infty}\frac{\rho_n}{n} = \frac{\rho}{\mathbb{E}[M^1-M^0]}
\]
almost surely on the event $\{M^0 = 0\}$. By Proposition \ref{prop:renewal_sites} and Kolmogorov's zero-one law (see e.g. \cite[Theorem 2.37]{K2020}), we obtain that the tail-sigma field
\[
\bigcap_{i=0}^\infty\sigma(R^{k+1}-R^k,\rho_{R^k+1}-\rho_{R^k},\dots,\rho_{R^{k+1}}-\rho_{R^{k+1}-1}: k \ge i)
\]
is $\mathbb{P}$-trivial. Noting that
\[
\begin{split}
\bigcap_{i=0}^\infty& \,\sigma(R^{k+1}-R^k,\rho_{R^k+1}-\rho_{R^k},\dots,\rho_{R^{k+1}}-\rho_{R^{k+1}-1}: k \ge i) \\
&= \bigcap_{i=0}^\infty\sigma(R^{k+1}-R^k,\rho_{j+1}-\rho_j:k\ge i, j\ge R^i)
\end{split}
\]
and $\infty > R^i \ge i$ almost surely, yields that also
\[
\mathcal{H}_{\text{tail},\rho} := \bigcap_{n=1}^\infty\sigma(\rho_{k+1}-\rho_k:k\ge n) \subset \bigcap_{i=0}^\infty\sigma\left(R^{k+1}-R^k,\rho_{j+1}-\rho_j:\begin{matrix}k\ge i, \\j\ge R^i\end{matrix}\right)
\]
is $\mathbb{P}$-trivial. Because $\rho_n < \infty$ almost surely under $\mathbb{P}$, the event
\[
\left\{\lim_{n\to\infty}\frac{\rho_n}{n} = \frac{\rho}{\mathbb{E}[M^1-M^0]}\right\}
\]
is contained in that tail-sigma field $\mathcal{H}_{\text{tail},\rho}$ and, as reasoned above, has positive probability. Thus, we must already have
\[
\lim_{n\to\infty}\frac{\rho_n}{n} = \frac{\rho}{\mathbb{E}[M^1-M^0]}
\]
almost surely under $\mathbb{P}$. \\
Hence, again using Proposition \ref{auxiliaryfrontconvergence}, we obtain that
\[
\lim_{t\to\infty}\frac{r_t}{t} = \frac{\mathbb{E}[M^1-M^0]}{\rho}
\]
$\mathbb{P}$-almost surely and have finished the proof of Theorem \ref{Theorem: lln} by setting
\[
\gamma := \frac{\mathbb{E}[M^1-M^0]}{\rho} .\qedhere
\]
\end{proof} 
\subsection{Proof of Theorem \ref{Theorem: ballistic growth}}\label{subsec:ball}
In this section we prove the weaker result, which, however, works for a broader range of $\alpha$ and also any initial configuration $w$. Precisely, we want to show Theorem \ref{Theorem: ballistic growth}, i.e., that there are $C_1,C_2 > 0$ such that
 \[
 C_1\le \liminf_{t\to\infty}\frac{r_t}{t}\le\limsup_{t\to\infty}\frac{r_t}{t}\le C_2
 \]
 holds $\mathbb{P}$-almost surely.
The upper bound follows by Equation \eqref{frontupbound}. For the lower bound we use a quite similar construction as for $M^0$, however this time we don't give a special treatment to the first $k_0$ sites. We recall the definitions of $(L_k)_{k\ge0},N, M$ and $T$ from Definition \ref{Msite} and that $T<\infty$ holds $\mathbb{P}$-almost surely by Lemma \ref{lem:welldef}. By definition, we have 
\[
K_{M+n} \le n
\] 
for all $n\ge k_0$ and thus by Proposition \ref{Couplinglowbound} we have
\begin{equation}\label{eq:lowerboundthm2}
\mathds{1}_{t\ge T}\left(M+k_0-1+\sup\left\{n\ge1:\sum_{k=k_0}^n\nu_{M+k}\le t-T\right\} \right)\le r_t.
\end{equation}
In light of Proposition \ref{auxiliaryfrontconvergence}, we hence need to show
\begin{equation}\label{nuLLN}
\lim_{n\to\infty}\frac{1}{n}\left(T +\sum_{k=k_0}^{n}\nu_{M+k}\right) = \gamma
\end{equation}
for some $\gamma > 0$. Because $T<\infty$ almost surely under $\mathbb{P}$, we only need to show that
\[
\lim_{n\to\infty}\frac{1}{n}\sum_{k=k_0}^n\nu_{M+k} = \gamma
\]
for some $\gamma > 0$.
By Lemma \ref{Tdistconv}, the variables $(\nu_{M+n})_{n\ge k_0}$ are asymptotically identically distributed, and by Lemma \ref{conditinallyphimixing}, they have weak dependence. We will use $L^q$-mixingales introduced by \cite{M1975} to formally prove the claimed strong law of large numbers. We recall their definition next.
\theoremstyle{definition}\begin{definition}
Let $\{X_i:i\ge 1\}$ be a sequence of random variables on some probability space $(\Xi,\mathcal{H},P)$ with $E[\vert X_i\vert] < \infty$ for all $i\ge1$ and $\{\mathcal{H}_i:i\in\Z\}$ a nondecreasing sequence of sub $\sigma$-fields of $\mathcal{H}$. For $q\ge 1$, the collection $\{X_i,\mathcal{H}_j:i\ge 1,j\in\Z\}$ is called an $L^q$-mixingale if there exist nonnegative constants $\{c_i:i\ge 1\}$ and $\{\psi(m):m\ge 0\}$ such that $\lim_{m\to\infty}\psi(m) = 0$ and for all $i\ge 1,m\ge 0$ we have
 \begin{itemize}
 \item $\lVert E[X_i\vert \mathcal{H}_{i-m}]\rVert_q := \left(E\left[\left\vert E[X_i\vert \mathcal{H}_{i-m}]\right\vert^q\right]\right)^{\frac{1}{q}}\le c_i \psi(m)$ and
 \item $\lVert X_i -E[X_i\vert \mathcal{H}_{i+m}]\rVert_q := \left(E\left[\left\vert X_i -E[X_i\vert \mathcal{H}_{i+m}]\right\vert^q\right]\right)^{\frac{1}{q}}\le c_i \psi(m+1)$ 
 \end{itemize}
\end{definition}
We note that the first condition imposes the variables $\{X_i:i\ge 1\}$ to have zero mean, because
\[
\vert E[X_i]\vert = \vert E[E[X_i\vert \mathcal{H}_{i-m}]]\vert \le \lVert E[X_i\vert \mathcal{H}_{i-m}]\rVert_1 \le \lVert E[X_i\vert \mathcal{H}_{i-m}]\rVert_q \le c_i\psi(m) \to 0
\]
as $m\to\infty$. We will thus verify the following lemma.
\theoremstyle{plain}\begin{lemma}\label{Tkmixingale}
Set
\[
X_i := \nu_{M+k_0-1+i}-\mathbb{E}[\nu_{M+k_0-1+i}]~\text{ and }~\mathcal{H}_j := \begin{cases}\sigma(\nu_{M+k_0},\dots, \nu_{M+k_0+j-1}),&j\ge 1\\ \{\emptyset,\Omega\},&j \le 0\end{cases}
\]
for $i\ge 1,j\in\Z$. Then for any $1 < q < \frac{(4\wedge\alpha)+1}{2}$ the collection $\{X_i,\mathcal{H}_j:i\ge1,j\in\Z\}$ is an $L^q$-mixingale on the probability space $(\Omega,\mathcal{F},\mathbb{P})$. In particular we have
\[
c_i = 2\lVert \nu_{M+k_0-1+i}-\mathbb{E}[\nu_{M+k_0-1+i}]\rVert_q~~\text{ and }~~\psi(m) = \phi(m)^{\frac{q-1}{q}}
\]
with $\phi$ as in Lemma \ref{conditinallyphimixing}.
\end{lemma}
\begin{proof}
 First we note that by Corollary \ref{cor:overnufinmom} the expectations and conditional expectations exist and are finite. Next we note that the second condition of a $L^q$-mixingale is trivially fulfilled because $X_i$ is $\mathcal{H}_i$ measurable; hence the left side is just $0$. To verify the first condition we use the classical result for $\phi$-mixing sequences of Serfling \cite[Theorem 2.2]{S1968} that shows
 \[
\mathbb{E}\left[\left\vert\mathbb{E}[X_i\vert \mathcal{H}_{i-m}\right\vert^q\right]\le 2^q\phi(m)^{q-1}\mathbb{E}[\vert X_i\vert^q].
\]
This already finishes the proof.
\end{proof}
To conclude the final convergence, we need the following lemma to deal with the error term that appears due to centering the not identically distributed variables $(\nu_{M+k_0-1+i})_{i\ge1}$.
\theoremstyle{plain}\begin{lemma}\label{lem:nullaverage}
Let $(a_n)_n\subset \R$ be a sequence with $\lim_{n\to\infty}a_n = 0$. Then
 \[
 \lim_{n\to\infty}\frac{1}{n}\sum_{k=1}^n a_n = 0
 \]
\end{lemma}
\begin{proof}
Let $\varepsilon > 0$ and choose $N_1\in\N$ such that $\vert a_n\vert < \min\left\{1,\frac{\varepsilon}{3}\right\}$ for all $n > N_1$. Then in particular for all $n > N_1$ we have
 \[
 \frac{1}{n-N_1}\left\vert\sum_{k = N_1+1}^n a_k\right\vert < \min\left\{1,\frac{\varepsilon}{3}\right\}.
 \]
Setting
 \[
 N_2 := \max\left\{N_1,\left\lceil\frac{3}{\varepsilon}\left\vert\sum_{k=1}^{N_1}a_k\right\vert\right\rceil,\left\lceil N_1\frac{3}{\varepsilon}\right\rceil\right\}
 \]
 we obtain for $n>N_2$
 \[
 \begin{split}
 \frac{1}{n}\left\vert\sum_{k=1}^na_k\right\vert &= \frac{1}{n}\left\vert\sum_{k=1}^{N_1}a_k + \sum_{k=N_1+1}^n a_k\right\vert =\frac{1}{n}\left\vert\sum_{k=1}^{N_1}a_k + \frac{n-N_1}{n-N_1}\sum_{k=N_1+1}^n a_k\right\vert \\
 &\le \frac{1}{n}\left\vert\sum_{k=1}^{N_1}a_k\right\vert +\frac{1}{n}\frac{N_1}{n-N_1} \left\vert\sum_{k=N_1+1}^n a_k\right\vert +\frac{1}{n-N_1}\left\vert\sum_{k=N_1+1}^na_k\right\vert \\
 & < \left\lceil\frac{3}{\varepsilon}\left\vert\sum_{k=1}^{N_1}a_k\right\vert\right\rceil^{-1}\left\vert\sum_{k=1}^{N_1}a_k\right\vert + \left\lceil N_1\frac{3}{\varepsilon}\right\rceil^{-1}N_1\cdot 1 + \frac{\varepsilon}{3} \le \varepsilon,
 \end{split}
 \]
 which finishes the proof.
\end{proof}
With these tools at hand, we can now finally prove the needed law of large numbers.
\theoremstyle{plain}\begin{lemma}\label{Tnconvergence}
 It holds
 \[
 \lim_{n\to\infty} \frac{T+\sum_{k=k_0}^n \nu_{M+k}}{n} = \mathbf{E}[\nu]
 \]
 $\mathbb{P}$-almost surely.
\end{lemma}
\begin{proof}
Using Lemma \ref{lem:welldef}, we obtain that
 \[
 \lim_{n\to\infty}\frac{T}{n} = 0
 \]
 $\mathbb{P}$-almost surely. \\
 Now, fixing some $q\in \left(1,\frac{(4\wedge\alpha)+1}{2}\right)$, we want to apply \cite[Corollary 1]{DJ1995} to conclude the convergence of 
 \[
 \lim_{n\to\infty} \frac{1}{n}\sum_{k=k_0}^{k_0+n}\nu_{M+k}-\mathbb{E}[\nu_{M+k}] = 0.
 \]
 For this purpose, we have to check the following two conditions: we need to have $\sup_{i\ge1} c_i < \infty$ and $\psi(m) =\mathcal{O}\left(\log(m)^{-(1+\delta)}\right)$ for some $\delta > 0$.\\
 Using the constants given in Lemma \ref{Tkmixingale} and applying Corollary \ref{cor:momentconvergence}, we see that
 \[
 \begin{split}
 c_i &\le 2\left(\lVert \nu_{M+k_0-1+i}\rVert_q + \lVert\mathbb{E}[\nu_{M+k_0-1+i}]\rVert_q\right) \\ &= 2\left(\mathbb{E}[\vert \nu_{M+k_0-1+i}\vert^q]^{\frac{1}{q}} +\mathbb{E}[\nu_{M+k_0-1+i}]\right) \\&\to 2\left(\mathbf{E}[\nu^q]^{\frac{1}{q}}+\mathbf{E}[\nu]\right)~~~(i\to\infty).
 \end{split}
 \] 
 In particular, $\sup_{i\ge 1}c_i<\infty$, which is the first condition of \cite[Corollary 1]{DJ1995}. From Lemma \ref{conditinallyphimixing} we obtain that \[\psi(m) \le C m^{(-\alpha+\varepsilon)\frac{q-1}{q}} = \mathcal{O}\left(\frac{1}{\log(m)^2}\right),\]
 which is the second condition of \cite[Corollary 1]{DJ1995}. Hence we can apply the strong law of large numbers for $L^q$-mixingales to conclude
\begin{equation}\label{eq:Tncenterconv}
\lim_{n\to\infty}\frac{1}{n}\sum_{j=1}^n \nu_{M+k_0-1+j}-\mathbb{E}[\nu_{M+k_0-1+j}] = 0.
\end{equation}
 Thus putting everything together we obtain
\[
\begin{split}
\lim_{n\to\infty}&\left\vert\frac{1}{n}\left(T+\sum_{k=k_0}^n \nu_{M+k}\right)-\mathbf{E}[\nu]\right\vert \le \lim_{n\to\infty}\left\vert\frac{T}{n}\right\vert+\left\vert\frac{1}{n}\sum_{k=k_{0}}^n \nu_{M+k}-\mathbf{E}[\nu]\right\vert \\
&= \lim_{n\to\infty}\left\vert\frac{n-k_{0}+1}{n}\frac{1}{n-k_{0}+1}\sum_{k=k_{0}}^n \nu_{M+k}-\mathbb{E}[\nu_{M+k}]+\mathbb{E}[\nu_{M+k}] - \mathbf{E}[\nu]\right\vert\\
&\le\lim_{n\to\infty}\left\vert\frac{1}{n-k_{0}+1}\sum_{k=k_{0}}^n \nu_{M+k}-\mathbb{E}[\nu_{M+k}]\right\vert \\ 
&\quad\quad\quad+\left\vert\frac{1}{n-k_0+1}\sum_{k=k_{0}}^n\mathbb{E}[\nu_{M+k}] - \mathbf{E}[\nu]\right\vert \\
&= 0,
\end{split}
\]
where in the last line we used \eqref{eq:Tncenterconv} and Lemma \ref{lem:nullaverage}, noting that by Corollary \ref{cor:momentconvergence} we have $\lim_{n\to\infty}\mathbb{E}[\nu_{M+n}]=\mathbf{E}[\nu]$.
\end{proof}
\begin{proof}[Proof of Theorem \ref{Theorem: ballistic growth}]
Using Lemma \ref{Tnconvergence} and Proposition \ref{auxiliaryfrontconvergence}, we obtain the lower bound. For $\vartheta >0 $ let $f_\vartheta$ be as in Theorem \ref{Theorem:well-def}. For the upper bound we note that by \eqref{frontupboundinfinite}, taking $\gamma > 0$ large enough such that $c_{\gamma,\vartheta} > 0$,
\[
\begin{split}
\mathbf{P}(\mathcal{S}(\eta))&\mathbb{P}(r_t > \gamma t) \le \mathbf{P}(r_t > \gamma t) = \int_\Omega\mathbf{P}(r_t(\eta(\omega)) > \gamma t)\diff\mathbf{P}(\omega) \\
&\le \int_\Omega\e^{-c_{\gamma,\vartheta} t}f_\vartheta(\eta(\omega))\diff\mathbf{P}(\omega) = \e^{-c_{\gamma,\vartheta}t}\mathbf{E}[f_\vartheta(\eta)] .
\end{split}
\]
By assumption the initial configuration satisfies $\mathbf{E}[f_\vartheta(\eta)] < \infty$, since otherwise the whole process is not even well defined. This concludes the proof, because this bound is clearly integrable in $t$.
\end{proof}
\subsection{Proof of Theorem \ref{Theorem: polynomial decay}}\label{subsec:proof_poldec}
We finally consider the SIMI on $\Z$, where initially all vertices $x\neq 0$ are inhabited by a host and there are a random number of active parasites, distributed according to $A$ at the origin. Since we start with only finitely many active parasites, we do not need to worry about the state space topology, as a coupling with a branching random walk again shows that the number of infected sites, hence the number of active parasites, grows at most linearly in time. However, since deaths can now occur at both boundaries, which carry their individual immunities, the event of survival can, a priori, no longer be expressed independent of the paths given by $\mathbf{Y}$. In the special case that $I$ has a geometric distribution with parameter $p\in(0,1]$, this, however, is still possible, because the infection probability does not depend on the number of previous attempts, i.e., $\mathbf{P}(I=k+1\vert I>k) = p$ for all $k\ge0$. In the general case, however, $\mathbf{P}(I=k+1\vert I>k)$ does depend on $k$, which will make a successful infection dependent on the number of parasites that died at each boundary. Using Lemma \ref{Uinf}, we can, however, get rid of this dependence by considering the event that parasites born on some $x>0$ never reach the left host boundary and parasites born on $x<0$ never reach the right host boundary. We recall the following classical result on the speed of continuous-time simple symmetric random walks on $\Z$:
\theoremstyle{plain}\begin{lemma}\cite[Lemma 8]{C2007}\label{RWsublin}
Let $\{X_t:t\ge0\}$ be a continuous time simple symmetric random walk on $\Z$ with jump rate $2$, starting from $x\le -1$ on some probability space $(\Omega^\prime,\mathcal{F}^\prime,P)$. For any $c>0$, let
\[
\tau_c := \inf\{t\ge0:X_t \ge \lfloor ct\rfloor\},
\]
then
\[
P(\tau_c = \infty) \ge \begin{cases}1-\e^{(1+x)\theta_c},&x\le -2\\\exp\left(-\frac{2}{c}\right)\left(1-\e^{-\theta_c}\right),&x=-1\end{cases}
\]
where $\theta_c >0 $ is the positive solution of $c\theta-2(\cosh\theta-1) = 0$.
\end{lemma}
\theoremstyle{definition}\begin{definition}
 For $j\ge1$ we set
 \[
 K_{n,j} := \inf\left\{k\ge j:\sum_{\ell = 1}^kA_{n-\ell} \ge \beta_A k,\sum_{\ell = 1}^kI_{n+1-\ell} \le \beta_I k\right\}.
\]
Analogously we define
\[
\mathcal{W}_{n,j} := \{(x,i)\in\Z\times \N:n-K_{n,j}\le x\le n-1, 1\le i\le A_{x}\}
\]
and
\[
\nu_{n,j} := \inf\left\{t\ge0:\sum_{(x,i)\in \mathcal{W}_{n,j}}\mathds{1}_{\tau^{x,i}_{n-x}\le t} \ge\beta_I K_{n,j}\right\}.
\]
Finally we define for any $l\in\Z,n\ge1$ the time
\[
\nu_n(l) := \min\{\nu_{l+n,1},\dots,\nu_{l+n,n\wedge k_0}\}.
\]
\end{definition}
\begin{remark}
 Analogously to Lemma \ref{Couplinglowbound}, we obtain that
 \[
 \mathds{1}_{K_{n,j}\le n}(\rho_n-\rho_{n-1}) \le \nu_{n,j} 
 \]
 for all $n\ge1, 1\le j \le n$. Clearly $K_{n,j} \le K_{n,j+1}$ and hence for any $l\in\Z,n\ge1$ also
 \[
\mathds{1}_{K_{l+n,n\wedge k_0}\le n}(\rho_{l+n}-\rho_{l+n-1})\le \nu_n(l). 
 \] Furthermore, the sequence $(\nu_n(l))_{n\ge k_0}$ is $\phi$-mixing and has finite $q$-th moments for any $q < \frac{\alpha + 1}{2}$ which follows analogously as Lemma \ref{NuMoments} and Lemma \ref{conditinallyphimixing}.
\end{remark}
\theoremstyle{plain}\begin{lemma}\label{Uinf}
Assume $\alpha > 3$. Then for any $\gamma < \frac{1}{\mathbf{E}[\nu_{1,k_0}]}$ we have
\[
 \mathbf{P}\left(\bigcap_{k=1}^\infty\{K_{k,k\wedge k_0}\le k\},\bigcap_{n=1}^\infty\left\{\sum_{k=1}^n\nu_k(0) \le \frac{n}{\gamma}\right\}\right) > 0.
\]
\end{lemma}
\begin{proof}
We recall
 \[
 m(X,k)= \inf\{j\ge k:\mathbf{P}(X= j)>0\}
 \]
 for any random variable $X$ on $\mathbf{\Omega}$. We let $\mu := \mathbb{E}[\nu_{1,k_0}]$ and $\gamma < \frac{1}{\mu}$. For $n\in\N,\varepsilon\in (0,\frac{1}{\gamma})$ we define $G_{n,\varepsilon}$ to be the event that \[I_1=\dots=I_n=m(I,1), A_0=\dots=A_{n-1}=m(A,\beta_A),\] and all random walks $\{Y^{x,i}:0\le x< n,1\le i \le m(A,\beta_A)\}$ take $n+k_0$ steps to the right before time $\varepsilon$. By definition of the event, we have
\[
G_{n,\varepsilon}\subset\bigcap_{k=1}^n\{K_{k,k_0} = k\wedge k_0,\nu_k(0)< \varepsilon\}
\]and clearly $\mathbf{P}(G_{n,\varepsilon})>0$ for all $n\in\N,\varepsilon \in(0,\frac{1}{\gamma})$. Next let $H_{n,\varepsilon}$ be the event
\[
\bigcap_{k=1}^\infty\left\{K_{n+k,k\wedge k_0} \le k, \sum_{j=1}^k\nu_{n+j,j\wedge k_0} < \frac{n+k}{\gamma} - n\varepsilon\right\},
\]
which is independent of $G_{n,\varepsilon}$. By the definitions of $G_{n,\varepsilon}$ and $H_{n,\varepsilon}$ we have
\[
\bigcap_{k=1}^\infty\{K_{k,k\wedge k_0}\le k\}\cap\bigcap_{n=1}^\infty\left\{\sum_{k=1}^n\nu_k(0) \le \frac{n}{\gamma}\right\} \supset G_{n,\varepsilon}\cap H_{n,\varepsilon},
\]
hence, using the independence of $G_{n,\varepsilon}$ and $H_{n,\varepsilon}$, we only need to show that $\mathbf{P}(H_{n,\varepsilon})>0$ for some $n\in\N,\varepsilon \in(0,\frac{1}{\gamma})$. Observe that
\begin{equation}\label{Hnpos}
\begin{split}
 \mathbf{P}(H_{n,\varepsilon}^C) &\le \mathbf{P}\left(\bigcup_{k=1}^\infty\{K_{n+k,k\wedge k_0}> k\}\right)+\mathbf{P}\left(\bigcup_{k=1}^{k_0}\left\{\sum_{j=1}^k\nu_{n+j,j}\ge \delta\left(\frac{n+k}{\gamma}-n\varepsilon\right)\right\}\right) \\
 &\quad+\mathbf{P}\left(\bigcup_{k=k_0+1}^\infty\left\{\sum_{j=k_0+1}^k\nu_{n+j,k_0}\ge (1-\delta)\left(n\left(\frac{1}{\gamma}-\varepsilon\right)+\frac{k-k_0}{\gamma}\right)\right\}\right)
\end{split}
\end{equation}
for any $\delta\in(0,1)$ such that $\frac{1-\delta}{\gamma}>\mu$.\\ The first term does not depend on $n$, because the random variables $(K_{n,j})_{n\in\Z}$ are identically distributed, and the complement of the event considered in the first term is contained in the event
\[
\bigcap_{k=1}^\infty\left\{\sum_{j=1}^k I_{n+j} \le \beta_I k,\sum_{j=1}^k A_{n+j -1}\ge \beta_A k\right\},
\]
which has positive probability. In particular, the first term of \eqref{Hnpos} is smaller than $1$.\\
For the second term of \eqref{Hnpos} we observe that
\[
\begin{split}
\mathbf{P}&\left(\bigcup_{k=1}^{k_0}\left\{\sum_{j=1}^k\nu_{n+j,j}\ge \delta\left(\frac{n+k}{\gamma}-n\varepsilon\right)\right\}\right) = \mathbf{P}\left(\bigcup_{k=1}^{k_0}\left\{\sum_{j=1}^k\nu_{j,j}\ge \delta\left(\frac{n+k}{\gamma}-n\varepsilon\right)\right\}\right),
\end{split}
\]
which tends to $0$ for $n\rightarrow\infty$, because $\nu_{1,1},\dots,\nu_{k_0,k_0} < \infty$ almost surely by the same proof as Lemma \ref{NuMoments}. To estimate the last term, we note that by an analogous statement as claimed in Lemma \ref{Qconvspeed}, for any $2<q<\frac{\alpha+1}{2}$ we have
\[
\begin{split}
 \mathbf{P}&\left(\sum_{j=k_0+1}^k\nu_{n+j,k_0}\ge (1-\delta)\left(n\left(\frac{1}{\gamma}-\varepsilon\right)+\frac{k-k_0}{\gamma}\right)\right) \\
 &=\mathbf{P}\left(\sum_{j=1}^{k-k_0}\nu_{j,k_0}-\mu\ge (1-\delta)\left(n\left(\frac{1}{\gamma}-\varepsilon\right)+\frac{k-k_0}{\gamma}\right)-(k-k_0)\mu]\right) \\
 &\le C(k-k_0)^{\frac{q}{2}}\left((k-k_0)\left(\frac{1-\delta}{\gamma}-\mu\right)+n(1-\delta)\left(\frac{1}{\gamma}-\varepsilon\right)\right)^{-q}.
\end{split}
\]
Now for $\vartheta\in\left(\frac{q+2}{2q},1\right)$ Young's inequality yields
\[
\begin{split}
 &(k-k_0)\left(\frac{1-\delta}{\gamma}-\mu\right)+n(1-\delta)\left(\frac{1}{\gamma}-\varepsilon\right)\\ 
 = &\,\vartheta\frac{k-k_0}{\vartheta}\left(\frac{1-\delta}{\gamma}-\mu\right)+(1-\vartheta)\frac{n(1-\delta)}{1-\vartheta}\left(\frac{1}{\gamma}-\varepsilon\right)\\
 \ge &\left[\frac{k-k_0}{\vartheta}\left(\frac{1-\delta}{\gamma}-\mu]\right)\right]^{\vartheta}\left[\frac{n(1-\delta)}{1-\vartheta}\left(\frac{1}{\gamma}-\varepsilon\right)\right]^{1-\vartheta}
\end{split}
\]
and thus we can estimate the third term in \eqref{Hnpos} using the union bound with
\[
\begin{split}
 \sum_{k=k_0+1}^\infty C(k-k_0)^{\frac{q}{2}-\vartheta q}\left[\frac{1}{\vartheta}\left(\frac{1-\delta}{\gamma}-\mu]\right)\right]^{-q \vartheta}\left[\frac{n(1-\delta)}{1-\vartheta}\left(\frac{1}{\gamma}-\varepsilon\right)\right]^{q(\vartheta-1)}.
\end{split}
\]
By definition, $\frac{q}{2}-\vartheta q <-1$, hence the sum is finite, and clearly $n^{q(\vartheta-1)}$ goes to zero as $n$ tends to $\infty$. Hence for large enough $n$ the sum in \eqref{Hnpos} is less than $1$, which yields that $\mathbf{P}(H_{n,\varepsilon})>0$ and finishes the proof.
\end{proof}

\begin{proof}[Proof of Theorem \ref{Theorem: polynomial decay}]
 We will construct a subset $\mathcal{V}\subset\mathcal{S}$ of the event of survival, which has positive probability. On $\mathcal{V}$ we will have $A_0 = 2,I_1=I_{-1} = 1$, as well as linearly moving left and right fronts only using parasites born on the right or left side, respectively, while parasites from the other side do not catch up to the front. For simplicity we assumed that $\mathbf{P}(A=2)>0,\mathbf{P}(I=1) > 0$. In the following it will be clear how to adapt the proof to other cases. \\ Analogously to $(K_{n,j})_{j\ge1,n\in\Z},
 (\nu_{n,j})_{j\ge1,n\in\Z}$, one constructs $(\hat{K}_{n,j})_{j\ge1,n\in\Z}$, $(\hat{\nu}_{n,j})_{j\ge1,n\in\Z}$ on $\mathbf{\Omega}$ by reflecting $\mathbf{I},\mathbf{A},\mathbf{Y}$ around the origin, i.e.,
 \[
 \hat{K}_{n,j} = \inf\left\{k\ge j:\sum_{i=n+1}^{n+k} A_{i}\ge \beta_Ak,\sum_{i=n}^{n+k-1} I_i \le \beta_I k\right\},
 \] $\hat{\nu}_{n,j}$ is the first time more than $\beta_I \hat{K}_{n,j}$ parasites born between $n+1$ and $n+\hat{K}_{n,j}$ reached site $n$. Finally setting $\hat{\nu}_n(l) := \min\{\hat{\nu}_{l-n,1},\dots,\hat{\nu}_{l-n,n\wedge k_0}\}$ for $n\ge 1$ and $l\in\Z$, we observe by Lemma \ref{Uinf} that for $\gamma < \frac{1}{\mathbf{E}[\nu_{1,k_0}]}$
 \[
 \begin{split}
 \mathbf{P}&\left(\bigcap_{k=1}^\infty\{K_{k,k\wedge k_0}\le k\},\bigcap_{n=1}^\infty\left\{\sum_{k=1}^n\nu_k(0) \le \frac{n}{\gamma}\right\}\right) \\&= \mathbf{P}\left(\bigcap_{k=1}^\infty\{\hat{K}_{k,k\wedge k_0}\le k\},\bigcap_{n=1}^\infty\left\{\sum_{k=1}^n\hat{\nu}_k(0) \le \frac{n}{\gamma}\right\}\right)
 \end{split}
 \]
 is positive. We define the events
 \[
 \begin{split}
 B_\text{right}(\gamma) &:= \bigcap_{x>0}\bigcap_{i=1}^{A_x}\bigcap_{t\ge0}\left\{x+Y^{x,i}_t > -\lfloor \gamma t\rfloor\right\}, ~~
 B_\text{left}(\gamma) := \bigcap_{x<0}\bigcap_{i=1}^{A_x}\bigcap_{t\ge0}\left\{x+Y^{x,i}_t < \lfloor \gamma t\rfloor)\right\}.
 \end{split}
 \]
 Applying Lemma \ref{RWsublin} and due to the independence of $\mathbf{Y}$ and $\mathbf{A}$, we obtain that $B_\text{right}(\gamma)$ is independent of $B_\text{left}(\gamma)$ and
 \[
 \begin{split}
 \mathbf{P}(B_\text{left}(\gamma)) = \mathbf{P}(B_\text{right}(\gamma)) \ge \mathbf{E}\left[\left(\exp\left(-\frac{2}{\gamma}\right)\left(1-\e^{-\theta_\gamma}\right)\right)^{A_1}\right]\prod_{x=2}^\infty\mathbf{E}\left[\left(\left(1-\e^{(1-x)\theta_\gamma}\right)\right)^{A_x}\right].
 \end{split}
 \]
Taylor expanding
 \[
 \mathbf{E}[(1-z)^A] = 1 - \mathbf{E}[A]z +h(z)z
 \]
 for $z\in(0,1)$ and some $\vert h(z)\vert = o(1)$ as $z\to0$, and that for some $C > 1$ we have \[\log(1-z) \ge -Cz~\text{ for all }~0\le z\le\max_{x\ge 2}\e^{-(1-x)\theta_\gamma}\left(\mathbf{E}[A]-h\left(\e^{(1-x)\theta_\gamma}\right)\right)\] yields 
 \[
 \begin{split}
 \prod_{x=2}^\infty\mathbf{E}\left[\left(1-\e^{(1-x)\theta_\gamma}\right)^{A_x}\right] &= \exp\left(\sum_{x=2}^\infty\log\left(1-\mathbf{E}[A]\e^{(1-x)\theta_\gamma}+h\left(\e^{(1-x)\theta_\gamma}\right)\e^{(1-x)\theta_\gamma}\right)\right)\\
 &\ge \exp\left(-\sum_{x=2}^\infty C\e^{(1-x)\theta_\gamma}\left(\mathbf{E}[A]-h\left(\e^{(1-x)\theta_\gamma}\right)\right)\right) > 0.
 \end{split}
 \]
In conclusion, this yields $\mathbf{P}(B_\text{left}(\gamma)) = \mathbf{P}(B_\text{left}(\gamma)) > 0$.
Next we define
 \[
 C_\text{right}(\gamma) := \bigcap_{k=1}^\infty\{K_{1+k,k\wedge k_0}\le k\}\cap\bigcap_{n=1}^\infty\left\{\sum_{k=1}^n\nu_k(1) \le \frac{n}{\gamma}\right\}
 \]
 \[
 C_\text{left}(\gamma) := \bigcap_{k=1}^\infty\{\hat{K}_{-1-k,k\wedge k_0}\le k\}\cap\bigcap_{n=1}^\infty\left\{\sum_{k=1}^n\hat{\nu}_k(-1) \le \frac{n}{\gamma}\right\}.
 \]
To show that $\mathbf{P}((B_\text{right}(\gamma),C_\text{right}(\gamma)) > 0$ we will have to show an analog result as in Lemma \ref{Uinf} but under the measure $\mathbf{P}(\cdot\vert B_\text{right}(\gamma))$. \\
 First we note that $\{Y^{x,i}:x>0,i\in\N\}$ is still an independent collection, just not an identically distributed one, under the measure $\mathbf{P}(\cdot\vert B_\text{right}(\gamma))$. Because clearly the event $B_\text{right}(\gamma)$ gives the random walks a drift to the right, for any $x>0,i\in\N$ and $n\in\N,t\ge0$ we have
 \[
 \mathbf{P}\left(\left.\sup_{0\le s\le t}Y^{x,i}_s < n\right\vert B_\text{right}(\gamma)\right) \le \mathbf{P}\left(\sup_{0\le s\le t}Y^{x,i}_s < n\right).
 \]
 Reasoning just as in Lemma \ref{NuMoments}, we obtain that $\nu_1(1),\dots,\nu_{k_0-1}(1) <\infty$ a.s.~under $\mathbf{P}(\cdot\vert B_\text{right}(\gamma))$ and $\nu_{k_0}(1),\nu_{k_0+1}(1),\dots\in L^q(\Omega,\mathcal{F},\mathbf{P}(\cdot\vert B_\text{right}(\gamma)))$ for any $2 \le q < \frac{\alpha+1}{2}$.
Moreover, the arguments in Lemma \ref{conditinallyphimixing} only depend on $\mathbf{A}$ and $\mathbf{I}$, thus $\nu_1(1),\nu_2(1),\dots$ is also $\phi$-mixing under the measure $\mathbf{P}(\cdot\vert B_\text{right}(\gamma))$. Hence, one obtains analogously to the proof of Lemma \ref{Uinf} that
 \[
 \mathbf{P}(C_\text{right}(\gamma)\vert B_\text{right}(\gamma))>0. 
 \]
In particular, this means that for any $0<\gamma_1<\gamma_2<\frac{1}{\mathbf{E}[\nu_{1,k_0}]}$ and $\varepsilon>0$ the event $\mathcal{V}$, given by the intersection of
 \begin{center}
 $ \{A_0 = 2\}$ and $\{I_1 = I_{-1}=1\}$, \\
 Parasite $(0,1)$ jumps for the first time at time $\tau^{0,1}_1\in(1,1+\varepsilon)$ and jumps onto site $1$,\\
 Parasite $(0,2)$ jumps for the first time at time $\tau^{0,2}_{-2}\in(1,1+\varepsilon)$ and jumps onto site $-1$, \\
 $B_\text{right}(\gamma_1)\cap C_\text{right}(\gamma_2)\cap B_\text{left}(\gamma_1)\cap C_\text{left}(\gamma_2)$
 \end{center}
 has positive probability. Choosing now 
 \[
 \gamma_2\in\left(0,\frac{1}{\mathbf{E}[\nu_{1,k_0}]}\right),~~\gamma_1\in\left(0,\gamma_2\right),~~\text{ and }~~\varepsilon\in\left(0,\frac{1}{\gamma_1}-\frac{1}{\gamma_2}\right) 
 \]
 we have that on the event $\mathcal{V}$ the right and left host boundaries $(r_t)_{t\ge0},(l_t)_{t\ge0}$ satisfy
\begin{align*}
 r_t \ge \begin{cases}
 0, & t<\tau^{0,1}_1 \\
 1+\lfloor\gamma_2(t-\tau^{0,1}_1)\rfloor, & t\ge\tau^{0,1}_1
 \end{cases}~\text{ and }~l_t \le \begin{cases}
 0, & t<\tau^{0,2}_{-1} \\
 -1-\lfloor\gamma_2(t-\tau^{0,2}_{-1})\rfloor, & t\ge\tau^{0,2}_{-1}.
 \end{cases}
\end{align*}
Recalling the sets $\mathcal{L}_t$ containing the labels of living parasites at time $t$, $\mathcal{G}_t$ containing the labels of parasites that died until time $t$, and $F_t(x,i)\in\Z$ as the position of the parasite with label $(x,i)$ at time $t$, on the event $\mathcal{V}$ it also holds that
 \[
 \{F_t(x,i):(x,i)\in\mathcal{L}_t\cup\mathcal{G}_t, x > 0\} \begin{cases} = \emptyset, &t < \tau_1^{0,1}\\ \subset \Z_{>-\lfloor\gamma_1(t-\tau^{0,1}_1)\rfloor}, &t\ge\tau^{0,1}_1\end{cases}
 \]
 and
 \[
\{F_t(x,i):(x,i)\in\mathcal{L}_t\cup\mathcal{G}_t,x<0\} \begin{cases}=\emptyset,&t < \tau_{-1}^{0,2}\\ \subset\Z_{<\lfloor\gamma_1(t-\tau^{0,2}_{-1})\rfloor},&t\ge\tau^{0,2}_{-1}\end{cases}.
 \]
 The choice of parameters gives us now that
 \[
 \begin{split}
 \forall t\ge \tau^{0,1}_1: \min_{\substack{(x,i)\in\mathcal{L}_t\cup\mathcal{G}_t,\\x > 0}} F_t(x,i) & > l_t\\
 \forall t\ge \tau^{0,2}_{-1}: \max_{\substack{(x,i)\in\mathcal{L}_t\cup\mathcal{G}_t,\\x < 0}} F_t(x,i) & < r_t.\\
 \end{split}
 \]
 In words, any parasite born on some $x>0$ moves slower than with linear speed $\gamma_1$ to the left and thus never catches up to the left front of hosts, which moves at least at linear speed $\gamma_2$, and vice versa for parasites born on $x<0$. Hence, on the given event, the actual process with both directions filled with hosts moves just as two copies of the process with just one side filled, conditioned on the event $B_\text{right}(\gamma_1)\cap C_\text{right}(\gamma_2)$. As by definition of the event $C_\text{right}(\gamma_2)$ these processes survive, we obtain the survival of the process with initially all sites $x\neq 0$ inhabited by a host. 
 \end{proof}

\section{Construction of the spatial infection model with host immunity}\label{Construction}

In this section we formally construct the SIMI on $\Z$ with initial configurations as above as a strong Markov process on some probability space. We construct the process in two different ways. The first is to assign each parasite a label and sample its entire path after it enters the system by picking from a collection of i.i.d.~simple symmetric random walks $(Y^{x,i})_{x\in\Z,i\in\N}$. Of course the sampled path will become a virtual path once the parasite has died; nonetheless, it is still sampled entirely, as this is needed to construct the auxiliary jump times. We will call the so-constructed process the \textit{tagged} system. \\
The second construction is to assign each site $x$ in $\Z$ with a collection of Poisson point processes $\{P_{x,n}^\rightarrow, P_{x,n}^\leftarrow: n\in\N\}$, which will be used to determine the jump times and directions of parasites leaving the site $x$, where only the Poisson point processes with index $n\le \eta_{t-}(x)$ are allowed to be used. In this construction we won't be tracking each parasite individually, but only the amounts of parasites on each site. We call the so-constructed process the \textit{untagged} system. The natural coupling for different initial configurations in this system will allow us to show that the front is almost surely monotone in the initial configuration, whereas for the tagged system, as Example \ref{unmonontonecoupl} shows, this is not the case. Using this monotone coupling, we can then follow the lines of \cite{C2009} to show the Feller property of the constructed Markov process.
\subsection{Tagged system} 
We first construct the process in a coupled way for any finite starting configuration and then define the process for arbitrary starting configurations as the almost sure limit of the processes started in certain approximating, finite initial configurations. We will carry out this construction in detail. A similar construction is used in \cite{C2009, C2007}, but they rely on an almost sure monotonicity argument that, as we will see in Example \ref{unmonontonecoupl}, does not hold in our case. However, we follow the same approach of first defining the process for initial configurations with finitely many parasites and then extending it to infinitely many parasites by performing a limit in the state space. \\
In Definition \ref{Statespace} we give the state space in which the limit will be defined, and in Proposition \ref{polishspace} we show that this space is Polish. In Proposition \ref{unifconv} we then show the claimed convergence, and in Theorem \ref{StrongMarkov} we show the strong Markov property of the process.\\
We consider a probability space $\mathbf{\Omega}^\prime = (\Omega^\prime,\mathcal{F}^\prime,\mathbf{P}^\prime)$ on which the following independent random variables are defined.\\
There is a simple symmetric random walk starting in $0$ and jumping at rate $2$ called $Y = (Y_t)_{t\ge0}$ and an independent collection
\[
\mathbf{Y} := \{Y^{x,i}:x\in\Z,i\in\N\}
\]
of simple symmetric random walks starting in $0$ and jumping at rate $2$.\\
There is an $\N_0$-valued random variable $A\in L^1(\mathbf{\Omega}^\prime)$ and a collection
\[
\mathbf{A} := \{A_k:k\in\Z\}
\]
of i.i.d.~$\N_0$-valued variables, distributed as $A$.\\
There is an $\N$-valued random variable $I\in L^1(\mathbf{\Omega}^\prime)$ and a collection
\[
\mathbf{I} := \{I_k:k\in\Z\}
\]
of i.i.d.~$\N$-valued variables, distributed as $I$. \\
We give the construction of the process for a one-sided host population. The simpler case of a two-sided host population is treated in Remark \ref{rem:twoside}.
To begin the construction, we first specify the state space of the process. Because of the geometry of $\Z$ and since parasites jump only to nearest neighbors, the set of sites occupied by hosts and the set of sites free of hosts is separated by a single site $r\in\Z$, which we call the (right) front in the following. In the definition of the approximating jump times $\{\nu_n:n\in\Z\}$, we used the virtual paths of ghost parasites after their birth. For this reason, we include those in the state space and make the following definition.
\theoremstyle{definition}
\theoremstyle{definition}\begin{definition} \label{Statespace}
A parasite configuration is a $5$-tuple $w=(r,\mathcal{L},\mathcal{G},F,\iota)$ with 
\[r\in\Z,\mathcal{L},\mathcal{G}\subset\Z\times\N,F:\mathcal{L}\cup\mathcal{G}\to\Z,\iota\in\N\]
such that 
\begin{itemize}
 \item parasites are either dead or alive: $\mathcal{L}\cap\mathcal{G} = \emptyset$,
 \item hosts to the right of $r$ are not yet infected: \[\sup\left\{x\in\Z\vert\exists i\in\N:(x,i)\in\mathcal{L}\cup\mathcal{G}\right\} \le r,\]
 \item living parasites cannot sit on top of hosts: $\sup\{F(x,i):(x,i)\in\mathcal{L}\} \le r$,
 \item and the technical conditions that for some $\theta>0$
\begin{equation}\label{thetasum}
\sum_{(x,i)\in\mathcal{L}}\exp(\theta F(x,i))< \infty ~~~\text{ and }~~~\vert\mathcal{G}\vert<\infty.
\end{equation}
\end{itemize}
For $\theta>0$ we denote by $\mathbb{L}_\theta$ the space of all parasite configurations that satisfy \eqref{thetasum} and define
\[
\begin{split}
\xi(w,w^\prime,x,i) &:= \left\vert\mathds{1}_{(x,i)\in\mathcal{L}}\e^{\theta (F(x,i)-r)}-\mathds{1}_{(x,i)\in\mathcal{L}^\prime}\e^{\theta (F^\prime(x,i)-r^\prime)}\right\vert ,\\
\xi^\prime(w,w^\prime,x,i) &:= \vert\mathds{1}_{(x,i)\in\mathcal{G}}F(x,i)-\mathds{1}_{(x,i)\in\mathcal{G}^\prime}F^\prime(x,i)\vert
\end{split}
\] 
for $w=(r,\mathcal{L},\mathcal{G},F,\iota),w^\prime\in\mathbb{L}_\theta,(x,i)\in\Z\times\N$, in order to define the metric $d_\theta$ on $\mathbb{L}_\theta$ by
\begin{align*}
d_\theta(w,w^\prime) := \vert r-r^\prime\vert +\vert \iota-\iota^\prime\vert&+ \sum_{(x,i)\in\Z\times\N}\xi(w,w^\prime,x,i) +\xi^\prime(w,w^\prime,x,i) .
\end{align*}
\end{definition}
\begin{remark}
For a configuration $w=(r,\mathcal{L},\mathcal{G},F,\iota)$, the site $r$ will be the rightmost site without host cells. The set $\mathcal{L}$, called living parasites, will be the birth labels of parasites that are still alive, and $\mathcal{G}$, called ghost parasites, will be the birth labels of parasites that were alive at some time but already died. $F$ assigns each parasite its current position, and $\iota$ is the immunity of site $r+1$, or equivalently the number of attempts needed before the next successful infection.
\end{remark}
\theoremstyle{plain}
As we will see later, the metric $d_\theta$ is sufficient to define the process for initial configurations with infinitely many parasites, by approximation with finitely many parasites. But first we need to make sure that the state space above has at least some good properties.
\begin{proposition}\label{polishspace}
Let $\theta >0$, then $(\mathbb{L}_\theta,d_\theta)$ defined above is a complete and separable metric space. Also, $(\mathbb{L}_\theta,d_\theta)$ is not locally compact.
\end{proposition}
\begin{proof}
It is clear that $d_\theta$ defines a metric. Moreover, because $\mathbb{L}_\theta$ contains the function space
\[
\left\{\eta:\N_0\to\N_0: \sum_{n=0}^\infty\eta(n)\e^{-\theta n}<\infty\right\}
\]
 it cannot be locally compact, simply by considering the sequence of configurations placing $\lfloor\varepsilon \e^{\theta n}\rfloor$ living parasites on site $r-n$ and leaving all other sites empty, which is a sequence with no convergent subsequence in a $\varepsilon$-Ball.\\
Now we show that $(\mathbb{L}_\theta,d_\theta)$ is separable. The set 
\[
\mathbb{A} := \left\{(r,\mathcal{L},\mathcal{G},F,\iota)\in\mathbb{L}_\theta:\left\vert\mathcal{L}\vert < \infty\right.\right\}
\]
is countable, which can be seen by observing that as sets
\[
\begin{split}
\mathbb{A} &= \bigcup_{k=0}^\infty\bigcup_{m=0}^\infty \left\{(r,\mathcal{L},\mathcal{G},F,\iota)\in\mathbb{L}_\theta:\left\vert\mathcal{L}\vert = k \right. ,\left\vert\mathcal{G}\vert = m \right. \right\} \\ &\overset{\sim}{\subset}\bigcup_{k=0}^\infty\bigcup_{m=0}^\infty \Z\times\binom{\Z\times\N}{k}\times\binom{\Z\times\N}{m}\times\Z^{m+k}\times\N.
\end{split}
\]
For $l\ge 1$ we define a map $w^l:\mathbb{L}_\theta \to \mathbb{A}$ as follows. For $w=(r,\mathcal{L},\mathcal{G},F,\iota)$ set
\[
\mathcal{L}^l(w):= \{(x,i)\in\mathcal{L}: F(x,i) > r-l\}~\text{ and }~w^l(w):= \left(r,\mathcal{L}^l(w),\mathcal{G},F_{\vert_{\mathcal{L}^l(w)\cup\mathcal{G}}},\iota\right),
\]
that is, we consider only living parasites that are placed at a distance less than $l$ to the front.
Using \eqref{thetasum}, and that by definition $F(x,i)-r+l > 0$ for all $(x,i)\in\mathcal{L}^l(w)$, we obtain
\[
\vert\mathcal{L}^l(w)\vert < \e^{\theta l}\sum_{(x,i)\in\mathcal{L}^l(w)}\exp(\theta(F(x,i)-r))< \infty.
\]
Thus $w^l(w)\in\mathbb{A}$ for all $l\in\N,w\in\mathbb{L}_\theta$.
Since 
\[
\mathcal{L}^1(w)\subset\mathcal{L}^2(w)\subset\dots ~~\text{ and }~~\mathcal{L} = \bigcup_{l=1}^\infty \mathcal{L}^l(w),
\]
Condition \eqref{thetasum} yields
\[
d_\theta(w,w^l(w)) = \sum_{(x,i)\in\mathcal{L}\setminus\mathcal{L}^l(w)}\exp(\theta(F(x,i)-r)) \to 0~~~(l\to\infty).
\]
This shows that $\mathbb{A}$ is dense, and thus $\mathbb{L}_\theta$ is separable. \\
Next, we show that $\mathbb{L}_\theta$ is also complete. Suppose that \[(w_n)_{n\in\N} =((r_n, \mathcal{L}_n, \mathcal{G}_n, F_n, \iota_n))_{n\in \N}\] is a Cauchy sequence in $\mathbb{L}_\theta$. Choosing $\varepsilon = 1$ we find an $N_0\in\N$ such that $d_\theta(w_n,w_m)<1$ for all $n,m\ge N_0$. Since $d_\theta(w_n,w_m)$ is at least
\[
 \vert r_n-r_m\vert +\vert \iota_n-\iota_m\vert+ \sum_{(x,i)\in\mathcal{G}_n\cup\mathcal{G}_m} \vert \mathds{1}_{(x,i)\in\mathcal{G}_n}F_n(x,i) - \mathds{1}_{(x,i)\in\mathcal{G}_m}F_m(x,i)\vert,
\]
which is integer valued, this implies that there is an $r\in\Z,\mathcal{G}\subset\Z\times\N,\iota\in\N$ and $\widetilde{F}:\mathcal{G}\to \Z$ such that 
\[
r_n = r,~~~\mathcal{G}_n=\mathcal{G},~~~(F_n)_{\vert_{\mathcal{G}}} = \widetilde{F},~~~\iota_n=\iota
\]for all $n\ge N_0$. This yields that for $m,n\ge N_0$ we have
\[
d(w_n,w_m) = \sum_{(x,i)\in\Z\times\N}\e^{-\theta r}\left\vert\mathds{1}_{(x,i)\in\mathcal{L}_n}\e^{\,\theta F_n(x,i)}-\mathds{1}_{(x,i)\in\mathcal{L}_m}\e^{\,\theta F_m(x,i)}\right\vert.
\]
For $l\in \N$, setting
\[\varepsilon(l) := \e^{-\theta(l-1)}-\e^{-\theta l},\] we define 
\[
N_l := \inf\{N \ge N_{l-1}\vert \forall m,n \ge N: d(w_n,w_m) < \varepsilon(l)\}.
\]
By definition of $\varepsilon(l)$, this means that $w^l(w_n) = w^l(w_m)$ for all $l\in\N$ and $m,n\ge N_l$.
We define
\[
\mathcal{L} : = \bigcup_{l=1}^\infty \mathcal{L}^l(w_{N_l}),
\] 
as well as
$F:\mathcal{L}\cup\mathcal{G}\to\Z$ by
\[
F(x,i) : = \mathds{1}_{(x,i)\in\mathcal{L}}F_{N_l}(x,i) + \mathds{1}_{(x,i)\in\mathcal{G}} \widetilde{F}(x,i),
\]
where $l\in\N$ is such that $(x,i)\in\mathcal{L}^l(w_{N_l})$ and set $w :=(r,\mathcal{L},\mathcal{G},F,\iota)$, noting that for $l\in\N$ we have $w^l(w) = w^l(w_{N_l})$.
Now let $\varepsilon > 0$ and $L\in\N$ such that
\[
\varepsilon(L) \le \varepsilon.
\]
For fixed $l> L$ and all $n,m\ge N_L$ we have $w^L(w_n) = w^L(w_m)=w^L(w)$ and obtain
\[
\begin{split}
&\sum_{(x,i)\in\mathcal{L}^l(w_n)\cup\mathcal{L}^l(w_{m})}\e^{\,-\theta r}\left\vert\mathds{1}_{(x,i)\in\mathcal{L}_n}\e^{\,\theta F_n(x,i)}-\mathds{1}_{(x,i)\in\mathcal{L}_{m}}\e^{\,\theta F_{m}(x,i)}\right\vert \\
=&\sum_{(x,i)\in(\mathcal{L}^l(w_n)\cup\mathcal{L}^l(w_{m}))\setminus\mathcal{L}^L(w)}\e^{\,-\theta r}\left\vert\mathds{1}_{(x,i)\in\mathcal{L}_n}\e^{\,\theta F_n(x,i)}-\mathds{1}_{(x,i)\in\mathcal{L}_{m}}\e^{\,\theta F_{m}(x,i)}\right\vert \\
\le &\sum_{(x,i)\in(\mathcal{L}_n\cup\mathcal{L}_{m})\setminus\mathcal{L}^L(w)}\frac{\e^{\,-\theta r}}{1-\e^{-\theta}}\left\vert\mathds{1}_{(x,i)\in\mathcal{L}_n}\e^{\,\theta F_n(x,i)}-\mathds{1}_{(x,i)\in\mathcal{L}_{m}}\e^{\,\theta F_{m}(x,i)}\right\vert\\
= &\frac{d_\theta(w_n,w_m)}{1-\e^{-\theta}} < \frac{\varepsilon}{1-\e^{-\theta}},
\end{split}
\]
where the inequality in the third line can be shown as follows. For each $(x,i)\in(\mathcal{L}_n\cup\mathcal{L}_m)\setminus\mathcal{L}^L(w)$ we distinguish between four cases 
\[
\begin{split}
\text{I: }&(x,i) \notin\mathcal{L}^l(w_n)\text{ and }(x,i)\notin\mathcal{L}^l(w_m)\\
\text{II: }&(x,i) \in\mathcal{L}^l(w_n) \text{ and }(x,i)\notin\mathcal{L}^l(w_m) \\
\text{III: }&(x,i) \notin\mathcal{L}^l(w_n)\text{ and }(x,i)\in\mathcal{L}^l(w_m)\\
\text{IV: }&(x,i) \in \mathcal{L}^l(w_n)\text{ and }(x,i)\in\mathcal{L}^l(w_m).
\end{split}
\]
In the first case we just added a positive term so the inequality holds, and in the last case we just divided the already existing term by $1-\e^{-\theta}<1$, which only makes it larger. In the second case we have to show that
\[
\e^{\theta(F_n(x,i)-r)} \le \frac{\exp(\theta(F_n(x,i)-r))-\exp(\theta(F_m(x,i)-r))}{1-\e^{-\theta}}.
\]
Because $(x,i)\notin\mathcal{L}^l(w_m)$, it holds that $F_m(x,i)-r \le -l$ and $(x,i)\in\mathcal{L}^l(w_n)\setminus\mathcal{L}^L(w_n)$ implies $-l < F_n(x,i)-r\le -L$. Hence, it remains to show that
\[
\e^{-\theta y}\le \frac{\e^{-\theta y}-\e^{-\theta z}}{1-\e^{-\theta}}
\]
for all $y\in\{L,\dots,l-1\},z\ge l$. For fixed $y$, the right-hand side is clearly increasing in $z$, and thus it suffices to show the inequality for $z=l$. But for $z=l$ we have that
\[
\frac{\e^{-\theta y}}{\e^{-\theta y}-\e^{-\theta l}} = \frac{1}{1-\e^{-\theta(l-y)}},
\]
which
attains its maximum at $y=l-1$ and thus yields the claimed inequality. The third case is analogous to the second one.\\ 
Hence, letting $m\to\infty$, we obtain
\[
\sum_{(x,i)\in\mathcal{L}^l(w_n)\cup\mathcal{L}^l(w)}\e^{\,-\theta r}\left\vert\mathds{1}_{(x,i)\in\mathcal{L}_n}\e^{\,\theta F_n(x,i)}-\mathds{1}_{(x,i)\in\mathcal{L}}\e^{\,\theta F(x,i)}\right\vert\le \frac{\varepsilon}{1-\e^{-\theta}}
\]
for any $n\ge N_L$. Thus, we can bound $\sum_{(x,i)\in\mathcal{L}^l(w)}\exp(\theta(F(x,i)-r))$ by 
\begin{align*}
 &\sum_{(x,i)\in\mathcal{L}^l(w_n)\cup\mathcal{L}^l(w)}\e^{\,-\theta r}\left\vert\mathds{1}_{(x,i)\in\mathcal{L}_n}\e^{\,\theta F_n(x,i)}-\mathds{1}_{(x,i)\in\mathcal{L}}\e^{\,\theta F(x,i)}\right\vert\ \\&\quad+ \sum_{(x,i)\in\mathcal{L}^l(w_n)}\exp(\theta(F_n(x,i)-r)) \\&\le \frac{\varepsilon}{1-\e^{-\theta}} + \sum_{(x,i)\in\mathcal{L}_n}\exp(\theta(F_n(x,i)-r)) < \infty.
\end{align*}

Letting $l\to \infty$ yields $w\in\mathbb{L}_\theta$ and 
\[ d(w_n, w)=
\sum_{(x,i)\in\mathcal{L}_n\cup\mathcal{L}}\e^{\,-\theta r}\left\vert\mathds{1}_{(x,i)\in\mathcal{L}_n}\e^{\,\theta F_n(x,i)}-\mathds{1}_{(x,i)\in\mathcal{L}}\e^{\,\theta F(x,i)}\right\vert\le \frac{\varepsilon}{1-\e^{-\theta}}
\]
for any $n\ge N_L$, thus
\[
\lim_{n\to\infty} w_n = w
\]
in $\mathbb{L}_{\theta}$. This shows that $(\mathbb{L}_\theta,d_\theta)$ is complete and thus finishes the proof.
\end{proof}
We begin by constructing the process for a finite starting configuration. That is, for any $w\in\mathbb{A}$ we now use the collections $\mathbf{Y},\mathbf{A},\mathbf{I}$ to define a strong Markov process \[(X_t(w))_{t\ge0} = ((r_t,\mathcal{L}_t,\mathcal{G}_t,F_t,\iota_t)(w))_{t\ge 0}\] with càdlàg sample paths in $\mathbb{A}$ on the probability space $\mathbf{\Omega}^\prime$.\\
Let $w\in\mathbb{A}$ and set $\sigma_0 := 0,\rho_0 = 0,r_0:=r,\iota_0:=\iota,\mathcal{L}_0:=\mathcal{L},\mathcal{G}_0 = \mathcal{G}$ and \[F_t(x,i) := F(x,i)+Y^{x,i}_{t} \text{ for all }(x,i)\in\mathcal{L}_0\cup\mathcal{G}_0,t\ge \sigma_0.\] Assume that for $n \ge 0$ we have already defined
\[\sigma_0,\dots,\sigma_{n},~~(r_t,\mathcal{L}_t,\mathcal{G}_t,F_t,\iota_t)\in\mathbb{A}\text{ for all }0\le t\le\sigma_{n}\text{ and }\rho_0,\dots,\rho_{r_{\sigma_n}-r}\le \sigma_n.\]
If $\mathcal{L}_{\sigma_n} = \emptyset$, we set $r_t = r_{\sigma_n},\mathcal{L}_t = \emptyset,\mathcal{G}_t = \mathcal{G}_{\sigma_n},\iota_t = \iota_{\sigma_n}$ for all $t>\sigma_n$ and $\sigma_m = \infty$ for all $m>n$.\\
Otherwise, if $\mathcal{L}_{\sigma_n} \neq\emptyset$, let 
\[
\sigma_{n+1} :=\inf\{t>\sigma_{n}\,\vert\,\exists (x,i)\in\mathcal{L}_{\sigma_{n}}:F_t(x,i) = r_{\sigma_{n}}+1\}.
\]
By construction we have $0< \vert\mathcal{L}_{\sigma_{n}}\vert\le\vert \mathcal{L}\vert +\sum_{i=1}^n A_{r+i}<\infty$, thus $\sigma_n<\sigma_{n+1}<\infty$ almost surely, and there is almost surely a unique $(x_n,i_n)\in\mathcal{L}_{\sigma_{n}}$ such that $F_{\sigma_{n+1}}(x_n,i_n) = r_{\sigma_{n}}+1$. Moreover, $\sigma_{n+1}-\sigma_n$ is stochastically dominated by an exponentially distributed random variable with random parameter $\vert\mathcal{L}\vert+\sum_{i=1}^n A_{r+i}$, corresponding to the case that any previous infection was successful, all living parasites are located on $r_{\sigma_n}$, and the first jump goes to the right. In particular, because $\mathbf{E}[A]<\infty$ and thus
\[
\sum_{n=0}^\infty \frac{1}{\vert\mathcal{L}\vert+\sum_{i=1}^n A_{r+i}} = \infty
\]
almost surely, this implies that $\lim_{n\to\infty}\sigma_n = \infty$ almost surely.\\
For $t\in(\sigma_{n},\sigma_{n+1})$ we set $r_t = r_{\sigma_{n}}, \iota_t = \iota_{\sigma_{n}}, \mathcal{L}_t = \mathcal{L}_{\sigma_{n}},\mathcal{G}_t = \mathcal{G}_{\sigma_{n}}$ and then finally set $\mathcal{G}_{\sigma_{n+1}} = \mathcal{G}_{\sigma_{n}}\cup\{(x_n,i_n)\}$. 
\begin{itemize}
 \item If $\iota_{\sigma_{n}} = 1$: Set $r_{\sigma_{n+1}}=r_{\sigma_n}+1,\iota_{\sigma_{n+1}} = I_{r_{\sigma_{n}}+1},\rho_{r_{\sigma_{n+1}}-r} = \sigma_{n+1}$. Also, if $A_{r_{\sigma_{n+1}}} > 0$ set
 \[\mathcal{L}_{\sigma_{n+1}}= \mathcal{L}_{\sigma_n}\setminus\{(x_n,i_n)\}\cup\{(r_{\sigma_{n+1}},1),\dots,(r_{\sigma_{n+1}},A_{r_{\sigma_{n+1}}})\}\]
 with $F_t(r_{\sigma_{n+1}},i) = r_{\sigma_{n+1}} + Y^{r_{\sigma_{n+1}},i}_{t-\sigma_{n+1}}$ for $1\le i\le A_{r_{\sigma_{n+1}}},t\ge\sigma_{n+1}$, and if $A_{r_{\sigma_{n+1}}}= 0$ set
 \[
 \mathcal{L}_{\sigma_{n+1}}= \mathcal{L}_{\sigma_n}\setminus\{(x_n,i_n)\}.
 \]
 \item If $\iota_{\sigma_{n}} > 1$: Set $r_{\sigma_{n+1}}=r_{\sigma_n},\iota_{\sigma_{n+1}} = \iota_{\sigma_n}-1$ and
 \[\mathcal{L}_{\sigma_{n+1}}= \mathcal{L}_{\sigma_n}\setminus\{(x_n,i_n)\}.\]
\end{itemize}
\begin{remark}\label{rem:twoside}
 If the host population is two-sided, then any initial configuration considered contains only finitely many parasites. Hence the above construction can easily be adapted to construct the process in this case. 
\end{remark}

Before we continue and define $((X_t(w))_{t\ge 0}$ for any $w\in\mathbb{L}_\theta$ by using the approximation $\lim_{l\to\infty}w^l(w) = w$ in $\mathbb{L}_\theta$, we will show some useful inequalities. For $w=(r,\mathcal{L},\mathcal{G},F,\iota)\in\mathbb{L}_\theta$ and $(x,i)\in\mathcal{L}\cup\mathcal{G},m\in\N_0$ we set $[m] = \{1,\dots,m\}$ and define the maps
\[
\begin{split}
 F^{(x,i),\pm}(w):\mathcal{L}\cup\mathcal{G}&\to\Z\\
 (y,j)&\mapsto\begin{cases}F(x,i)\pm 1,&(y,j)=(x,i)\\F(y,j),&\text{else}\end{cases}\\
 F^{(x,i),m}(w):\mathcal{L}\cup\{r+1\}\times[m]\cup\mathcal{G}&\to\Z\\
 (y,j)&\mapsto\begin{cases}r+1,&y=r+1\text{ or }(y,j) = (x,i)\\F(y,j),&\text{else}\end{cases}
\end{split}
\]
and note that for $m=0$ we have $F^{(x,i),0}= F^{(x,i),+}$. With these maps, for $m\in\N_0,k\in\N$ and $(x,i)\in\mathcal{L}\cup\mathcal{G}$ we define the configurations
\[
\begin{split}
w^{(x,i),\pm} &:= (r,\mathcal{L},\mathcal{G},F^{(x,i),\pm}(w),\iota),\\
w^{(x,i),m,k} &:= (r+1,\mathcal{L}\setminus\{(x,i)\}\cup\{r+1\}\times[m],\mathcal{G}\cup\{(x,i)\},F^{(x,i),m}(w),k)\\
w^{(x,i),\text{ f}} &:= (r,\mathcal{L}\setminus\{(x,i)\},\mathcal{G}\cup\{(x,i)\},F^{(x,i),+}(w),\iota-1).
\end{split}
\]
For $\vartheta \ge\theta$ we define the functions
\[
\begin{split}
 f_\vartheta:\mathbb{L}_\theta&\to \R\\
 w&\mapsto \sum_{(x,i)\in\mathcal{L}\cup\mathcal{G}}\exp(\vartheta F(x,i)).
\end{split}
\] 
and note that
\[
\begin{split}
\vert f_\vartheta(w^{(x,i),\pm})-f_\vartheta(w)\vert &= \e^{\vartheta F(x,i)}\left\vert\e^{\pm\vartheta}-1\right\vert \\
\vert f_\vartheta(w^{(x,i),m,k})-f_\vartheta(w)\vert &= \e^{\vartheta r}\left(\e^{\vartheta}(m+1)-1\right) \\
\vert f_\vartheta(w^{(x,i),\text{f}})-f_\vartheta(w)\vert &= \e^{\vartheta r}\left( \e^\vartheta-1\right).
\end{split}
\]
A simple calculation shows that for any $w\in\mathbb{A}$ we have
\begin{align*}
Lf_\vartheta(w)&:= \lim_{t\to 0+}\frac{\mathbf{E}[f_\vartheta(X_t(w))-f_\vartheta(w)]}{t} \\
&= \sum_{(x,i)\in\mathcal{L}\setminus\mathcal{L}^1(w)\cup\mathcal{G}} f_\vartheta(w^{(x,i),+})+f_\vartheta(w^{(x,i),-})-2f_\vartheta(w) \\
&\quad+ \sum_{(x,i)\in\mathcal{L}^1(w)}f_\vartheta(w^{(x,i),-})+\mathds{1}_{\iota=1}\sum_{m=0}^\infty\sum_{k=1}^\infty b_me_kf_\vartheta(w^{(x,i),m,k}) \\
&\qquad\qquad\qquad\quad+\mathds{1}_{\iota>1}f_\vartheta(w^{(x,i),\text{ f}})-2f_\vartheta(w),
\end{align*}
where $e_k = \mathbf{P}(I=k)$ and $b_m =\mathbf{P}(A=m)$.
This immediately follows from the construction of $X_t(w)$ using a finite subset of $\mathbf{Y},\mathbf{I},\mathbf{A}$. To shorten notation we set
\[
\mathds{1}_{\iota=1}\sum_{m=0}^\infty\sum_{k=1}^\infty b_me_k =: \sum_{m,k}\delta_m^k,
\]
and compute
\[
\begin{split}
Lf_\vartheta(w) &= \sum_{(x,i)\in\mathcal{L}\setminus\mathcal{L}^1(w)\cup\mathcal{G}}\sum_{(y,j)\in\mathcal{L}\cup\mathcal{G}}\e^{\vartheta F^{(x,i),+}(y,j)}+\e^{\vartheta F^{(x,i),-}(y,j)}-2\e^{\vartheta F(y,j)} \\
&\quad+ \sum_{(x,i)\in\mathcal{L}^1(w)} \left(\sum_{(y,j)\in\mathcal{L}\cup\mathcal{G}}\e^{\vartheta F^{(x,i),-}(y,j)}-2\e^{\vartheta F(y,j)}\right.\\
&\qquad\qquad\qquad\qquad+ \sum_{m,k}\delta_m^k\sum_{(y,j)\in\mathcal{L}\cup\{r+1\}\times[m]\cup\mathcal{G}}\e^{\vartheta F^{(x,i), m, k}(y,j)}\\
&\qquad\qquad\qquad\qquad+\left.\mathds{1}_{\iota>1}\sum_{(y,j)\in\mathcal{L}\cup\mathcal{G}}\e^{\vartheta F^{(x,i),+}(y,j)} \right) 
\end{split}
\]
\[
\begin{split}
 &= \sum_{(x,i)\in\mathcal{L}\setminus\mathcal{L}^1(w)\cup\mathcal{G}} \e^{\vartheta(F(x,i)+1)}+\e^{\vartheta(F(x,i)-1)}-2\e^{\vartheta F(x,i)} \\
&\quad+ \sum_{(x,i)\in\mathcal{L}^1(w)}\e^{\vartheta(r-1)}-2\e^{\vartheta r}\\
&\quad+\sum_{m,k}\delta_m^k\left(\e^{\vartheta(r+1)}+\sum_{(y,j)\in\{r+1\}\times[m]}\e^{\vartheta(r+1)}\right)\\
&\quad+\mathds{1}_{\iota>1}\e^{\vartheta(r+1)}\\
&= \sum_{(x,i)\in\mathcal{L}\setminus\mathcal{L}^1(w)} \e^{\vartheta F(x,i)}(\e^\vartheta+\e^{-\vartheta}-2) \\
&\quad+ \sum_{(x,i)\in\mathcal{L}^1(w)}\e^{\vartheta r}\left(\e^{-\vartheta}-2+\left(1+\mathds{1}_{\iota=1}\sum_{m=0}^\infty mb_m\right)\e^{\vartheta}\right)
\end{split}
\]
This shows that $Lf_\vartheta(w) \in[\lambda_{1,\vartheta}f_\vartheta(w),\lambda_{2,\vartheta}f_\vartheta(w)]$ for \[0< \e^\vartheta+\e^{-\vartheta}-2 =:\lambda_{1,\vartheta} < \e^{-\vartheta}-2+(1+\mathbf{E}[A])\e^{\vartheta} =:\lambda_{2,\vartheta}.\]
 Thus
\begin{equation}\label{normexpect}
\mathbf{E}[f_\vartheta(X_t(w))]\le \e^{\lambda_{2,\vartheta}t}f_\vartheta(w)
\end{equation}
and $(f_\vartheta(X_t(w))_{t\ge0}$ is a nonnegative submartingale. Doob's inequality yields
\[
\mathbf{P}\left(\sup_{0\le s\le t}f_\vartheta(X_s(w))\ge M\right)\le \frac{\mathbf{E}[f_\vartheta(X_t(w))]}{M}.
\]
Since by definition 
\[r_t(w) = \begin{cases}\sup_{0\le s\le t}\sup_{(x,i)\in\mathcal{L}_s(w)} F_s(w)(x,i),&\text{if }\sup_{0\le s\le t}\max_{(x,i)\in\mathcal{L}}F_s(x,i) \ge r\\r,&\text{ else}\end{cases}\] 
we obtain for $t\ge0$ that
\begin{equation}\label{frontupbound}
\begin{split}
\mathbf{P}&(r_t(w)-r> \gamma t) = \mathbf{P}\left(r_t(w)-r > \gamma t, \sup_{0\le s\le t}\max_{(x,i)\in\mathcal{L}}F_s(x,i) \ge r\right) \\
&= \mathbf{P}\left(\sup_{0\le s\le t}\sup_{(x,i)\in\mathcal{L}_s(w)} F_s(w)(x,i)-r > \gamma t, \sup_{0\le s\le t}\max_{(x,i)\in\mathcal{L}}F_s(x,i)\ge r\right)\\
&\le \mathbf{P}\left(\sup_{0\le s\le t}f_\vartheta(X_s(w))> \e^{\vartheta(r+\gamma t)}\right) \le \e^{-c_{\gamma,\vartheta}t}f_\vartheta(w)\e^{-\vartheta r},
\end{split}
\end{equation} with $c_{\gamma,\vartheta} := \gamma\vartheta-\lambda_{2,\vartheta} > 0$ for large $\gamma$. 

Similarly as in \cite{C2007}, we constructed the model in a way that allows us to couple arbitrary initial configurations $w\in \mathbb{A}$. However, in our coupling the front $(r_t(w))_{t\ge0}$ is not monotone in the initial condition. To see this, we give the following concrete example.

\begin{example}\label{unmonontonecoupl}
Let 
 \[
 \mathcal{L}_1 = \{(1,1),(1,2)\},\mathcal{L}_2 = \mathcal{L}_1\cup\{(0,1),(0,2)\}
 \]
 and \[w_i=(1,\mathcal{L}_i,\emptyset,(x,i)\mapsto x,1)~~~(i\in\{1,2\}).\]
We will show that for all $t_0>0$ 
\[
\mathbf{P}(r_{t_0}(w_1) \ge 4,r_{t_0}(w_2) \le 3) > 0.
\]
 For a label $(x,i)\in\Z\times\{1,2\}$ let $(t_{x,i}^n)_{n\in\N}$ be the i.i.d.~Exp$(2)$-distributed jump times of $Y^{x,i}$.
 We assume that all jumps occurring are jumps to the right, $I_2=I_3=I_4=1$, i.e., every infection attempt is successful, and $A_2=A_3=A_4=2$, which clearly has positive probability.
We begin by investigating the system started from configuration $2$, where initially $(0,1),(0,2),(1,1),(1,2)$ are active. Assume 
 \begin{equation}\label{02slow}t_{0,2}^1>t_0,\end{equation}
 and hence particle $(0,2)$ does not move at all before time $t_0$, which happens with positive probability. Next, assume that
 \begin{equation}\label{01fast}t_{0,1}^1+t_{0,1}^2<\min\{t_{1,1}^1,t_{1,2}^1\}\end{equation} hence particle $(0,1)$ is the first which reaches site $2$ and dies, after waking up particles on site $2$ at time $t_{0,1}^1+t_{0,1}^2$. Next, we assume \begin{equation}\label{21fast}t_{0,1}^1+t_{0,1}^2+t_{2,1}^1<\min\{t_{1,1}^1+t_{1,1}^2,t_{1,2}^1+t_{1,2}^2, t_{0,1}^1+t_{0,1}^2+t_{2,2}^1\},\end{equation} so that particle $(2,1)$ wakes up the particles on site $3$ and dies at time $t_{0,1}^1+t_{0,1}^2+t_{2,1}^1$. At last, we assume
 \begin{equation}\label{11fast}t_{1,1}^1+t_{1,1}^2+t_{1,1}^3 < \min\left\{\begin{matrix}t_{1,2}^1+t_{1,2}^2+t_{1,2}^3 ,t_{0,1}^1+t_{0,1}^2+t_{2,2}^1,\\t_{0,1}^1+t_{0,1}^2+t_{2,1}^1+t_{3,1}^1,t_{0,1}^1+t_{0,1}^2+t_{2,1}^1+t_{3,2}^1\end{matrix}\right\} \end{equation}
 and hence particle $(1,1)$ wakes up the particles on site $4$ at time $T_2=t_{1,1}^1+t_{1,1}^2+t_{1,1}^3$.\\
Next, we study the system started in configuration $1$, where initially only $\{(1,1),(1,2)\}$ are active. We assume \begin{equation}\label{12slow}t_{1,1}^1<t_{1,2}^1\end{equation} and hence particle $(1,1)$ dies when waking up the particles on site $2$. Next we assume 
 \begin{equation}\label{12fast}
t_{1,2}^1+t_{1,2}^2 < \min\{t_{1,1}^1+t_{2,1}^1,t_{1,1}^1+t_{2,2}^1\}
 \end{equation}
which means particle $(1,2)$ dies when waking up the particles on site $3$. Next we assume that
 \begin{equation}\label{21faster}
t_{1,1}^1+t_{2,1}^1+t_{2,1}^2 < \min\{t_{1,1}^1+t_{2,2}^1, t_{1,2}^1+t_{1,2}^2 +t_{3,1}^1,t_{1,2}^1+t_{1,2}^2 +t_{3,2}^1\}
 \end{equation}
 and hence particle $(2,1)$ wakes up the particles on site $4$ at time $T_1 = t_{1,1}^1+t_{2,1}^1+t_{2,1}^2$.
 Now if \begin{equation}\label{w1fast}t_{1,1}^1 +t_{2,1}^1+t_{2,1}^2 < t_0 < t_{1,1}^1+t_{1,1}^2+t_{1,1}^3\end{equation} we have $T_1<t_0 < T_2$ or, in other words,
 \begin{equation}\label{41slow}
r_{t_0}(w_2)\le 3 < 4 \le r_{t_0}(w_1).
 \end{equation}
Clearly, the event that \eqref{02slow}, \eqref{01fast}, \eqref{21fast}, \eqref{11fast}, \eqref{12slow}, \eqref{12fast}, \eqref{21faster}, and \eqref{w1fast} has positive probability, which concludes the proof. \\
The following table depicts an example configuration for times $(t_{x,i}^n)_{1\le n\le 3}$ that satisfies the assumptions above for $t_0=9$ (times that don't appear in any condition are left blank and can attain any value, as they do not matter for the outcome).\\
\begin{tabular}[ht]{l|c|c|c|c|c|c|c|c}
 & $(0,1)$&$(0,2)$&$(1,1)$&$(1,2)$&$(2,1)$&$(2,2)$&$(3,1)$&$(3,2)$ \\
\hline
$t_{x,i}^1$&$1$&$11$&$4$&$5$&$3$&$9$&$6$&$6$ \\
$t_{x,i}^2$&$1$&&$1.5$&$1$&$1$&&\\
$t_{x,i}^3$&&&$4$&$4$&&
\end{tabular}\\
In this example we have $T_1 = 8$, but $T_2=9.5$.
\end{example}
This example also shows that we cannot simply use that for $t\ge0$ the map $l\mapsto r_t(w^l(w))$ is monotone to define $r_t(w)$ as its limit like it was done in \cite{C2007}.
In contrast, we use that almost surely $l\mapsto r_t(w^l(w))$ is constant for large enough $l$, because the additional random walks have not yet reached site $r$.
\begin{proposition}\label{unifconv}
 Let $r\in\Z$ be fixed and $w = (r,\mathcal{L},\mathcal{G},F,I_{r+1})$ be a random initial configuration taking values in $\mathbb{L}_\theta$, defined on $\mathbf{\Omega}$, that is independent of \[
 \{A_x,I_{x+1}: x >r,i\in\N\}\cup\mathbf{Y}
 \] and such that
 \[
\mathbf{E}\left[\sum_{(x,i)\in\mathcal{L}}\exp(\theta(F(x,i)-r))\right] < \infty.
 \] Then there is a set $B^w\in\mathcal{F}$, depending only on $w$, such that $\mathbf{P}(B^w)=1$ and for all $\omega\in B^w$ and all $T\ge0$
 \[
 \{(X_t(w^l(w))(\omega))_{t\in[0,T]}:l\in\N\}
 \]
 is a Cauchy sequence in the uniform topology on $D([0,T],\mathbb{L}_\theta)$.
\end{proposition}
\begin{proof}
Let $w= (r,\mathcal{L},\mathcal{G},F,I_{r+1})$ be as above and let $l\in\N$. By the independence assumption, we can construct the processes
\[
(X_t(w^l(w)))_{t\ge0}
\]
by first sampling a realization of $w$ and then using the collections $\mathbf{Y},\mathbf{I},\mathbf{A}$ to obtain the process with that initial configuration. We will show that for all $T\ge 0$ these processes almost surely converge uniformly over $t\in[0,T]$ in $\mathbb{L}_\theta$. Hence, we fix an $\varepsilon>0$ and $T\ge0$. We set
\[
B^w := \bigcap_{S=1}^\infty\bigcap_{n=1}^\infty\bigcup_{l=1}^\infty\left\{\sup_{0\le t\le S}\sum_{(x,i)\in\mathcal{L}\setminus\mathcal{L}^l(w)}\exp(\theta(F(x,i)+Y^{x,i}_t-r))\le \frac{1}{n}\right\}
\]
and note that by continuity from above and below of $\mathbf{P}$ we have
\[
\begin{split}
\mathbf{P}(B^w) &= \lim_{S\to\infty}\lim_{n\to\infty}\lim_{l\to\infty}\mathbf{P}\left(\sup_{0\le t\le S}\sum_{(x,i)\in\mathcal{L}\setminus\mathcal{L}^l(w)}\exp(\theta(F(x,i)+Y^{x,i}_t-r))\le \frac{1}{n}\right)
\end{split}
\]
We now investigate the probability
\[
\mathbf{P}\left(\sup_{0\le t\le S}\sum_{(x,i)\in\mathcal{L}\setminus\mathcal{L}^l(w)}\exp(\theta(F(x,i)+Y^{x,i}_t-r)) > \frac{1}{n}\right).
\]
Since for a simple symmetric random walk $(Y_t)_{t\ge0}$ jumping at rate $2$, the process
\[
(\exp(\theta Y_t-2t(\cosh\vartheta -1)))_{t\ge0}
\]
is a càdlàg martingale, for any $k\ge 1$, also
\[
(M_t^{l,k})_{t\ge0} := \left(\sum_{(x,i)\in\mathcal{L}^{l+k}(w)\setminus\mathcal{L}^l(w)}\exp(\theta(F(x,i)+Y^{x,i}_t-r)-2t(\cosh\vartheta -1))\right)_{t\ge0}
\]
is a càdlàg martingale w.r.t. 
\[
(\mathcal{H}_t)_{t\ge0} := \left(\sigma\left(w,\mathds{1}_{(x,i)\in\mathcal{L}}Y^{x,i}_s:(x,i)\in\Z\times\N,0\le s\le t\right)\right)_{t\ge0}.
\] Thus, Doob's inequality yields
\begin{equation}\label{eq:martinbound}
\mathbf{P}\left(\sup_{0\le t\le S}M_t^{l,k} > \lambda\right) \le \frac{1}{\lambda}\mathbf{E}[M_0^{l,k}] = \frac{1}{\lambda} \mathbf{E}\left[\sum_{(x,i)\in\mathcal{L}^{l+k}(w)\setminus\mathcal{L}^l(w)}\exp(\theta(F(x,i)-r))\right]
\end{equation}
 for any $\lambda > 0$. For fixed $t\ge 0$ the sequence
\[
\left(\sum_{(x,i)\in\mathcal{L}^{l+k}(w)\setminus\mathcal{L}^l(w)}\exp(\theta(F(x,i)+Y^{x,i}_t-r))\right)_{k\ge1}
\]
is increasing almost surely, since only more positive terms get added for each $k$. Thus
\[
\begin{split}
\mathbf{P}&\left(\sup_{0\le t\le S}\sum_{(x,i)\in\mathcal{L}\setminus\mathcal{L}^l(w)}\exp(\theta(F(x,i)+Y^{x,i}_t-r)) > \frac{1}{n}\right) \\
&= \mathbf{P}\left(\bigcup_{k=1}^\infty\left\{\sup_{0\le t\le S}\sum_{(x,i)\in\mathcal{L}^{l+k}(w)\setminus\mathcal{L}^l(w)}\exp(\theta(F(x,i)+Y^{x,i}_t-r)) > \frac{1}{n}\right\}\right) \\
&=\lim_{k\to\infty}\mathbf{P}\left(\left\{\sup_{0\le t\le S}\sum_{(x,i)\in\mathcal{L}^{l+k}(w)\setminus\mathcal{L}^l(w)}\exp(\theta(F(x,i)+Y^{x,i}_t-r)) > \frac{1}{n}\right\}\right) .
\end{split}
\]
Together with \eqref{eq:martinbound}, using $\lambda = \frac{\exp(-2S(\cosh\theta-1))}{n}$, we obtain
\begin{equation}\label{eq:largebound}
\begin{split}
\mathbf{P}&\left(\sup_{0\le t\le S}\sum_{(x,i)\in\mathcal{L}\setminus\mathcal{L}^l(w)}\exp(\theta(F(x,i)+Y^{x,i}_t-r)) > \frac{1}{n}\right) \\
&\le n\exp(2(\cosh\vartheta -1)S)\lim_{k\to\infty}\mathbf {E}\left[\sum_{(x,i)\in\mathcal{L}^{l+k}(w)\setminus\mathcal{L}^l(w)}\exp(\theta(F(x,i)-r))\right]\\
&\overset{\text{mon. conv.}}{=} n\exp(2(\cosh\vartheta -1)S)\mathbf {E}\left[\sum_{(x,i)\in\mathcal{L}\setminus\mathcal{L}^l(w)}\exp(\theta(F(x,i)-r))\right].
\end{split}
\end{equation}
By the assumption
\[
\mathbf{E}\left[\sum_{(x,i)\in\mathcal{L}}\exp(\theta(F(x,i)-r))\right] < \infty
\] 
we can apply dominated convergence to take the limit $l\to\infty$ inside the expectation, which yields
\[
\lim_{l\to\infty}\mathbf{P}\left(\sup_{0\le t\le S}\sum_{(x,i)\in\mathcal{L}\setminus\mathcal{L}^l(w)}\exp(\theta(F(x,i)+Y^{x,i}_t-r)) > \frac{1}{n}\right) = 0,
\]
because the sum 
\[
\sum_{(x,i)\in\mathcal{L}}\exp(\theta(F(x,i)-r))
\]
has finite expectation, thus in particular it is almost surely finite and hence its tailsums converge to $0$ almost surely.
In conclusion, this yields the claim that $\mathbf{P}(B^w)= 1$. Choosing $n=2$ yields that 
\[
B^w\subset\bigcap_{S=1}^\infty\bigcup_{l=1}^\infty\left\{\sup_{\substack{(x,i)\in\mathcal{L}\setminus\mathcal{L}^l(w),\\0\le t\le S}} F(x,i)+Y^{x,i}_t < r\right\}.
\] 
In particular, this means that for all $\omega\in B^w$ there is an $N(\omega,T)$ such that for all $l\ge N(\omega,T)$, none of the living parasites with initial position below $r-l$ have even reached $r$ by time $T$ and were just moving as simple random walks without any interaction. Hence for $k>0$ the configurations $X_t(w^l(w))$ and $X_t(w^{l+k}(w))$ only differ by the parasites with birthlabel in $\mathcal{L}^{l+k}(w)\cap(\mathcal{L}\setminus\mathcal{L}^l(w))$. Thus for all $\omega\in B,t\in[0,T]$ and $l\ge N(\omega,T),k\ge0$
\[
\begin{split}
d_\theta(X_t(w^{l}(w)),X_t(w^{l+k}(w))&= \sum_{(x,i)\in\mathcal{L}^{l+k}(w)\cap(\mathcal{L}\setminus\mathcal{L}^{l}(w))}\exp(\theta(F(x,i)+Y^{x,i}_t-r_t(w^{l}(w))) \\
&\le \sum_{(x,i)\in\mathcal{L}^{l+k}(w)\cap(\mathcal{L}\setminus\mathcal{L}^{l}(w))}\exp(\theta(F(x,i)+Y^{x,i}_t-r))\\&\le \sum_{(x,i)\in\mathcal{L}\setminus\mathcal{L}^{l}(w)}\exp(\theta(F(x,i)+Y^{x,i}_t-r)).
\end{split}
\]
By definition of $B^w$, for all $\omega\in B^w$ we can choose $S= \lfloor T+1\rfloor,n=\left\lceil\frac{1}\varepsilon\right\rceil$, there is an $N^\prime(\omega,T,\varepsilon) \ge N(\omega,T)$ such that for all $l\ge N^\prime(\omega,T,\varepsilon)$
\[
\sup_{0\le t\le S}\sum_{(x,i)\in\mathcal{L}\setminus\mathcal{L}^{l}(w)}\exp(\theta(F(x,i)+Y^{x,i}_t-r))\le \frac{1}{n},
\]
hence, for all
$k\ge0,t\in[0,T]$
\[
d_\theta(X_t(w^{l}(w)),X_t(w^{l+k}(w))) \le \varepsilon.
\]
which is the claim.
\end{proof}
\begin{remark}
 This shows that for all $\omega\in B^w$ the sequence \[\{(X_t(w^l(w(\omega)))(\omega))_{t\ge0}:l\in\N\}\] converges to a limit in $D([0,\infty),\mathbb{L}_\theta)$, which we denote by $(X_t(w)(\omega))_{t\ge0}$.
\end{remark}
By taking the limits $\lim_{l\to\infty} X_t(w^l(w)) = X_t(w)$ in \eqref{normexpect} and \eqref{frontupbound}, we obtain
\begin{equation}\label{frontupboundinfinite}
\mathbf{P}(r_t(w) - r > \gamma t\vert w) \le \e^{-c_{\gamma,\vartheta} t}f_\vartheta(w)\e^{-\vartheta r}
\end{equation}
and 
\[
\mathbf{E}[f_\vartheta(X_t(w))\vert w]\le \e^{\lambda_{2,\vartheta}t}f_\vartheta(w).
\]
For any $t\ge0$ we define linear operators on the space of bounded measurable functions $\B(\mathbb{L}_\theta)$ by
\[
\begin{split}
 S_t: \B(\mathbb{L}_\theta)&\to\B(\mathbb{L}_\theta)\\
 f&\mapsto\mathbf{E}[f(X_t(\cdot))]
\end{split}
\]
and observe that 
\[
\lVert S_t\rVert =\sup_{\lVert f\rVert \le 1}\sup_{w\in\mathbb{L}_\theta}\vert\mathbf{E}[f(X_t(w))]\vert \le \sup_{\lVert f\rVert \le 1}\mathbf{E}[\lVert f\rVert] = 1.
\]
However, considering the sequence $w_n = (0,\{0\}\times\{1,\dots,n\},\emptyset,x\mapsto x, 1)$ and $f: w\mapsto \mathds{1}_{r>0}$, it is easy to see that the map $t\mapsto S_tf$ is not continuous for general $f\in B(\mathbb{L}_\theta)$.
To show the Markov and then the strong Markov property of $(X_t(w))_{t\ge0}$, we will show that $(S_t)_{t\ge0}$ forms a semi-group on the class $\BUC(\mathbb{L}_\theta)$ of bounded and uniformly continuous functions, i.e.
\[
\begin{split}
 \forall t\ge0\forall f\in\BUC(\mathbb{L}_\theta)&: S_tf\in\BUC(\mathbb{L}_\theta) \\
 \forall s,t\ge0\forall f\in\BUC(\mathbb{L}_\theta)&: S_{s+t}f = S_{s}(S_tf).
\end{split}
\]
It is classically known, e.g., \cite[Theorem 2.1]{B1999}, that $\BUC(\mathbb{L}_\theta)$ are convergence determining, and hence the semi-group property shows the Markov property. The strong Markov property follows because $\BUC(\mathbb{L}_\theta)$ is a class of continuous functions, and the sample paths are càdlàg by construction.
We begin with introducing some notation to obtain a result of the dependence of $(X_t(w))_{t\ge0}$ from its initial configuration.
 \theoremstyle{definition}
 \theoremstyle{definition}\begin{definition}
 Let $w,w^\prime\in\mathbb{L}_\theta$ with $r=r^\prime,\iota=\iota^\prime$ and $F_{\vert_{\mathcal{G}}} = F^\prime_{\vert_{\mathcal{G}^\prime}}$. Then we call
 \[
 \overline{\mathcal{L}}(w,w^\prime) := \left\{(x,i)\in\mathcal{L}\cup\mathcal{L}^\prime: \begin{matrix}(x,i)\notin\mathcal{L}\text{ or }(x,i)\notin\mathcal{L}^\prime \text{ or }\\ (x,i)\in\mathcal{L}\cap\mathcal{L}^\prime \text{ and } F(x,i) \neq F^\prime(x,i)\end{matrix}\right\}
 \]
 the distinguishing parasites for $w$ and $w^\prime$.
 \end{definition}
\theoremstyle{plain}
\theoremstyle{plain}\begin{lemma}\label{BUC}
 For all $\varepsilon>0$ and $T\ge0$ there is a $\delta\in(0,1)$ such that for all $w,w^\prime\in\mathbb{L}_\theta$ with
 \[
 d_\theta(w,w^\prime) < \delta
 \]
 there is a $(\sigma(Y^{x,i}_s:0\le s\le t,(x,i)\in\overline{\mathcal{L}}(w,w^\prime)))_{t\ge0}$ stopping time $\tau$, such that 
 \begin{itemize}
 \item $\mathbf{P}(\tau \le T) < \varepsilon$,
 \item $\sup_{0\le t\le T}\mathds{1}_{\tau > T}d_\theta(X_t(w),X_t(w^\prime)) <\varepsilon$.
 \end{itemize}
 \end{lemma}
\begin{proof}
If $d_\theta(w,w^\prime)<\delta$, then $\delta < 1$ implies $r=r^\prime,\mathcal{G}=\mathcal{G}^\prime,F_{\vert_{\mathcal{G}}}= F^\prime_{\vert_{\mathcal{G}}},\iota=\iota^\prime$, which we assume for the rest of the proof. We define $\widetilde{F}(w,w^\prime):\mathcal{L}\cup\mathcal{L}^\prime\cup\mathcal{G}\to\Z$ by
\[
(x,i)\mapsto\begin{cases}F(x,i),&(x,i)\in\mathcal{G}\\F(x,i),&(x,i)\in\mathcal{L}\cup\mathcal{L}^\prime\setminus\mathcal{L}^\prime\\F^\prime(x,i),&(x,i)\in\mathcal{L}\cup\mathcal{L}^\prime\setminus\mathcal{L}\\F(x,i)\vee F^\prime(x,i),&(x,i)\in\mathcal{L}\cap\mathcal{L}^\prime \end{cases}.
\]
For $n\in\N$ we define
\[
\begin{split}
A_{T}^n&:=\left\{\sup_{0\le t\le T}\sum_{(x,i)\in \overline{\mathcal{L}}(w,w^\prime)}\exp(\theta(\widetilde{F}(w,w^\prime)(x,i)+Y^{x,i}_t-r))\le \frac{1}{n}\right\}.
\end{split}
\]
 For $\omega\in A_{T}^n$, none of the frogs that distinguish $w$ and $w^\prime$ reached $r+1$ by time $T$ in either $X_t(w)$ or $X_t(w^\prime)$; therefore, these configurations only differ by the current positions of frogs in $\overline{\mathcal{L}}(w,w^\prime)$. Hence, $r_t(w)(\omega) = r_t(w^\prime)(\omega)$ and 
\begin{align*}
d_\theta(X_t(w),X_t(w^\prime)) &= \sum_{(x,i)\in(\mathcal{L}\cup\mathcal{L}^\prime)\setminus\mathcal{L}^\prime}\e^{\theta(F(x,i) + Y^{x,i}_t- r_t(w))} \\ &\quad+\sum_{(x,i)\in(\mathcal{L}\cup\mathcal{L}^\prime)\setminus\mathcal{L}}\e^{\theta(F^\prime(x,i) + Y^{x,i}_t- r_t(w))} \\
&\quad+\sum_{(x,i)\in \mathcal{L}\cap\mathcal{L}^\prime\cap\overline{\mathcal{L}}(w,w^\prime)}\e^{\theta(Y^{x,i}_t-r_t(w))}\left\vert\e^{\theta F(x,i)}-\e^{\theta F^\prime(x,i)}\right\vert \\
&\le \sum_{(x,i)\in \overline{\mathcal{L}}(w,w^\prime)}\exp(\theta(\widetilde{F}(w,w^\prime)(x,i)+Y^{x,i}_t-r)).
\end{align*}
This implies
\[
\sup_{0\le t\le T}d_\theta(X_t(w),X_t(w^\prime))\mathds{1}_{A_{T}^n} \le \frac{1}{n}.
\]
Now we define 
\[
\tau^n := \inf\left\{t\ge 0:\sum_{(x,i)\in \overline{\mathcal{L}}(w,w^\prime)}\exp(\theta(\widetilde{F}(w,w^\prime)(x,i)+Y^{x,i}_t-r))> \frac{1}{n}\right\}
\]
which clearly is an $(\sigma(Y^{x,i}_s:0\le s\le t,(x,i)\in\overline{\mathcal{L}}(w,w^\prime)))_{t\ge0}$-stopping time and note that
\[
 \{\tau^n > T\} = A_{T}^n.
\]
Analogously to the estimation \eqref{eq:largebound} in the proof of Proposition \ref{unifconv}, we have
\[
\begin{split}
\mathbf{P}\left(\Omega\setminus A_{T}^n\right) &\le n\exp(2T(\cosh\theta-1)) \sum_{(x,i)\in \overline{\mathcal{L}}(w,w^\prime)}\exp(\theta(\widetilde{F}(w,w^\prime)(x,i)-r)) \\
&\le n\exp(2T(\cosh\theta -1))\frac{d_\theta(w,w^\prime)}{1-\e^{-\theta}},
\end{split}
\]
where the last inequality follows from
\[
\vert\e^{\theta F(x,i)}-\e^{\theta F^\prime(x,i)}\vert \ge (1-\e^{-\theta})\e^{\theta(F(x,i)\vee F^\prime(x,i))}.
\]
Hence choosing $n:=\left\lfloor 1 +\frac{1}{\varepsilon}\right\rfloor$ and 
\[
\delta = \frac{1-\e^{-\theta}}{n^2\exp(2T(\cosh\theta -1))}
\]
yields the claim.
\end{proof}
With this refinement we can show that $S_t$ maps $\BUC(\mathbb{L}_\theta)$ onto itself.
\theoremstyle{plain}\begin{lemma}\label{BUCinvariant}
 Let $f\in\BUC(\mathbb{L}_\theta)$ and $T\ge 0$. Then for any $\varepsilon>0$ there is a $\delta>0$ such that for all $w,w^\prime\in\mathbb{L}_\theta$ 
 \[
 d_\theta(w,w^\prime)<\delta~~~\Rightarrow~~~\sup_{0\le t\le T}\left\vert\mathbf{E}[f(X_t(w))-f(X_t(w^\prime))]\right\vert < \varepsilon.
 \]
 In particular, $S_tf\in\BUC(\mathbb{L}_\theta)$ for any $t\ge0$.
\end{lemma}
\begin{proof}
Let $f\in\BUC(\mathbb{L}_\theta)$ and $\varepsilon > 0$. Since $f$ is uniformly continuous, there is an $\varepsilon_0\in\left(0,\frac{\varepsilon}{4\lVert f\rVert_\infty}\right)$ such that $d_\theta(w,w^\prime)<\varepsilon_0$ implies $\vert f(w)-f(w^{\prime})\vert < \frac{\varepsilon}{2}$. Now by Lemma \ref{BUC} we find a $\delta\in (0,1)$ such that for all $w,w^{\prime}\in\mathbb{L}_\theta$ with $d_\theta(w,w^{\prime})< \delta$ there is a random time $\tau$ with
\[
\mathbf{P}(\tau \le T) < \varepsilon_0
\]
and 
\[
\sup_{0\le t\le T}\mathds{1}_{\tau > T}d_\theta(X_t(w),X_t(w^{\prime})) < \varepsilon_0.
\]
Then for any $w,w^{\prime}\in\mathbb{L}_\theta$ with $d_\theta(w,w^{\prime})<\delta$ and all $t\in[0,T]$ we have
\[
\begin{split}
\vert\mathbf{E}[f(X_t(w))-f(X_t(w^{\prime}))]\vert &\le \mathbf{E}[\mathds{1}_{\tau > T}\vert f(X_t(w))-f(X_t(w^{\prime}))\vert] + \mathbf{E}[\mathds{1}_{\tau \le T}2\lVert f\rVert_\infty] \\
& < \frac{\varepsilon}{2} + 2\lVert f\rVert_\infty \varepsilon_0 < \varepsilon.
\end{split}
\]
Since $w\mapsto \mathbf{E}[f(X_t(w))]$ is clearly bounded by $\lVert f\rVert_\infty$, this finishes the proof.
\end{proof}
\theoremstyle{plain}\begin{lemma}\label{BUCsemigroup}
 For $f\in\BUC(\mathbb{L}_\theta),w\in\mathbb{L}_\theta$ and $s,t\ge0$ we have
 \[
 (S_{s+t}f)(w) = (S_s(S_tf))(w).
 \]
\end{lemma}
\begin{proof}
First we note, that for $w\in\mathbb{A}$ we already know that $(X_t(w))_{t\ge0}$ is a Markov process in $\mathbb{A}$, hence
\[
S_{s+t}f(w) = (S_s(S_tf))(w)
\]
for any $w\in\mathbb{A},f\in\BUC(\mathbb{L}_\theta)$. \\
 Let $\varepsilon > 0$ and choose $\delta_1\in\left(0,\frac{\varepsilon}{2\lVert f\rVert}\right)$ as in Lemma \ref{BUCinvariant} with $T=s+t$. Then choose $\delta_2 \in(0 ,\delta_1)$ as in Lemma \ref{BUC} with $T=s+t,\varepsilon = \delta_1$. Finally choose $l\in\N$ such that $d_\theta(w,w^l(w))< \delta_2$ and compute
 \begin{align*}
 \vert S_{s+t}f(w) - S_sS_tf(w)\vert &\le\vert S_{s+t}f(w)-S_{s+t}f(w^l(w))\vert \\&\quad+ \vert S_{s+t}f(w^l(w))-S_sS_tf(w^l(w))\vert\\
 &\quad+\vert S_sS_tf(w^l(w)) - S_sS_tf(w)\vert\\
 &\le \varepsilon + 0 + \vert S_sS_tf(w^l(w)) - S_sS_tf(w)\vert.
 \end{align*}
 Now to estimate the final term we choose $\tau$ as in Lemma \ref{BUC} corresponding to $w,w^l(w)$ and compute
 \begin{align*}
 \vert S_sS_tf(w^l(w)) - S_sS_tf(w)\vert & = \left\vert\mathbf{E}\left[S_tf(X_s(w^l(w)))-S_tf(X_s(w))\right] \right\vert\\
 &\le \mathbf{E}\left[\mathds{1}_{\tau>t+s}\left\vert S_tf(X_s(w^l(w)))-S_tf(X_s(w))\right\vert\right] \\
 & \quad + \mathbf{E}\left[\mathds{1}_{\tau\le t+s}\left\vert S_tf(X_s(w^l(w)))-S_tf(X_s(w))\right\vert\right] \\
 &\le \varepsilon + 2\lVert f\rVert \delta_1 < 2\varepsilon.
 \end{align*}
 Because these estimations hold for any $\varepsilon>0$, this concludes the proof.
\end{proof}
This shows that for any $w\in\mathbb{L}_\theta$ the process $(X_t(w))_{t\ge0}$ is a Markov process. To show the strong Markov property we simply use that its semi-group preserves the separating class $\BUC(\mathbb{L}_\theta)$ of functions that, in particular, are continuous.
\theoremstyle{plain}\begin{theorem}\label{StrongMarkov}
 For $w\in\mathbb{L}_\theta$ the process $(X_t(w))_{t\ge0}$ is strong Markov with respect to its natural filtration and almost surely has càdlàg sample paths.
\end{theorem}
\begin{proof}
By construction, $(X_t(w))_{t\ge0}$ has càdlàg sample paths for all $w\in\mathbb{A}$, and by uniform convergence,
\[
\lim_{l\to\infty} \sup_{0\le t\le T}d_\theta(X_t(w),X_t(w^l(w))) = 0,
\]
so also $(X_t(w))_{t\ge 0}$ has càdlàg sample paths. \\
Let $\sigma$ be an almost surely finite $(\mathcal{F}_t(w))_{t\ge0}$ stopping time; then there is a sequence $(\sigma_n)_{n\in\N}$ of $(\mathcal{F}_t(w))_{t\ge0}$ stopping times with a countable range, such that $\sigma_n\downarrow\sigma$ for $n\to\infty$. Because each $\sigma_n$ has a countable range, we obtain that $(X_t(w))_{t\ge0}$ is strong Markov at $\sigma_n$. Using $\sigma_n \ge \sigma$ we obtain $\mathcal{F}_\sigma(w)\subset\mathcal{F}_{\sigma_n}(w)$ and hence for each $f\in\BUC(\mathbb{L}_\theta)$ the right continuity of $t\mapsto X_t(w)$ and the continuity of $w\mapsto \mathbf{E}[f(X_t(w))]$ yield
\[
\begin{split}
\mathbf{E}[f(X_{\sigma+t}(w))\vert\mathcal{F}_\sigma(w)](\omega) &= \lim_{n\to\infty}\mathbf{E}[f(X_{\sigma_n+t}(w))\vert\mathcal{F}_\sigma(w)](\omega)\\
&= \lim_{n\to\infty}\mathbf{E}[\mathbf{E}[f(X_{\sigma_n+t}(w))\vert\mathcal{F}_{\sigma_n}(w)]\vert\mathcal{F}_\sigma(w)](\omega) \\
 &= \lim_{n\to\infty}\mathbf{E}\left[\left.\int_\Omega f(X_t(X_{\sigma_n}(w))(\omega^\prime))\diff\mathbf{P}(\omega^\prime)\right\vert\mathcal{F}_\sigma(w)\right](\omega) \\
 &=\mathbf{E}\left[\left.\int_\Omega f(X_t(X_{\sigma}(w))(\omega^\prime))\diff\mathbf{P}(\omega^\prime)\right\vert\mathcal{F}_\sigma(w)\right](\omega) \\
 &= \int_\Omega f(X_t(X_{\sigma}(w)(\omega)) (\omega^\prime))\diff\mathbf{P}(\omega^\prime).
\end{split}
\]
This concludes the proof, because $\BUC(\mathbb{L}_\theta)$ is separating.
 \end{proof}

 \subsection{Untagged system}\label{Construction:unt}
In this subsection we construct the untagged system as described at the beginning of this section. In the construction we only consider finite initial configurations. The extension to infinite initial configurations follows by similar ideas as in the previous section. In Proposition \ref{prop:untaged_mon} we show that this coupling is monotone in the initial configuration.\\
 In this construction we no longer keep track of the birthplaces of parasites and just store the current configuration of all parasites. Thus the state space of configurations will be given by
 \[
 \mathbb{S}_\theta:= \left\{(r,\eta,\overline{\eta},\iota)\in\Z\times\N_0^{\Z}\times\N_0^{\Z}\times\N: \begin{matrix}\supp\eta\subset(-\infty,r], \sum_{x\in\Z}\overline{\eta}(x) < \infty,\\\sum_{x\le r}\eta(x)\e^{\theta (x-r)} < \infty\end{matrix}\right\}.
 \]
 The front is given by $r$; for $x\in\Z$, the number of living parasites on site $x$ is given by $\eta(x)$, and the number of ghost parasites on $x$ is given by $\overline{\eta}(x)$, and finally the immunity of the host at $r+1$ is given by $\iota$.
To construct the process we will use a collection of Poisson point processes to sample the jumps for each site. To this end, let $\{P_{x,n}^{\leftarrow},P_{x,n}^{\rightarrow}:x\in\Z,n\in\N\}$ and $\{G_{x,n}^\leftarrow,G_{x,n}^\rightarrow:x\in\Z,n\in\N\}$ be independent collections of independent Poisson point processes on $\R$ with intensity $1$ and set
\[
P(x,n) := \sum_{k=1}^n P^{\leftarrow}_{x,k} + P^{\rightarrow}_{x,k},~~G(x,n) := \sum_{k=1}^n G^{\leftarrow}_{x,k} + G^{\rightarrow}_{x,k}
\]
which are independent Poisson point processes on $\R$ with intensity $2n$. By convention we set $P(x,0) := G(x,0) := \emptyset$. We assume that these processes are defined on the same probability space as $\mathbf{A},\mathbf{I}$. \\ For $l,r\in \Z$, a configuration $\eta\in\N_0^{\Z},\overline{\eta}\in\N_0^{\Z}$ with $\supp \eta ,\supp\overline{\eta}\subset[l,r]$, $\iota\in\N$, and a time $t_0\in\R$, we define the process $(\zeta_t(r,\eta,\overline{\eta},\iota;t_0))_{t\ge t_0} $ given by the tuple
\[ (r_t(r,\eta,\overline{\eta},\iota;t_0),\eta_t(r,\eta,\overline{\eta},\iota;t_0),\overline{\eta}_t(r,\eta,\overline{\eta},\iota;t_0)\iota_t(r,\eta,\overline{\eta},\iota;t_0))_{t\ge t_0}\]
as follows. The collection $\{P_{x,n}^{\leftarrow}:x\in\Z,n\in\N\}$ gives the time points at which a living parasite jumps from $x$ to $x-1$, where a jump time $t\in P_{x,n}$ is only allowed to be used if $\eta_{t-}(x) \ge n$. Analogously, the times that a living parasite jumps from a site $x$ to $x+1$ are given by the collection $\{P_{x,n}^{\rightarrow}:x\in\Z,n\in\N\}$. The collection $\{G_{x,n}^{\leftarrow},G_{x,n}^\rightarrow:x\in\Z,n\in\N\}$ determines the jumps of ghost parasites in the same way. Whenever for some time $t$ a jump of a living parasite onto the site $r_{t-} + 1$ happens, we use $\iota_{t-}$ to determine the outcome of the infection attempt. If $\iota_{t-} = 1$, then \[r_t = r_{t-} + 1,\eta_t := \eta_{t-} -\delta_{r_{t-}} + A_{r_t}\delta_{r_t}, \overline{\eta}_t := \overline{\eta}_t + \delta_{r_t}\text{ and } \iota_t := I_{r_t+1}.\] Otherwise, simply \[r_t = r_{t-},\iota_t = \iota_{t-} -1, \overline{\eta}_t := \overline{\eta}_t + \delta_{r_t+1}\text{ and }\eta_t = \eta_{t-}-\delta_{r_t}.\] 
It is clear that the so-constructed process $(r_t(r,\eta,\overline{\eta},\iota; 0))_{t\ge0}$ has the same distribution as the process $(r_t(w))_{t\ge0}$ constructed from the tagged system in the previous Section \ref{Construction} with $w$ given by
\[
(r,\{(x,i)\in\Z\times\N: 1\le i \le \eta(x)\},\{(x,i)\in\Z\times\N:1\le i\le \overline{\eta}(x)\},(x,i)\mapsto x,\iota).
\]
Also, it is clear that the resulting process is a strong Markov process with respect to its natural filtration and also with respect to the filtration
\[
\sigma(\mathbf{A},\mathbf{I},\zeta,\{[0,t]\cap P(x,n),[0,t]\cap G(x,n):x\in\Z,n\in\N\}).
\]
In this coupling we have the following monotonicity property.
\begin{proposition}\label{prop:untaged_mon}
 Let $t_0\in\R,l,r\in\Z$ and $\eta^1,\eta^2\in\N_0^{\Z}$ with \[\supp \eta_1,\supp \eta_2\subset(l,r]\] such that
 \[
 \sum_{k\ge x} \eta^1(k) \le \sum_{k\ge x}\eta^2(k)
 \]
 for all $x\in\Z$. In particular, we allow $t_0,r,\eta^1,\eta^2$ to be random and the conditions above to hold almost surely. Let $(r_t^1,\eta_t^1,\iota_t^1)_{t\ge t_0}$ be the process with initial configuration $(r,\eta^1,I_{r+1})$ at time $t_0$ and let $(r_t^2,\eta_t^2,\iota_t^2)_{t\ge t_0}$ be the process with initial configuration $(r,\eta^2,I_{r+1})$ at time $t_0$. Then almost surely for all $t\ge t_0$ we have
 \[
 r_t^1 \le r_t^2.
 \]
\end{proposition}
\begin{proof}
For $n\ge 0$ we set
 \[
 \rho_n^1 := \inf\{t\ge t_0: r_t^1 \ge r+n\},~~\rho_n^2 := \inf\{t\ge t_0: r_t^2 \ge r+n\}.
 \]
 First we note that the event of parasite extinction in the $i$-th process for $i\in\{1,2\}$ is given by
 \[
 \mathcal{D}^i = \bigcup_{n\ge 1}\left\{\sum_{k=1}^n I_{r+k} > \sum_{x\le r}\eta^i(x) + \sum_{k=1}^{n-1}A_{r+k}\right\}.
 \]
Let $n^i_\text{d}\ge 1$ be the random number such that the extinction in the $i$-th process occurs when the front is at $r+n^i_\text{d}-1$ and $\tau^i_\text{d}$ be the time at which it occurs. In particular, by assumption, we have $n^1_\text{d} \le n^2_\text{d}$. We will now show that for $0\le n < n^1_\text{d}$ we have $\rho^1_n\ge \rho^2_n$. Since by definition, we have $\rho^1_{n^1_\text{d}} = \infty$, this shows the claim. For $t\in[t_0,\rho_1^1\wedge\rho_1^2)$ the amount of parasites that died so far in the $i$-th process is given by $I_{r+1}-\iota^i_t$. Since both processes use the same Poisson point process, accounting for the lost mass due to parasite death, this implies that
 \[
 I_{r+1}-\iota_t^1 + \sum_{k\ge x} \eta_t^1(k) \le I_{r+1}-\iota_t^2 + \sum_{k\ge x}\eta_t^2(k)
 \]
 for all $x\in\Z$ and $t\in[t_0,\rho_1^1\wedge \rho_2^1)$. In particular, since $\rho_1^i$ is given by 
 \[
 \inf\left\{s\ge t_0:\left\vert\left\{t\in [t_0,s]\cap\sum_{k=1}^\infty P_{r,k}^{\rightarrow}: t\in P(r,\eta^i_{t-}(r))\right\}\right\vert = I_{r+1}\right\},
 \]
 we obtain that $\rho_1^2 \le \rho_1^1$, because the only way for $\eta_t^2(r)$ to be less than $\eta_t^1(r)$ is if the difference has already tried to infect the host at $r+1$. Arguing analogously as above, we then must have
 \[
 -\iota_t+\sum_{k=r+1}^{r_t^1+1} I_k +\sum_{k\ge x}\eta_t^1(k) \le-\iota_t+\sum_{k=r+1}^{r_t^2+1} I_k +\sum_{k\ge x}\eta_t^2(k) 
 \]
 for all $x\in\Z$ and $t\in [\rho_1^2,\rho_2^2\wedge\rho_2^1)$, which again implies $\rho_2^2 \le \rho_2^1$. Repeating this yields $\rho_n^2 \le \rho_n^1$ for all $0\le n < n^1_\text{d}$ and thereby finishes the proof.
\end{proof}
With this monotonicity result, we can show the Feller property claimed in Theorem \ref{Theorem:well-def} by following the same argumentation as done for the classical frog model in \cite[section 6]{C2009}.
\begin{proof}[Proof of Theorem \ref{Theorem:well-def}]
This proposition implies that the constructed process is monotone in the initial configuration; in particular, for any $\zeta = (r,\eta,\overline{\eta},\iota)\in\mathbb{S}_\theta$ and $l\in\N$, we can define 
\[
\zeta^l = (r,\mathds{1}_{x> r-l}\eta(\cdot),\overline{\eta},\iota)
\]
which contains only finitely many parasites. By the lemma above, we then have that for any fixed $t\ge0$ the sequence
\[
(r_t(\zeta^l;0))_{l\ge1}
\]
is almost surely monotone and in particular has a limit
\[
r_t(\zeta;0) := \lim_{l\to\infty} r_t(\zeta^l;0) \in \Z\cup\{+\infty\}.
\]
Following the same arguments as in \cite[section 6]{C2009}, this yields that we obtain a well-defined strong Markov process
\[
(r_t(\zeta),\eta_t(\zeta),\overline{\eta}_t(\zeta),\iota_t(\zeta);0)_{t\ge0}
\]
for any configuration $\zeta\in\mathbb{S}_\theta$. Also, setting $a_m := \mathbb{P}(A= m)$ and $i_k := \mathbb{P}(I= k)$ the process has the generator
\[
\begin{split}
Lf(\zeta) = &\sum_{\substack{x,y\le r\\\vert x-y\vert = 1}}\eta(x)(f(r,\eta-\delta_x+\delta_y,\overline{\eta},\iota)-f(r,\eta,\overline{\eta},\iota)) \\
&+ \sum_{\substack{x,y\in\Z\\\vert x-y\vert = 1}}\overline{\eta}(x)(f(r,\eta,\overline{\eta}-\delta_x+\delta_y,\iota)-f(r,\eta,\overline{\eta},\iota)) \\
&+\eta(r)\left(\begin{matrix}\mathds{1}_{\iota=1}\sum_{m=0}^\infty a_m\sum_{k=1}^\infty i_k f(r+1,\eta-\delta_r+m\delta_{r+1},\overline{\eta}+\delta_{r+1},k)\\ +\mathds{1}_{\iota > 1}f(r,\eta-\delta_r,\overline{\eta}+\delta_{r+1},\iota) \\-f(r,\eta,\overline{\eta},\iota)\end{matrix}\right)
\end{split}
\]
acting over functions $f\in \BUC(\mathbb{S}_\theta)$ such that the sums above converge. In particular, for $\vartheta \ge \theta > 0$ and
\[
f_\vartheta(\zeta) := \sum_{x\in\Z}(\eta+\overline{\eta})(x)\e^{\vartheta x},
\]
and $\zeta\in\mathbb{S}_\theta$ with only finitely many living parasites we can compute
\[
Lf_\vartheta(\zeta) \in [(e^\vartheta+\e^{-\vartheta}-2)f_\vartheta(\zeta),(\e^{-\vartheta}+\mathbb{E}[A]\e^{\vartheta} - 2) f_\vartheta(\zeta)]
\]
and
\begin{equation}\label{eq:quadmart}
(L(f_\vartheta^2)-2f_\vartheta Lf_\vartheta)(\zeta) \le \mathbb{E}[A^2]\lambda_{3,\vartheta}f_{2\vartheta}(\zeta)
\end{equation}
for some constant $\lambda_{3,\vartheta} \in (0,\infty)$. We note that the introduction of ghost parasites is necessary to obtain the lower bound for $Lf_\vartheta(\zeta)$. These claims follow by just multiplying out $f_\vartheta^2$, plugging it in the generator $L$, and then carefully going through the cases for different $x_1,x_2,y_1,y_2$ to see the cancellations, and we omit this lengthy calculation at this point. For $\zeta \in\mathbb{S}_\theta$ let $\mathbb{P}_\zeta$ be the completion with respect to the canonical filtration on the space of càdlàg functions $D([0,\infty),\mathbb{S}_\theta)$ of the measure
\[
\mathbb{P}((\zeta_t(\zeta;0))_{t\ge0} \in \cdot).
\]
 Then using the estimations above and following the same argumentation as in \cite[section 6]{C2009}, we can see that $(\mathbb{P}_\zeta)_{\zeta\in\mathbb{S}_\theta}$ forms a Feller process on $\mathbb{S}_\theta$, by first identifying the compact sets of $\mathbb{S}_\theta$ as in \cite[Lemma 22]{C2009} and then using \eqref{eq:quadmart} to obtain an analog of \cite[Lemma 23]{C2009} and conclude the result as in \cite[Proposition 5]{C2009}. Finally we observe that forgetting about the ghost parasites, that is projecting the process onto its other three components, is continuous and retains the Markov property and the Feller property.
\end{proof}

\end{document}